%% file: draft.tex
\documentclass[english,ruled]{article}
\usepackage[T1]{fontenc}
\usepackage[latin9]{inputenc}
\usepackage{verbatim}
\usepackage{algorithm2e}
\usepackage{amsmath}
\usepackage{amsthm}
\usepackage{amssymb}
\usepackage{graphicx}
\usepackage{xcolor}
\usepackage{multicol}
\usepackage{multirow}
\usepackage{geometry}
\usepackage{booktabs}
\usepackage{enumitem}
\usepackage{setspace}
\makeatletter
\usepackage[toc,page,header]{appendix}
\usepackage{minitoc}
\usepackage{pgfplots}
\usepackage{pdflscape}
\usepackage{hyperref}
\usepackage{authblk}
\pgfplotsset{compat=1.15}

\newboolean{doublecolumn}
\setboolean{doublecolumn}{false}
\newboolean{arxiv}
\setboolean{arxiv}{false}

\newcommand\prove[1]{{\ifthenelse{\boolean{arxiv}}{}{#1}}}
\newcommand\proveapp[1]{{\ifthenelse{\boolean{arxiv}}{#1}{}}}

\renewcommand \thepart{}
\renewcommand \partname{}

\theoremstyle{plain}
\newtheorem{thm}{\protect\theoremname}[section]
\theoremstyle{plain}
\newtheorem{lem}{\protect\lemmaname}[section]
\theoremstyle{rem}
\newtheorem{rem}{\protect\remname}
\theoremstyle{plain}

\providecommand{\corollaryname}{Corollary}
\theoremstyle{plain}
\newtheorem{coro}{\protect\corollaryname}
\theoremstyle{plain}

\theoremstyle{plain}

\usepackage{babel}
\providecommand{\propositionname}{Proposition}
\providecommand{\theoremname}{Theorem}

\usepackage{url}            
\usepackage{booktabs}       
\usepackage{amsfonts}       
\usepackage{nicefrac}       
\usepackage{authblk}
\usepackage{optidef}

\newcommand\condspace{{\ifthenelse{\boolean{doublecolumn}}{}{}}}
\newcommand\condsep{{\ifthenelse{\boolean{doublecolumn}}{\vspace{-\topsep}}{}}}

\makeatother

\usepackage{babel}

\usepackage[resetlabels]{multibib}
\usepackage{adjustbox}

\newcites{app}{References in appendix}

\providecommand{\lemmaname}{Lemma}
\providecommand{\remname}{Remark}
\providecommand{\theoremname}{Theorem}
\providecommand{\defname}{Definition}
\providecommand{\examplename}{Example}

\newcommand{\assign}{:=}
\newcommand{\cdummy}{\cdot}

\newcommand{\tmop}[1]{\ensuremath{\operatorname{#1}}}
\newcommand{\tmtextbf}[1]{\text{{\bfseries{#1}}}}
\newcommand{\tmtextit}[1]{\text{{\itshape{#1}}}}
\newcommand{\A}{\ensuremath{\mathbf{A}}}

\newcommand{\tma}{\ensuremath{\mathbf{a}}}
\newcommand{\D}{\mathbf{D}}
\newcommand{\tmd}{\mathbf{d}}
\newcommand{\x}{\mathbf{x}}
\newcommand{\maxf}[1]{\underset{#1}{\tmop{maximize}}}
\newcommand{\minf}[1]{\underset{#1}{\tmop{minimize}}}
\newcommand{\1}{\textbf{1}}
\newcommand{\e}{\mathbf{e}}
\newcommand{\0}{\textbf{0}}
\newcommand{\M}{\mathbf{M}}
\newcommand{\z}{\mathbf{z}}
\newcommand{\bs}{\mathbf{S}}
\newcommand{\I}{\mathbf{I}}
\newcommand{\diag}{\tmop{diag}}
\newcommand{\diagm}{\tmop{diagm}}

\newcommand{\tmP}{\mathbf{P}}
\newcommand{\Z}{\mathbf{Z}}
\newcommand{\X}{\mathbf{X}}

\newcommand{\tmv}{\mathbf{v}}

\begin{document}

\title{Scalable Approximate Optimal Diagonal Preconditioning}

\setlength{\parindent}{0pt}
\linespread{1.0}
\author[1]{Wenzhi Gao}
\author[2]{Zhaonan Qu}
\author[1,2]{Madeleine Udell}
\author[1,2]{Yinyu Ye}
\affil[1]{ICME, Stanford University}
\affil[2]{Management Science and Engineering, Stanford University\vspace{10pt}\\ \texttt{\{gwz,zhaonanq,udell,yyye\}@stanford.edu}}
\maketitle

\input{abstract.tex}
\input{sec_intro.tex}
\input{sec_preliminaries.tex}
\input{sec_subspace.tex}
\input{sec_semiinf.tex}
\input{sec_colgen.tex}

\input{sec_exp.tex}
\input{conclusion.tex}

\renewcommand \thepart{}
\renewcommand \partname{}

\bibliographystyle{plain}
\bibliography{ref.bib}
\doparttoc
\faketableofcontents
\part{}

\newpage
\appendix
\addcontentsline{toc}{section}{Appendix}
\part{Appendix} 
\parttoc

\input{app_proof.tex}

\input{add_exp.tex}

\end{document}

%% file: abstract.tex
\begin{abstract}
We consider the problem of finding the optimal diagonal preconditioner for a positive definite matrix. Although this problem has been shown to be solvable and various methods have been proposed, none of the existing approaches are scalable to matrices of large dimensions or when access is limited to black-box matrix-vector products, thereby significantly limiting their practical application. In view of these challenges, we propose practical algorithms applicable to finding approximate optimal diagonal preconditioners of large sparse systems. Our approach is based on the idea of dimension reduction and combines techniques from semi-definite programming (SDP), random projection, semi-infinite programming (SIP), and column generation. Numerical experiments demonstrate that our method scales to sparse matrices of size greater than $10^7$. Notably, our approach is efficient and implementable using only black-box matrix-vector product operations, making it highly practical for various applications.
\end{abstract}

%% file: sec_intro.tex
\section{Introduction} \label{sec:intro}

It is common to reduce
the condition number \[\kappa ( \M ) = 
\lambda_{\max} ( \M ) / \lambda_{\min} ( \M )\]
of a positive definite matrix $\M \in \mathbb{S}_{+ +}^n$ 
using a diagonal preconditioner $\D \in \mathbb{S}_{+ +}^n$ to form the preconditioned matrix $\D^{-1/2}\M\D^{-1/2}$.
In this paper, we seek to find the optimal diagonal preconditioner by solving the problem
\begin{equation}\label{eqn:formulation}
\begin{array}{ll}
\minf{\D \text{~diagonal}} & \kappa ( \D^{- 1/2} \M \D^{- 1/2} ) \\
\mbox{subject to} & \diag(\D) \geq 0
\end{array}
\end{equation}

Preconditioning is critical to ensure good performance of many algorithms across numerical optimization and numerical
linear algebra {\cite{saad2003iterative}}. 
Among the wide variety of proposed preconditioners
{\cite{benzi2003robust,benzi1999comparative,chen2006modified,frangella2023randomized}},
diagonal preconditioners are particularly popular in practice
{\cite{giselsson2014diagonal, pock2011diagonal, takapoui2016preconditioning}}.
They are easy to apply, as they simply scale the rows or columns of the matrix. More importantly, there are many fast heuristic methods which produce effective diagonal preconditioners {\cite{jacobi1845ueber,ruiz2001scaling,bradley2010algorithms, knight2013fast, knight2014symmetry}}.\\

Despite the success of heuristic diagonal preconditioners in practice, 
none of them is guaranteed to reduce the condition
number; indeed, they may even increase it. To address this issue, a number of researchers have studied the problem of \emph{optimal} diagonal preconditioning, which seeks the diagonal preconditioner that maximally reduces the condition number \cite{braatz1994minimizing, greenbaum1989optimal, van1969condition, shapiro1982optimally,watson1991algorithm, overton1992large}. These results show that finding the optimal diagonal
preconditioner can be formulated as a convex optimization problem. 
However, efficient optimization routines were not yet available then, so they mainly focused on theoretical aspects of the problem. In contrast, this work focuses on the practical and algorithmic aspects of optimal diagonal preconditioning. \\

More recent advances in optimization theory and practice 
have opened the door to algorithms for computing the optimal diagonal preconditioner
{\cite{jambulapati2023structured, jambulapati2020fast, lu2011minimizing, qu2022optimal}}.
For example, \cite{qu2022optimal} shows that \eqref{eqn:formulation} can be reformulated as a semi-definite programming problem (SDP) with $n+1$ variables, which in theory
can be solved to any specified accuracy in polynomial time \cite{wolkowicz2012handbook}. Numerical experiments in \cite{qu2022optimal} demonstrate that their SDP approach can find the optimal diagonal preconditioner for matrices of size up to 5000 with customized solvers.
\cite{overton1992large} designs a successive quadratic minimization procedure using the Cholesky decomposition of $\M$.
\cite{jambulapati2020fast} provides an alternative approach based on structured mixed packing and covering
SDPs. They provide near-optimal complexity guarantees for their algorithm but do not consider practical implementations.

Other researchers have studied the problem of optimizing the condition number 
allowing preconditioners in some convex set (not only diagonal) using tools from nonsmooth optimization \cite{marechal2009optimizing, chen2011minimizing} or
semi-definite programming \cite{lu2011minimizing}. Recently \cite{kunstner2023searching} proposes a multi-dimensional search method that finds a preconditioner whose condition number is within $\mathcal{O}(\sqrt{n})$ factor of the optimal condition number, and their approach can be embedded in an iterative procedure.\\

However, none of the proposed approaches scales to large matrices: most have $\mathcal{O}(n^3)$ complexity or worse, making them not useful for large-scale problems that require preconditioning the most. Moreover, in many applications, the matrix $\M$ itself is not available and is accessible only through matrix-vector products $\M\tmv$, while the existing methods require entry-wise access to $\M$. These issues significantly
limit the appeal of the previous methods to find the optimal diagonal preconditioner.
This paper, in contrast, considers the following problem:
\begin{center}
\quad{\textit{Can a \textbf{scalable} algorithm approximate the optimal diagonal preconditioner?}}
\end{center}
Or, can we efficiently compute a preconditioner to reduce the condition number to within a constant factor of that achieved by the optimal diagonal preconditioner? 

\subsection{Contributions}

In this paper, we answer the above question positively and provide practical algorithms to approximate the optimal diagonal preconditioner for large sparse matrices efficiently. Our results leverage and combine several
research directions in optimization, including semi-definite programming, random projection, semi-infinite programming, and column generation. We are guided by the simple yet powerful idea of dimension reduction, which manifests itself in several distinct ways. Our contributions are as follows.
\begin{enumerate}[leftmargin=*]
\item \emph{Change of basis.} We extend the framework of \cite{qu2022optimal} by introducing an SDP formulation to compute the optimal diagonal preconditioner restricted to a \emph{subspace} spanned by $k$ given diagonal matrices, which we call a basis. 
The subspace formulation increases the flexibility and scalability of the SDP approach. 
By taking a set of heuristic diagonal preconditioners as the basis,
we obtain a new diagonal preconditioner guaranteed to perform better than heuristic preconditioners. We also analyze the optimal preconditioner in a \emph{randomized}  subspace. It is shown that the benefit of expanding the random subspace marginally decreases.
\item \emph{Semi-infinite programming.} 
Even with $k=\mathcal{O}(1)$, solving an SDP is still costly for large $n$. 
We develop an efficient method to solve the subspace optimal diagonal preconditioning SDP problem using 
the interior point cutting plane method {\cite{sivaramakrishnan2002linear}}, 
leveraging ideas from semi-infinite programming (SIP)
to approximate the SDP constraint with an increasing number of linear constraints. 
The complexity of our approach is $\mathcal{O} ( (k^2 \mathcal{T}_{\lambda_{\min}} + k^4) \log ( 1/\varepsilon ) )$, which only involves several extremal eigenvalue computations if  $k$ is small. More importantly, our method can be implemented even when $\M$ is only accessible via matrix-vector products.
  
\item \emph{Column generation.} We develop a Frank-Wolfe type iterative method that monotonically reduces condition number by solving a sequence of SDPs over a basis of size $k=2$. 
This algorithm combines the idea of column generation and semi-infinite duality. 
The algorithm operates in a two dimensional subspace, and requires solving an SDP of only 3 scalar variables at each iteration. 
Moreover, our result suggests that extremal eigenvectors contain information on reducing the matrix condition number. 
This observation may be of independent interest in numerical analysis.
\end{enumerate}

\paragraph{Structure of the paper}

The rest of the manuscript is organized as follows. Section \ref{sec:preliminaries} discusses some notations and preliminaries; Section \ref{sec:subspace} introduces the formulation of the optimal diagonal preconditioner in a subspace, and proposes a randomized approach to improve its performance; Section \ref{sec:semiinf} shows how to efficiently solve the subspace formulation using techniques from semi-infinite programming; Section \ref{sec:colgen} introduces the column generation idea from linear programming to further enhance scalability;
Section \ref{sec:exp} presents numerical experiments to  demonstrate the practical appeal of our proposed method. 

%% file: sec_preliminaries.tex
\section{Preliminaries} \label{sec:preliminaries}

\paragraph{Notation.} Throughout the paper, we use bold letters $\A, \tma$ to denote matrices and vectors.
We use $\langle \cdot, \cdot \rangle$ to denote vector or matrix inner products and use $\| \cdot \|$ to denote the Euclidean norm; $\| \cdot \|_p$ denotes matrix or vector $p$-norm; $\tmop{tr}
(\A) = \sum_{i=1}^n a_{i, i}$ denotes the trace of a matrix; $\e_i$ denotes the $i$-th column of the
identity matrix $\I$; $\0$ denotes zero matrix or a zero vector; $\succeq$
denotes the partial order defined over the cone of positive semi-definite
matrices $\mathbb{S}_+^n$: $\A \succeq \textbf{B}$ if and only if $\A - \mathbf{B}$ is positive semi-definite. Given a positive definite matrix $\M \in
\mathbb{S}_{+ +}^n$, its maximum and minimum eigenvalues are given by
$\lambda_{\max} ( \M )$ and $\lambda_{\min} ( \M )$
respectively, and its condition number is defined by $\kappa ( \M )
\assign \frac{\lambda_{\max} ( \M )}{\lambda_{\min} ( \M
)}$. Given a column vector $\tmd \in \mathbb{R}^n$, $\diagm(
\tmd )$ denotes the diagonal matrix in $\mathbb{R}^{n \times n}$ with $\tmd$ on the diagonal, 
and $\tmd = \diag(\D)$ extracts the diagonal from $\D$. 
We use capital $\D$ and $\diagm(\tmd)$ interchangeably when their correspondence is clear from context. 
We define index set $[n]:=\{1,\ldots n\}$.\\

We start from a recent SDP formulation of the optimal diagonal preconditioning problem \cite{qu2022optimal}.

\begin{thm}[SDP formulation of optimal diagonal preconditioning \cite{qu2022optimal}]\ \label{thm:optprecond}

Given a positive definite matrix $\M \in \mathbb{S}_{+ +}^n$, 
the optimal diagonal preconditioner $\D^\star = \diagm(\tmd^\star)$ and solution to \eqref{eqn:formulation}
can be obtained by solving the semi-definite optimization problem
\begin{equation} \label{odp}
\begin{array}{ll}
\maxf{\tau, \tmd} & \qquad \tau \\
\textup{subject to} & \M \tau \preceq \diagm( \tmd ) \\
                   & \M \succeq \diagm( \tmd ) \\
                   & \tmd \geq 0
\end{array}
\end{equation}
with variables $\tmd \in \mathbb{R}^n_{++}$ and $\tau \in \mathbb{R}$. 
The solution $\tau^* = \kappa^{-1}$ equals the inverse condition number at the solution.
\end{thm}
Denote by $\tau^{\star}\leq 1$ the optimal value of \eqref{odp}, and
by $\kappa^{\star}:=\min_{\D \succeq \0\text{ diagonal}} \kappa ( \D^{- 1/2} \M \D^{- 1/2})$ the optimal condition number. We have $\kappa^{\star} = 1 / \tau^{\star}$. 
A feasible solution $(\tau,\tmd)$ to \eqref{odp} yields a preconditioner $\D=\tmop{diagm} ( \tmd )$ such that $\kappa(\D^{- 1/2} \M \D^{- 1/2}) \leq 1/\tau$. 
\eqref{odp} solves the optimal diagonal preconditioning problem by \emph{maximizing} the inverse condition number. The SDP formulation provides a theoretically viable way to compute
the optimal diagonal preconditioner. 
However, to our knowledge, even when $\M$ is sparse, no existing method can solve problem~\eqref{odp} to high accuracy for large matrices. Starting from \eqref{odp}, this paper aims to find \textit{scalable and accurate} methods to compute an approximate optimal diagonal preconditioner. According to the experiments, our method reduces the condition number of sparse matrices with size up to $10^7$ by half in less than 200 seconds.

%% file: sec_subspace.tex
\section{Optimal diagonal preconditioner in a subspace} \label{sec:subspace}

We begin by expressing a diagonal preconditioner in terms of a basis $\{ \e_i \e_i^{\top}, i \in [n] \}$ of the space of diagonal matrices:
\begin{equation} \label{eqn:f} \tag{F}
\maxf{\z, \tau} ~\tau \quad \text{subject to}  \quad \M \tau - \sum_{i = 1}^n \e_i\e_i^\top z_i \preceq \0, \quad \sum_{i=1}^{n}\e_i \e_i^{\top} z_i - \M\preceq \0.	
\end{equation}
where  $\z \in \mathbb{R}^n$ and $\D = \sum_{i=1}^{n}\e_i \e_i^{\top}z_{i} \in \tmop{span} \{ \e_i \e_i^{\top}, i \in [n] \}$. 
We can further write \eqref{eqn:f} in the following vector form \eqref{eqn:v}
\begin{equation} \label{eqn:v} \tag{V}
\maxf{\z, \tau} ~\tau \quad \text{subject to}  \quad \M \tau - \tmop{diagm} \Big(\sum_{i=1}^{n}\e_{i}z_{i}\Big) \preceq \0, \quad \tmop{diagm} \Big(\sum_{i=1}^{n}\e_{i}z_{i}\Big) - \M\preceq \0.	
\end{equation}
We see that \eqref{eqn:v} searches for the optimal diagonal preconditioner $\tmd^\star$ in $\mathbb{R}^n$ using the standard basis for the vector space of diagonal preconditioners. 
However, this basis has serious drawbacks: notably, the basis elements $\e_i \e_i^\top$ are not themselves valid preconditioners since they are only rank one. 
As a result, searching over the standard basis requires solving an SDP of $n + 1$ variables, which is prohibitive in practice for large $n$. \\ 

Now we introduce our first dimension reduction idea: 
instead of using $\mathbb{R}^n$,  we focus on a subspace of diagonal matrices represented by $\mathcal{S}_k := \tmop{span} \{\tmd_1, \ldots, \tmd_k\}\subseteq \mathbb{R}^n$ with $\dim (\mathcal{S}_k) = k \leq n$. Notably, these basis elements $\{\tmd_i\}$ can themselves represent positive definite diagonal preconditioners and thus lead to sensible problems even for $k \ll n$.
This proposal leads to the following SDP: 
\begin{equation}
\label{eqn:odp-subspace}
\tag{S}\qquad \tau_k^{\star} \assign \max_{\z, \tau} ~\tau \quad \text{subject to}  \quad \M \tau - \sum_{i = 1}^k \D_i z_i \preceq \0, \quad \sum_{i = 1}^k \D_i z_i - \M\preceq \0,
\end{equation}
and recall that $\D_i = \tmop{diagm}(\tmd_i)$. We refer to \eqref{eqn:odp-subspace} as the (subspace) optimal diagonal preconditioning SDP. \eqref{eqn:odp-subspace} generalizes \eqref{eqn:f} by replacing the basis matrices $\e_1 \e_1^{\top},\dots,\e_n \e_n^{\top}$ with $\D_1,\dots,\D_k$, and \eqref{eqn:odp-subspace} is feasible whenever $\mathcal{S}_k$ contains strictly positive vectors.
Intuitively, it finds the combination $\sum_{i = 1}^k \D_i z_i$ of given diagonal matrices that minimizes $\kappa((\sum_{i = 1}^k \D_i z_i)^{-1/2}\M(\sum_{i = 1}^k \D_i z_i)^{-1/2})$. Letting $\kappa_k^{\star}$ denote this minimal condition number achievable by preconditioners in $\mathcal{S}_k$, we have $\tau_k^{\star} = 1/\kappa_k^{\star}$. The subspace formulation \eqref{eqn:odp-subspace} lays the foundation of our algorithms. 

\subsection{Optimal diagonal preconditioner in a deterministic subspace}

An important question concerning the dimension reduction idea is how to construct the basis $\{\tmd_1, \ldots, \tmd_k\}$. We begin with the simplest setting, where a set of deterministic diagonal preconditioners based on $\M$ are readily available. Along with the development of iterative solvers, several popular heuristic diagonal preconditioners have been developed, such as the Jacobi preconditioner, Ruiz scaling \cite{ruiz2001scaling}, and the diagonal approximate inverse preconditioner \cite{benzi1999comparative}. These preconditioners are inexpensive to compute and behave favorably in practice. Therefore, they are suitable candidates to construct the deterministic subspace $\mathcal{S}$.\\ 

Given a set of $k$ heuristic diagonal preconditioners $\mathcal{H} \assign \{ \tmd_{h_1}, \ldots, \tmd_{h_k} \}$, the
heuristic subspace is given by $\mathcal{S}_{\mathcal{H}} \assign \tmop{span} \mathcal{H} $. Since each heuristic preconditioner is strictly positive, $\mathcal{S}_{\mathcal{H}}
\cap \mathbb{R}_{+ +}^n \neq \varnothing$. Moreover, the optimal preconditioner in $\mathcal{H}$ does at least as well as the best individual basis preconditioner.
\begin{thm}\label{thm:1}
  Given subspace $\mathcal{S}_k = \mathcal{S}_{\mathcal{H}} = \tmop{span} \{
  \tmd_{h_1}, \ldots, \tmd_{h_k} \}$, let
  $\D_\mathcal{H}^{\star}:=\sum_{i=1}^k \D_{h_i}z_i^\star$ denote the optimal diagonal preconditioner in $\mathcal{H}$, where $\z^\star$ solves \eqref{eqn:odp-subspace}. Then $\kappa((\D_\mathcal{H}^{\star})^{-1/2}\M (\D_\mathcal{H}^{\star})^{-1/2}) \leq \min_{i \in [k]} \{ \kappa (
     \D_{h_i}^{- 1/2} \M \D_{h_i}^{- 1/2}) \}$.
\end{thm}

\prove{
\begin{proof}
We recover each preconditioner's condition number by letting $\z = \e_i, i\in[k]$. By the optimality of SDP, we complete the proof.
\end{proof}
}

\begin{coro}
If ${\mathrm{\mathbf{1}}} \in \mathcal{H}$, then $\kappa((\D_\mathcal{H}^{\star})^{-1/2}\M(\D_\mathcal{H}^{\star})^{-1/2}) \leq \kappa ( \M )$.
\end{coro}

\begin{rem} \label{rem:1}
When $\mathcal{S}_k = \mathcal{S}_1 = \{\mathbf{1}\}$, \eqref{eqn:odp-subspace} simply computes the condition number of $\M$. In other words, the SDP formulation provides an alternative way to compute the condition number of a matrix.
\end{rem}

Now our approach constructs a better preconditioner in a deterministic subspace spanned by heuristic preconditioners. 
Experimental results demonstrate that it sometimes achieves significant reductions in condition number compared to individual basis preconditioners. 
However, there are drawbacks to using only heuristic diagonal preconditioners. First, we are limited by the availability of such preconditioners. Moreover, the whole-space optimal preconditioner $\tmd^\star$ may not be contained in $\mathcal{S}_\mathcal{H}$, and it is difficult to quantify the sub-optimality of $\tmd_\mathcal{H}^{\star}$ compared to $\tmd^\star$. These limitations motivate us to consider \emph{randomly} generated basis elements.
\subsection{Optimal diagonal preconditioner in a randomized subspace}
This section considers $\mathcal{S}_k$ to be a randomized subspace. Unlike the limited pool of heuristic preconditioners, this approach allows us to generate many random diagonal matrices arbitrarily. We sample $\tmd_1,
\ldots, \tmd_k$ i.i.d. from some distribution $\mathcal{P}$, and consider $\mathcal{S}_k = \tmop{span} \{ \1, \tmd_1, \ldots, \tmd_k \}$.\\

Now $\tau_k^{\star} = 1 / \kappa_k^{\star}$ is a random variable. Intuitively, as we increase $k$, $\dim (\mathcal{S}_k)$ increases along with the expected objective $\mathbb{E} [\tau_k^{\star}]$, thereby improving the optimal diagonal preconditioner in the randomized subspace. The precise behavior of $\kappa_k^{\star}$ relative to $\kappa^\star$ as $k$ increases is important because it informs us on the magnitude of $k$ in practice that can guarantee a good diagonal preconditioner by solving \eqref{eqn:odp-subspace}. To analyze the effect of increasing $k$, we choose a subgaussian distribution $\mathcal{P}$. Using the theory of random projection, we can establish the following two lemmas. For clarity of exposition, we leave all the proofs in this section in the appendix.

\begin{lem}[Approximate optimality: medium $k$] 
\label{lem:2}  
  Letting $\tau_k^{\star}, k\leq n$ be the optimal  value of \eqref{eqn:odp-subspace}. There exists some universal constant $L > 0$
  such that at least with probability $0.9$,
  \[ \tau_k^{\star} \geq \tau^{\star} - 20 \sqrt{1 + L n^{- 1 / 2}}
     \tfrac{\| \tmd^{\star} \|}{\lambda_{\min} ( \M )}
     \sqrt{1 - \tfrac{k}{n}}, \]
  where $\| \tmd^{\star} \|$ is the least-norm optimal diagonal preconditioner
  and $\lambda_{\min} ( \M )$ is the minimum eigenvalue of $\M$.
\end{lem}

\begin{lem}[Approximate optimality, large $k$]\label{lem:3}
  
  Under the same conditions as \tmtextbf{Lemma \ref{lem:2}}, there exist universal constants $s, S > 0$ such that at least with probability 
  $0.9$,
  \[ \tau_k^\star \geq \tau^{\star} - 2 \sqrt{\tfrac{\log (60 s n)}{S}} \tfrac{\|
     \tmd^{\star} \|}{\lambda_{\min} ( \M ) \sqrt{k}} . \]
\end{lem}

Combining these two results, we get the bound on $\kappa_k^{\star}$.
\begin{thm} \label{thm:2}
  Let $\kappa^\star_k$ be the optimal condition number given randomized subspace
  $\mathcal{S}_k$. We have, with probability at least $0.8$, that
  \[ \kappa^\star_k \leq \min \Bigg\{ \tfrac{\kappa^{\star}}{1 - G_1 \kappa^{\star}
     \sqrt{\frac{\log n}{k}}}, \tfrac{\kappa^{\star}}{1 - G_2 \kappa^{\star}
     \sqrt{1 - \frac{k}{n}}}, \kappa ( \M ) \Bigg\}, \]
  for all $n \geq k \geq \max \{ n - \frac{n}{(G_2 \kappa^{\star})^2}, (G_1
  \kappa^{\star})^2 \log n \}$, where $G_1, G_2 > 0$ are universal
  constants.
\end{thm}

An important implication of \tmtextbf{Theorem \ref{thm:2}} is that when $k$ increases from small to moderate values, expanding the subspace with more random diagonal matrices only \emph{marginally} contributes to reductions in condition number. Recall that the convergence of iterative methods typically has $\mathcal{O} ( {\kappa})$ dependency on the condition number. Therefore, even if the optimal condition number $\kappa^\star_k$ of \eqref{eqn:odp-subspace} satisfies $\kappa_k^\star \leq (1+\alpha) \kappa^{\star}$ for $\alpha = \mathcal{O}(1)$, the acceleration can still be 
significant. As we will illustrate in our experiments, we only need
$k = \mathcal{O} (1)$ to achieve this desirable improvement in
condition number.  \\

So far, we have focused on constructing the subspace $\mathcal{S}_k$ of basis matrices using heuristic or randomly generated matrices. However, there is still a challenge: to produce a useful preconditioner, a solution to \eqref{odp} should satisfy the SDP constraints since even tiny violations of those constraints will be magnified by the ill-conditioning of the original matrix $\M$. In practice, we find that matrices with condition numbers around $10^7$ require an absolute violation less than $10^{-10}$ to produce a useful preconditioner. However, even if $k = \mathcal{O}(1)$, most off-the-shelf methods are still not scalable for large $n$: interior point methods \cite{wolkowicz2012handbook} are robust and accurate, but they involve an $n \times n$ matrix decomposition in each iteration; the best first-order methods \cite{burer2003nonlinear, deng2022new, ding2023revisiting, ding2021optimal, o2016conic, yurtsever2021scalable, agrawal2020disciplined} are more scalable, but they are in general not accurate enough \cite{tu2014practical}. Moreover, in large problems, the matrix $\M$ itself may not be available and can only be accessed through matrix-vector products of the form $\M \mathbf{v}$. In view of these challenges, we propose an effective solution by leveraging an old branch of SDP research: semi-infinite programming (SIP) and interior point cutting plane methods \cite{sivaramakrishnan2002linear}. The cutting plane method is widely believed to be inefficient for general SDPs. However, as we will show in the next section and in numerical experiments, it turns out to be the most competitive method for the optimal preconditioning problem.

%% file: sec_semiinf.tex
\section{A semi-infinite approach to the optimal preconditioning SDP} \label{sec:semiinf}
In this section, we improve the scalability of the SDP approach to optimal diagonal preconditioning using semi-infinite programming (SIP) and cutting plane methods. The idea is to approximate semi-definite conic constraints with a set of linear constraints, thereby avoiding explicitly solving the whole SDP. This is a second application of the general dimension reduction idea.\\

Given SDP conic constraint $\bs \succeq \0$, its SIP reformulation turns it into an infinite number of \emph{linear}
constraints $\langle \tmv, \bs \tmv \rangle \geq 0, \text{ for all $\|
   \tmv \| = 1$}$. Applying this reformulation to \eqref{eqn:odp-subspace} gives 
\begin{eqnarray} \label{eqn:sip}
  \maxf{\tau, \z} & \tau & \\
  \text{subject to~~~} & \textstyle \sum_{i = 1}^k \langle \tmv_1, \D_i \tmv_1
  \rangle z_i - \langle \tmv_1, \M \tmv_1 \rangle \tau
  \geq 0, & \text{for all~}  \| \tmv_1 \| = 1 \nonumber \\
  & \textstyle\sum_{i = 1}^k \langle \tmv_2, \D_i \tmv_2 \rangle z_i
  \leq \langle \tmv_2, \M \tmv_2 \rangle & \text{for all~}  \|
  \tmv_2 \| = 1 \nonumber.
\end{eqnarray}
Using the SIP formulation, we can reduce the SDP \eqref{eqn:odp-subspace} to a linear program (LP) of $\mathcal{O}(k)$ variables
and an infinite number of constraints. In practice, it is impossible to enforce all the constraints. However, given a candidate solution $( \tau', \z')$, we know that an extremal eigenvalue separation oracle exists for checking whether the SDP constraints are satisfied, and if not, the oracle gives an separating constraint. Take the first block of constraints as
an example. We can perform the extremal eigenvalue computation to determine if
$\lambda_{\min} ( \sum_{i = 1}^k \D_i z_i' - \M \tau'
) \geq 0$. If not, this oracle returns a separating hyperplane parameterized by $\tmv' \in \mathbb{R}^n$
such that
\[ \langle \tmv', ( \textstyle\sum_{i = 1}^k \D_i z_i' - \M
   \tau ) \tmv' \rangle = \sum_{i = 1}^k \langle \tmv', \D_i
   \tmv' \rangle z_i' - \langle \tmv', \M \tmv' \rangle
   \tau' < 0. \]
   
By adding to LP the new constraint $\sum_{i = 1}^k \langle \tmv', \D_i
   \tmv' \rangle z_i - \langle \tmv', \M \tmv' \rangle
   \tau \geq 0$ parameterized by $\tmv'$, we cut $( \tau', \z')$ from the feasible region. We then repeat this process and compute an increasing number of linear constraints (cutting planes) that successively improve the approximation of the SDP feasible region. In practice, we start from a finite set of cutting planes parameterized by $\mathcal{V}^t_1, \mathcal{V}^t_2$, and solve the LP problem
   \begin{eqnarray} 
   \label{eq:finite-SIP}
  \maxf{\tau, \z} & \tau & \\
  \text{subject to~~~} & \textstyle\sum_{i = 1}^k \langle \tmv_1, \D_i \tmv_1
  \rangle z_i - \langle \tmv_1, \M \tmv_1 \rangle \tau
  \geq 0, & \text{for all~}   \tmv_1 \in \mathcal{V}^t_1, \nonumber \\
  & \textstyle\sum_{i = 1}^k \langle \tmv_2, \D_i \tmv_2 \rangle z_i 
  \leq \langle \tmv_2, \M \tmv_2 \rangle & \text{for all~} \tmv_2 \in \mathcal{V}^t_2 \nonumber.
\end{eqnarray}
After obtaining the optimal solution $( \tau^t, \z^t )$ to the finite LP \eqref{eq:finite-SIP}, we invoke the separation oracle. 
If the oracle does not return a separating hyperplane (up to some tolerance $\varepsilon$), it certifies that $( \tau^t, \z^t)$ is optimal for \eqref{eqn:sip}, therefore \eqref{eqn:odp-subspace}. On the other hand, if the oracle returns a separating hyperplane, we add the associated cutting plane to the finite LP and iterate till convergence. We summarize this procedure in \textbf{Algorithm \ref{alg:1}}.

\begin{algorithm}[h]\label{alg:1}
\KwIn{Initial constraint set $\mathcal{V}^1_1, \mathcal{V}^1_2$, tolerance parameter $\varepsilon$, maximum iteration $T$\;}
\SetKwRepeat{Do}{do}{while}
\For{$t = 1,\ldots,T$}{\

	\textbf{solve}  \eqref{eq:finite-SIP} with constraints defined by $\mathcal{V}^t_1, \mathcal{V}^t_2$ and obtain solution $(\tau^t, \z^t)$
	
	\textbf{compute $\lambda_{\min}(\sum_{i = 1}^k \D_i z_i^t - \M \tau^t)$} and eigenvector $\tmv_1^t$\\
	\textbf{compute $\lambda_{\min}(\M - \sum_{i = 1}^k \D_i z_i^t )$} and eigenvector $\tmv_2^t$
	
	\If{$\lambda_{\min}(\sum_{i = 1}^k \D_i z_i^t  - \M \tau^t) \leq -\varepsilon$}{
		$\mathcal{V}^{t + 1}_1 = \mathcal{V}^t_1 \cup \{\tmv_1^t\}$
	}
	\If{$\lambda_{\min}(\M - \sum_{i = 1}^k \D_i z_i^t ) \leq -\varepsilon$}{
		$\mathcal{V}^{t + 1}_2 = \mathcal{V}^t_2 \cup \{\tmv_2^t\}$
	}    
}

\caption{Cutting plane approach for \eqref{eqn:odp-subspace}}
\end{algorithm}
We make several remarks on the practical implementation of \textbf{Algorithm \ref{alg:1}}.
\begin{rem}
In practice, it is not necessary to compute the exact extremal eigenvector for a separation oracle. It is sufficient to find a direction $\tmv$ such that the corresponding quadratic form is negative.
\end{rem}
\begin{rem}
We can adopt an iterative extremal eigenvalue routine, such as the Lanczos method, as the separation oracle. As we observe in the experiments, an iterative oracle often takes $\mathcal{O} (\mathsf{nnz}(\M)) $ time. More importantly, iterative oracles only assume the availability of matrix-vector products. As a result, our method is applicable even if we only have access to $\M$ through black-box matrix-vector operations.
\end{rem}
Next, we analyze the complexity of \textbf{Algorithm \ref{alg:1}}.

\begin{lem} \label{lem:4}
  {\cite{krishnan2003properties}} Assume $\mathbf{1} \in \mathcal{S}_k$ and that each LP subproblem is solved to $\varepsilon$-accuracy. 
  The interior point cutting plane method requires
$K =\mathcal{O} ( (k^2 \mathcal{T}_S + k^4) \log ( \tfrac{\| \M
     \|}{\varepsilon} ) )$
  arithmetic operations to output an $\varepsilon$-feasible solution $(\hat{\tau},\hat{\z})$ such
  that
  \[ \textstyle\sum_{i = 1}^k \D_i \hat{z}_i - \M \hat{\tau} \succeq - \varepsilon
     \cdummy \I \qquad \M - \sum_{i = 1}^k \D_i \hat{z_i}  \succeq -
     \varepsilon \cdummy \I, \]
     where $\mathcal{T}_S$ is the cost of each separation oracle.
\end{lem}

\prove{
\begin{proof}
	Now that each LP is solved to accuracy $\varepsilon$, \textbf{Theorem 3} of \cite{krishnan2003properties} gives us the result.
\end{proof}
}
The next result shows we can turn an $\varepsilon$-feasible solution into an approximate optimal SDP solution.

\begin{lem}
  \label{lem:5} Under the same condition as \textbf{Lemma \ref{lem:4}}, given an $\varepsilon$-feasible solution to \eqref{eqn:odp-subspace}, we can get a feasible solution that achieves an objective value $\tau' = \tau_k^{\star} - \frac{2 \varepsilon}{\lambda_{\min} ( \M
     )}$. Moreover, we can extract a preconditioner that gives a condition number no larger than $1 / \tau'$.
\end{lem}

\prove{
\begin{proof}

Now that we assume $\1 \in \mathcal{S}_k$, without loss of generality we
assume $\sum_{i = 1}^k \D_i a_i = \I$.  Suppose we arrive at $(
\hat{\tau}, \hat{\z} )$ such that
\begin{align}
  \textstyle\M - \sum_{i = 1}^k \D_i \hat{z}_i & \succeq - \varepsilon \cdummy \I
  \nonumber\\
  \textstyle  \sum_{i = 1}^k \D_i \hat{z}_i - \M \hat{\tau} & \succeq - \varepsilon
  \cdummy \I . \nonumber
\end{align}
Then we take $( \tau', \z' ) = ( \hat{\tau} - \tfrac{2
\varepsilon}{\lambda_{\min} ( \M )}, \hat{\z} - \varepsilon
\tma )$ and deduce that
\begin{align}
 \textstyle \M - \sum_{i = 1}^k \D_i z'_i ={} & \M - \textstyle\sum_{i = 1}^k \D_i z'_i \nonumber\\
  ={} & \M -  \textstyle\sum_{i = 1}^k \D_i (\hat{z}_i - \varepsilon a_i) \nonumber\\
  ={} & \M - \textstyle \sum_{i = 1}^k \D_i \hat{z}_i + \textstyle( \sum_{i = 1}^k \D_i a_i
  ) \varepsilon \nonumber\\
  ={} & \M - \textstyle\sum_{i = 1}^k \D_i \hat{z}_i + \varepsilon \cdummy \I \label{proof-1} \\
  \succeq{} & \0, \nonumber
\end{align}
where \eqref{proof-1} uses the assumption that $\sum_{i = 1}^k \D_i a_i = \I$.
On the other hand,
\begin{align}
  \textstyle\sum_{i = 1}^k \D_i z'_i - \M \tau' ={} & \textstyle\sum_{i = 1}^k \D_i (\hat{z}_i -
  \varepsilon a_i) - \M ( \hat{\tau} - \frac{2
  \varepsilon}{\lambda_{\min} ( \M )} ) \nonumber\\
  ={} & \textstyle\sum_{i = 1}^k \D_i \hat{z}_i - \M \hat{\tau} - \varepsilon \cdummy \I +
  \frac{2 \varepsilon}{\lambda_{\min} ( \M )} \M \nonumber\\
  \succeq{} & - 2 \varepsilon \cdummy \I + \tfrac{2 \varepsilon}{\lambda_{\min}
   ( \M )} \M \nonumber \\
  \succeq{} & \0 \label{proof-2},
\end{align}

where \eqref{proof-2} uses the fact $\frac{1}{\lambda_{\min}}\M \succeq \I$. Therefore, $( \tau', \z')$ is a feasible solution to the original
problem. Now fix $\hat{\tmd} = \D \hat{\z}'$ and let $\hat{\D} := \text{diagm}(\hat{\tmd})$. We know that $(\hat{\tau}, 1)$ is a feasible solution to the following SDP
\begin{eqnarray*}
  \maxf{\tau, z} & \tau 
  ~~~~\text{subject to~~~} & \M - \hat{\D} z \succeq \I,~~
   \hat{\D} z - \M \tau \succeq \0. 
\end{eqnarray*}
This SDP computes $\kappa(\hat{\D}^{-1/2}\M\hat{\D}^{-1/2})$. By optimality $\kappa(\hat{\D}^{-1/2}\M\hat{\D}^{-1/2})\leq 1/\tau'$ this completes the proof.
\end{proof}
}
Finally, we combine the previous results to provide the overall complexity of the cutting plane approach. 

\begin{thm} \label{thm:3}
Given $\delta \in (0, 1)$ and $k =\mathcal{O} (1)$, it takes
$$K =\mathcal{O} \Big(\mathcal{T}_S\log \big( \delta^{-1}     \kappa(\M) \big) \Big)$$
  arithmetic operations to find an approximate optimal preconditioner in $\mathcal{S}_k$ with condition number $\kappa'$, such that
  $\kappa' \leq (1 + \delta) \kappa^{\star}_k$. In other words, for $\delta = \mathcal{O}(1)$, we only need
  $\mathcal{O} (\mathcal{T}_S)$ arithmetic operations to compute an approximate optimal preconditioner in $\mathcal{S}_k$.
\end{thm}

\prove{
\begin{proof}
Given $k =\mathcal{O} (1)$, if we take $K = 2 \mathcal{T}_S \log ( \frac{4}{\delta} \kappa ( \M ) ) ) =\mathcal{O}
( \mathcal{T}_S \log ( \frac{4}{\delta} \kappa ( \M )
) )$, then
\[ K = 2\mathcal{T}_S \log \big( \tfrac{4}{\delta} \kappa ( \M )
   \big) \geq \mathcal{T}_S \log \big( \tfrac{4}{\delta} \kappa ( \M
   ) \kappa_k^{\star} \big) =\mathcal{T}_S \log \big(
   \tfrac{4}{\delta} \tfrac{\| \M \|}{\lambda_{\min} ( \M
   )} \kappa_k^{\star} \big) \]
and by \tmtextbf{Lemma \ref{lem:5}}, we correspondingly have $\varepsilon = \frac{\delta}{4}
\lambda_{\min} ( \M ) \tau_k^{\star}$, giving
\[ \tau' = \tau_k^{\star} - \tfrac{2}{\lambda_{\min} ( \M )} \cdummy
   \tfrac{\delta}{4} \lambda_{\min} ( \M ) \tau_k^{\star} = ( 1
   - \tfrac{\delta}{2} ) \tau_k^{\star} \geq \tfrac{1}{1 + \delta}
   \tau_k^{\star}, \]
where $1 - \frac{\delta}{2} \geq \frac{1}{1 + \delta}$ for $\delta \in (0,
1)$. Then we take the inverse on both sides and deduce that
\[ \kappa' = 1 / \tau' \leq \tfrac{1 + \delta}{\tau_k^{\star}} = (1 + \delta)
   \kappa_k^{\star} . \]
This completes the proof.
\end{proof}
}

Till now, we have established scalable and efficient subroutines to compute approximate optimal diagonal preconditioners using the simple yet powerful idea of dimension reduction in two distinct ways.  The first dimension reduction searches for an optimal preconditioner in a subspace of dimension $k$ instead of $n$. The second dimension reduction transforms the optimal preconditioning SDP into an LP. However, $\mathcal{S}_k$ is constructed somewhat ``exogenously'', as we need to specify the basis elements beforehand, either based on existing heuristic preconditioners or randomly generated matrices. This naturally motivates the following question:
\begin{center}
\textit{Does $\M$ \textbf{itself} contain information on how to construct a better preconditioner?}
\end{center}

In the next section, we answer this question affirmatively. We show that after computing $\kappa(\D^{-1/2}\M\D^{-1/2})$ for a valid preconditioner $\tmd$, we automatically obtain an improving direction for $\tmd$ that further reduces the condition number ``optimally'' in some sense. This observation allows us to design an efficient iterative procedure that outputs preconditioners that monotonically reduce the condition number at every iteration, eliminating the need to rely on exogenously generated basis matrices.

%% file: sec_colgen.tex
\section{Preconditioner optimization via column (matrix) generation}
\label{sec:colgen}

In this section, we show how to extract new preconditioners automatically
based on the solution to the current subspace optimal preconditioning SDP. Recall that the condition number
$\kappa ({\D}^{-1/2} \M{\D}^{-1/2} )$ can be estimated by computing the extremal eigenvalues $\lambda_{\max} ({\D}^{-1/2} \M{\D}^{-1/2}
), \lambda_{\min} ({\D}^{-1/2} \M{\D}^{-1/2} )$ and $\kappa ({\D}^{-1/2} \M{\D}^{-1/2}
) = \frac{\lambda_{\max} ({\D}^{-1/2} \M{\D}^{-1/2} )}{\lambda_{\min} ({\D}^{-1/2} \M{\D}^{-1/2}
)}$ (e.g. {\cite{urschel2021uniform}}). 
Our key observation is that the extremal eigen\emph{vectors}
provide information that can be used to improve  $\tmd$.
This section shows how to exploit this idea 
efficiently using LP column generation, semi-infinite duality, and another dimension
reduction idea.  We begin with SDP duality.

\subsection{Dual interpretation of optimal diagonal preconditioning}

Given the subspace $\mathcal{S}_k$ of preconditioners spanned by $\{\tmd_1,\dots,\tmd_k\}$, the dual problem of SDP \eqref{eqn:odp-subspace} is given by
\begin{equation}
\begin{array}{lrl}\label{eqn:dm}
\tag{DS}& \qquad \tau_k^{\star} = \min_{\X_1, \X_2} & \langle \M, \X_1 \rangle \\
&  \text{subject to} & \langle \M, \X_2 \rangle = 1 \\
&  & \langle \D_i, \X_1 - \X_2 \rangle = 0, i \in [k]\\
&  & \X_1, \X_2 \succeq \0, 
\end{array}
\end{equation}
where $\X_1$ and $\X_2$ are the dual variables for the two SDP constraints in  \eqref{eqn:odp-subspace}.
Strong duality holds since $( \X_1, \X_2
) = ( \frac{1}{n} \M^{- 1}, \frac{1}{n} \M^{- 1} )$ is a feasible solution satisfying Slater's condition. 
Recall that the optimal whole-space diagonal preconditioner is denoted by $\D^{\star}$. We make the following observations about the dual formulation to provide intuitions and help motivate our matrix generation method.

\begin{enumerate}[leftmargin=*,itemsep=2pt,label=\textbf{\arabic*)}]
\item The dual problem of \eqref{odp}, which we denote by \eqref{eqn:df}, is a special case of \eqref{eqn:dm} by taking $\D_i = \e_i \e_i^\top, i \in [n]$.
\begin{equation}
\begin{array}{lrl}\label{eqn:df}
\tag{DF}& \qquad  \minf{\X_1, \X_2} & \langle \M, \X_1 \rangle \\
&  \text{subject to} & \langle \M, \X_2 \rangle = 1 \\
&  & \langle \e_i \e_i^\top, \X_1 - \X_2 \rangle = 0, i \in [n]\\
&  & \X_1, \X_2 \succeq \0. 
\end{array}
\end{equation}

\item The condition $\langle \D^\star, \X_1 - \X_2 \rangle = 0$ on the dual solution $(\X_1, \X_2)$ guarantees that \eqref{eqn:dm} identifies the whole-space optimal diagonal preconditioner $\D^\star$. This is because when the condition holds, adding $\D^\star$ to the constraints in \eqref{eqn:dm} cannot improve its objective. Equivalently, adding $\tmd^\star$ to $\{\tmd_1,\dots,\tmd_k\}$ in \eqref{eqn:odp-subspace} cannot improve the resulting condition number, i.e., $\{\tmd_1,\dots,\tmd_k\}$ already spans $\tmd^\star$.

\item \eqref{eqn:df} imposes the
constraint $\tmop{diag} ( \X_1 - \X_2 ) = \0$, which is a sufficient condition for $\langle \D^\star, \X_1 - \X_2 \rangle = 0$. Hence \eqref{eqn:df} is guaranteed to identify the optimal preconditioner $\D^\star$. 

\item Although \eqref{eqn:df} enforces the sufficient condition $\tmop{diag} ( \X_1 - \X_2 ) = \0$ for recovering $\D^\star$, it is inefficient in the following sense. In \eqref{eqn:df}, we impose $\tmop{diag} ( \X_1 - \X_2
) = \0$ using $n$ linear constraints involving $\e_i \e_i^{\top}$, while in \eqref{eqn:dm}, each constraint $\langle \D_i, \X_1 - \X_2 \rangle =
\sum_{j = 1}^n d_{i j}
\langle  \e_j \e_j^{\top}, \X_1 - \X_2 \rangle = 0$ 
is a linear combination of the constraints in \eqref{eqn:df}. Adding the optimal diagonal preconditioner $\D^{\star}$ gives the strongest linear constraint and maximizes the primal objective. Of course, this is infeasible since $\D^\star$ is unknown.

\item Relatedly, in \eqref{eqn:df} all $n$ basis matrices $\e_i \e_i^\top$ are required to create a valid (positive definite) diagonal preconditioner, whereas for \eqref{eqn:dm}, even $k=1$ (i.e., a single preconditioner) yields a meaningful problem. For example, when $\mathcal{S}_1 = \{\1\}$, \eqref{eqn:dm} has objective $1 / \kappa(\M)$, which corresponds to computing the condition number of $\M$. Similarly, given any strictly positive diagonal matrix $\D$, when $\mathcal{S}_1 = \{\tmd\}$ \eqref{eqn:dm} precisely computes (inverse of) the condition number $\kappa(\D^{-1/2}\M\D^{-1/2})$. As we gradually add more linear constraints, $\tau_k^{\star}$ is
non-decreasing and $\tau_k^{\star} = \tau^{\star} = 1 / \kappa ( (
\D^{\star} )^{-1/2} \M \D^{\star} )^{-1/2} )$ as soon as $\tmd^{\star} \in \tmop{span} \{
\tmd_1, \ldots, \tmd_k \}$.
\end{enumerate}
Duality provides valuable insights on how to construct a new preconditioner: now that adding a new preconditioner in the primal search space is equivalent to adding a linear constraint to the dual, finding a new preconditioner boils down to seeking the \emph{most effective} linear constraint. 

\subsection{Column generation for the optimal preconditioning SDP}

Suppose we have solved \eqref{eqn:odp-subspace} for some $\mathcal{S}_k$ with $\mathcal{S}_k \cap
\mathbb{R}_{+ +}^n \neq \varnothing$ and that we have the dual optimal
solution $( \X_1^{\star}, \X_2^{\star} )$ to \eqref{eqn:dm}. If $\tmop{diag} (
\X_1^{\star} - \X_2^{\star} ) = \0$, we must have found the whole-space optimal
diagonal preconditioner $\D^\star$ since $\langle \D^{\star}, \X_1^{\star} - \X_2^{\star} \rangle = 0$. Otherwise, $\tmop{diag} ( \X_1^{\star} -
\X_2^{\star} ) \neq \0$. We can identify an improving direction by finding the linear constraint that is \tmtextit{maximally violated} by $\X_1^{\star} - \X_2^{\star}$:
\begin{equation}\label{eqn:pricing}
  \hat{\tmd} = \underset{\| \tmd \|_p = 1}{\arg \max}~ ~
  \langle \tmd, \tmop{diag} ( \X_1^{\star} - \X_2^{\star} )
  \rangle,
\end{equation}
where for $p \in \{ 1, 2, \infty \}$ the problem has a closed form solution.
In the linear programming literature, \eqref{eqn:pricing} is known as the \emph{pricing problem} for
column generation \cite{lubbecke2005selected, oskoorouchi2007matrix, oskoorouchi2011interior}, and is used to generate a new column with maximum dual infeasibility. 
One slight difference here is that our method generates a diagonal matrix instead of generating a column. 
After obtaining $\hat{\tmd}$, we solve \eqref{eqn:odp-subspace} with $\mathcal{S}_{k+1}:=\tmop{span} \{\tmd_1, \ldots, \tmd_k\}\cup \{
\hat{\tmd} \} $ to obtain the new optimal diagonal preconditioner in the expanded subspace. Since $\X_1^{\star},
\X_2^{\star}$ are, in some sense, ``most infeasible'' with respect to the
constraint $\langle \hat{\D}, \X_1^{\star} - \X_2^{\star} \rangle
= 0$, we expect $\hat{\tmd}$ to perform better than a random matrix. \\

Two difficulties remain in implementing our method in practice:
First, how can the SIP approach, which works in the primal space,
efficiently construct the dual optimal variables $\X_1^{\star}, \X_2^{\star}$?
Second, as the number of basis matrices increases, the resulting SDP becomes more expensive. 
Somewhat interestingly, both challenges can be overcome
using semi-infinite duality and another dimension reduction idea.
\subsection{Semi-infinite duality and dual solution retrieval}
\begin{figure}
\centering
	\includegraphics[scale=0.5]{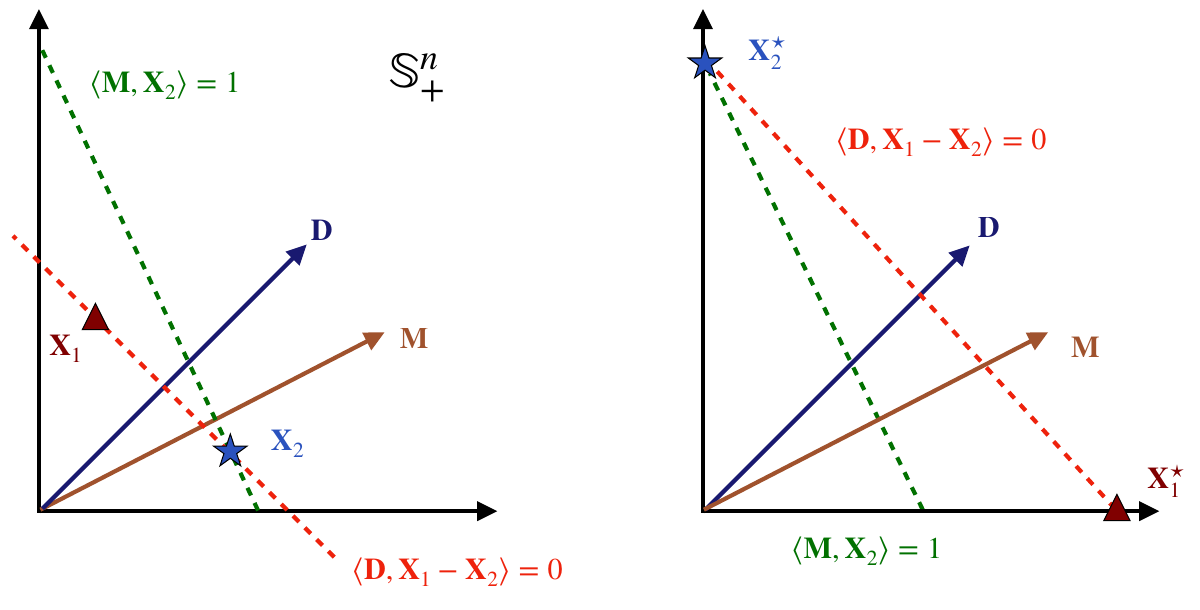}
\caption{Left: an illustration of the dual problem. Given $\X_2$ on the hyperplane $\langle \M, \X_2 \rangle =1$, 
we seek $\X_1$ such that \textbf{1)}. $\X_1 - \X_2$ is orthogonal to each $\D_i$ and \textbf{2)}. $\X_1$ minimizes $\langle \M, \X_1 \rangle$. Right: the optimal solution.}
\end{figure}
Now let us see how to obtain the dual solution. First,
write the dual problem of \eqref{eq:finite-SIP} associated with $\mathcal{V}_1,
\mathcal{V}_2$ as
\begin{eqnarray*}
  \minf{\X_1, \X_2, \x_1, \x_2} & \langle \M, \X_1 \rangle & \\
  \text{subject to} & \langle \M, \X_2 \rangle = 1 & \\
  & \langle \D_i, \X_1 - \X_2 \rangle = 0, & i \in [k]\\
  & \X_1 = \sum_{j = 1}^{| \mathcal{V}_1 |} x_{1 j} \tmv_1^j
  {\tmv_1^j}^{\top} & \\
  & \X_2 = \sum_{j = 1}^{| \mathcal{V}_2 |} x_{2 j} \tmv_2^j
  {\tmv_2^j}^{\top} & \\
  & \x_1, \x_2 \geq \0. & 
\end{eqnarray*}
After we certify the optimal solution to \eqref{eq:finite-SIP} also satisfies the SDP constraints, we can retrieve $\X_1^{\star} = \sum_{j = 1}^{|
\mathcal{V}_1 |} x_{1 j}^{\star} \tmv_1^j {\tmv_1^{j\top}}$ and $\X_2^{\star} =
\sum_{j = 1}^{| \mathcal{V}_2 |} x_{2 j}^{\star} \tmv_2^j {\tmv_2}^{j\top}$
from the optimal LP dual solution $(\x_1^\star, \x_2^\star)$. 
The new search direction relies only on the diagonal of $\X_1^{\star}$
and $\X_2^{\star}$, hence only $\mathcal{O} ((|
\mathcal{V}_1 | + | \mathcal{V}_2 |)n)$ arithmetic operations are needed.

\subsection{Dimension reduction for column generation}

As in LP column generation methods, subspace dimension $k$
increases as the algorithm progresses. Since \textbf{Algorithm \ref{alg:1}} has a $k^4$ dependency, the
cutting plane method slows down significantly as $k$ grows.
In the column generation literature \cite{lubbecke2005selected}, 
one solution is to drop some of the columns (here, diagonal matrices) when $k$ becomes large. 
However, this approach requires a careful strategy to decide both the size of $k$ and which matrices to drop. 
The structure of our problem allows a more efficient solution. 
Consider $k = 2$. We seek the best
linear combination of two matrices, 
$\tmd^{\text{{best}}} = z_1^{\star} \tmd_1 + z_2^{\star} \tmd_2$. 
After computing the best preconditioner $\tmd^{\text{{best}}}$ in the current subspace 
and obtaining a new direction $\hat{\tmd}$, 
instead of keeping $\tmd_1, \tmd_2$ in our subspace, we can
replace them by $\tmd^{\text{{best}}}$ and continue searching in the two-dimensional
subspace $\tmop{span} \{ \tmd^{\text{{best}}}, \hat{\tmd} \}$ (updating $\tmd_1 = \tmd^{\text{{best}}}, \tmd_2 = \hat{\tmd}$). 
By repeating this procedure, the algorithm monotonically reduces the condition number 
while maintaining a subspace of size $k = 2$. 
See \textbf{Figure \ref{fig:iter}} and \textbf{Algorithm \ref{alg:2}} for an illustration.

\begin{algorithm}[h]
\label{alg:2}
\KwIn{Initial basis $\{\tmd_1, \tmd_2\}$, maximum iteration $T$\;}
\SetKwRepeat{Do}{do}{while}
\For{$t = 1,\ldots,T$}{\

{\textbf{solve}  \eqref{eqn:odp-subspace} using SIP to obtain $\tmd^{\text{{best}}}=z_1^{\star} \tmd_1 + z_2^{\star} \tmd_2$ and $\tmop{diag}(\X_1^{\star} - \X_2^{\star})$}\\
{\textbf{solve} pricing problem \eqref{eqn:pricing} and get $\hat{\tmd}$}\\
{\textbf{let} $\tmd_1 \leftarrow \tmd^{\text{{best}}}$ and $\tmd_2 \leftarrow \hat{\tmd}$}
}
\caption{Iterating in the space of preconditioners}
\end{algorithm}

\begin{figure}[h]
\centering
\includegraphics[scale=0.55]{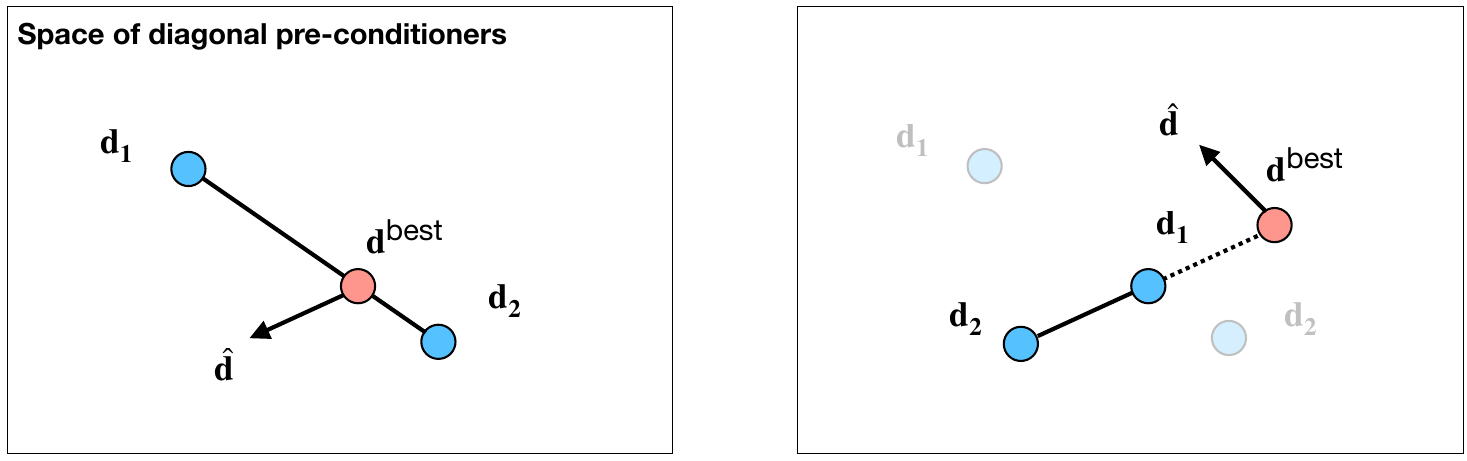}
\caption{Iterating in the space of preconditioners. In each iteration, we obtain the best linear combination between the best preconditioner so far $\tmd_1$ and an improving direction $\tmd_2$. After solving the SIP, we update the best preconditioner to $\tmd_1 \leftarrow \tmd^{\text{{best}}}$ and get a new improving direction $\tmd_2 \leftarrow \hat{\tmd}$ from the pricing problem. \label{fig:iter}}
\end{figure}
\begin{thm} \label{thm:4}
{\textbf{Algorithm} \ref{alg:2}} generates a sequence of improved preconditioners.
\end{thm}

\prove{
\begin{proof}
Assume $\mathcal{S}_2 = \tmop{span} \left\{ \tmd_1, \tmd_2 \right\}$ and we solve an SDP
for $(z_1^{\star}, z_2^{\star})$ such that $z_1^{\star} \tmd_1 + z_2^{\star}
\tmd_2$ gives the best condition number in $\mathcal{S}_2 .$ Now that the next
subspace $\mathcal{S}_2^+ = \tmop{span}~ \{ z_1^{\star} \tmd_1 + z_2^{\star} \tmd_2,
\hat{\tmd} \}$ contains $z_1^{\star} \tmd_1 + z_2^{\star} \tmd_2$,
 the optimal diagonal preconditioner in
$\mathcal{S}_2^+$ is no worse than $z_1^{\star} \tmd_1 + z_2^{\star} \tmd_2$.
This completes the proof.
\end{proof}
}

\begin{rem}
The procedure can be viewed as a (Frank-Wolfe style) iterative optimization in the space of preconditioners. The improving direction is obtained from SIP duality and the stepsize comes from another SIP. 
This procedure yields a simple, effective method to reduce the condition number, even though the general problem of minimizing the condition number is nonsmooth and nonconvex \cite{marechal2009optimizing}.
\end{rem}

\begin{rem}
Our approach shares spirits with the spectral bundle method \cite{helmberg2000spectral}. However, our method is motivated by SDP duality, column generation, and pricing problems instead of the conventional SDP penalty formulation, subgradient method, and model function. According to our tests, the cutting plane method performs more favorably than the spectral bundle method.
\end{rem}

\subsection{Implications of dimension reduction and extensions}
Importantly, the low dimension $k=2$ of \eqref{eqn:odp-subspace} guaranteed by \textbf{Algorithm \ref{alg:2}} offers additional computational benefits,
as SDPs have desirable rank properties when $k$ is small \cite{lemon2016low}:
generically, the rank of the optimal dual solution is 1 \cite{ding2021simplicity}.

\begin{thm}\label{thm:5}
Suppose $\tmop{span}~\{\tmd_1, \tmd_2\} \cap \mathbb{R}^n_{++} \neq \varnothing$. Then \eqref{eqn:dm} has an optimal solution $(\X_1^{\star}, \X_2^{\star})$ such that $\tmop{rank}(\X_1^{\star}) = \tmop{rank}(\X_2^{\star}) = 1$. Moreover, if $\tmop{rank}(\M - z_1 \D_1 - z_2 \D_2) \geq n-r$ for every $(z_1, z_2)$, all the optimal solutions to  \eqref{eqn:dm} satisfy $\max\{\tmop{rank}(\X_1^{\star}), \tmop{rank}(\X_2^{\star})\}\leq 2r - 1$.
\end{thm}

\prove{
\begin{proof}
The existence of a rank-one solution is guaranteed by \cite{pataki1998rank}, where we know there exist  $\X_1^\star$ and $\X_2^\star$ such that $\tmop{rank} (\X_1^\star \oplus \X_2^\star) \leq \sqrt{2k} = 2$. If $\X_1^\star = \0$, then the optimal value of \eqref{eqn:dm} is 0 and this implies there is no valid preconditioner, contradicting the assumption $\tmop{span}~\{\tmd_1, \tmd_2\} \cap \mathbb{R}^n_{++} \neq \varnothing$. On the other hand, $\X_2^\star \neq \0$ due to the constraint $\langle \M, \X_2 \rangle = 1$. Therefore both $\X_1^\star$ and $\X_2^\star$ must be rank-one.
The second claim follows from the SDP rank theorem \cite{lemon2016low}
\begin{align}
  \tmop{rank} ( \X_1^{\star} \oplus \X_2^{\star} ) \leq{} & 2 n -
  \min_{\tau^{\star}, z_1^{\star}, z_2^{\star}} \tmop{rank} ( ( \M -
  \D_1 z_1^{\star} - \D_1 z_2^{\star} ) \oplus ( \D_1 z_1^{\star} +
  \D_1 z_2^{\star} - \M \tau^{\star} ) ) \nonumber\\
  ={} & 2 n - \min_{z_1^{\star}, z_2^{\star}} \tmop{rank} ( \M - \D_1
  z_1^{\star} - \D_1 z_2^{\star} ) - \min_{z_1^{\star}, z_2^{\star}}
  \tmop{rank} ( \D_1 z_1^{\star} + \D_1 z_2^{\star} - \M \tau^{\star}
  ) \nonumber\\
  \leq{} & 2 n - \min_{z_1, z_2} \tmop{rank} ( \M - \D_1 z_1 - \D_1 z_2
  ) - \min_{z_1, z_2} \tmop{rank} ( \D_1 z_1 + \D_1 z_2 - \M
  ) \label{proof-3} \\
  \leq{} & 2 n - 2 (n - r) =2 r\nonumber
\end{align}

where \eqref{proof-3} uses the fact $\tau^\star > 0$ and that
\[ \tmop{rank} ( \D_1 z_1^{\star} + \D_1 z_2^{\star} - \M \tau^{\star}) = \tmop{rank} ( \D_1 (z_1^{\star} / \tau^{\star}) + \D_1 (z_2^{\star}  / \tau^{\star})- \M).\]
Since $\X_1^\star, \X_2^\star$ cannot be $\0$, this completes the proof.
\end{proof}
}

\textbf{Theorem \ref{thm:5}} suggests that optimal solutions to \eqref{eqn:dm} have low rank, which aligns with the low iteration complexity of \textbf{Algorithm \ref{alg:1}}. 
As an important implication of this low-rank property,
it is cheap to evaluate the dual optimal matrices
 $\X_1^{\star} = \sum_{j = 1}^{|
\mathcal{V}_1 |} x_{1 j}^{\star} \tmv_1^j {\tmv_1}^{j\top}$ and $\X_2^{\star} =
\sum_{j = 1}^{| \mathcal{V}_2 |} x_{2 j}^{\star} \tmv_2^j {\tmv_2}^{j\top}$, since $(\x_1^{\star}, \x_2^{\star}	)$ has a sparse support. 
This observation motivates a series of extensions.

\paragraph{Non-diagonal preconditioners.} 
Formulation \eqref{eqn:odp-subspace} can be extended to non-diagonal preconditioners if we allow $\D_i = \tmP_i$ to be non-diagonal.
\begin{equation}
\label{eqn:nondiag}
\maxf{\z, \tau} ~\tau \quad \text{subject to}  \quad \M \tau - \sum_{i = 1}^k \tmP_i z_i \preceq \0, \quad \sum_{i = 1}^k \tmP_i z_i - \M\preceq \0,
\end{equation}
This formulation seeks the best combination of a set of arbitrary matrices $\{\tmP_i\}$. If $\tmop{span}\{\tmP_i\} \cap \mathbb{S}^n_+$ is nonempty, all our previous results hold. Assuming $k = 2$ and we have access to $(\X_1^\star, \X_2^\star)$ as the sum of $\mathcal{O}(1)$ vector outer products, we can design the following pricing problems and their corresponding preconditioners.
\begin{itemize}[leftmargin=*,itemsep=0pt]
	\item \textbf{Preconditioners of arbitrary sparsity pattern}. \\Given specified sparsity pattern (for example, tri-diagonal matrices), solve the pricing problem
	\[
\maxf{\|\tmP\| = 1}~ ~
  \langle \tmP,  \X_1^{\star} - \X_2^{\star}
  \rangle \]
  \item \textbf{Low-rank plus diagonal preconditioner}.\\
Assume $\tmP_1$ is diagonal. Given a specified rank $r$, solve the pricing problem
  \[
\maxf{\| \tma_i \| = 1}~~
  \langle \textstyle \sum_{i=1}^r \tma_i \tma_i^\top,  ( \X_1^{\star} - \X_2^{\star} )
  \rangle, \]
 and we solve \eqref{eqn:nondiag} to find the best linear combination of $\tmP_1$ and  $\sum_{i=1}^r \tma_i \tma_i^\top$
\end{itemize}

\paragraph{Design of sparse approximate inverse preconditioners}
We conclude by elaborating on the comment at the beginning of the section that the \emph{extremal} eigenvectors contain information that improves the condition number. We focus on the simplest case where $\D=\I$. Assume that two extremal eigenvalues of $\M$ have multiplicity 1 with eigenvectors $\tmv_{\max}$ and $\tmv_{\min}$. Taking $k = 1$ and $\mathcal{S}_1 = \{\1\}$, then \eqref{eqn:odp-subspace} computes $\kappa(\M)$:
\begin{equation*}
 \maxf{z, \tau} ~\tau \quad \text{subject to}  \quad \M \tau -  \I z \preceq \0, \quad \I z \preceq \M.
\end{equation*}
At the optimal solution $(z^\star,\tau^\star)$,  $\tau^\star=1/\kappa(\M)$ and $z=\lambda_{\min}(\M)$.
By \textbf{Theorem \ref{thm:5}} we know that $\tmop{rank}(\X_1^{\star}) = \tmop{rank}(\X_2^{\star}) = 1$. Moreover, by complementary slackness, we conclude $\X_1^\star =  \tmv_{\max}\tmv_{\max}^\top, \X_2^\star = \tmv_{\min} \tmv_{\min}^\top$ respectively. Therefore, the magnitude of each coordinate of $(\tmv_{\min} \tmv_{\min}^\top - \tmv_{\max}\tmv_{\max}^\top)$ can be viewed as a score function that measures \textit{which sparsity pattern decreases the condition number most}. This motivates a new heuristic to find the sparsity pattern for sparse approximate inverse preconditioners.

%% file: sec_exp.tex
\section{Numerical experiments} \label{sec:exp}

In this section, we conduct numerical experiments to illustrate the practical
efficiency of the proposed method. Our experiment is divided into three parts. In
the first part, we focus on the theoretical properties of our proposed methods and verify our theoretical findings on simulated data. We test our method on real-life and large-scale matrix instances in the second part. Lastly, we test some extensions of our method.

\subsection{Experiment setup} \label{sec:setup}
\paragraph{Synthetic data.} For synthetic data, given sparsity parameter
$\sigma$ and regularization parameter $\alpha > 0$, we generate sparse random
matrix $\A \in \mathbb{R}^{n \times n}$ and take $\M = \A^{\top} \A + \alpha
\cdummy \I$. Each nonzero element of $\A$ is sampled from \tmtextbf{1)}
uniform distribution $\mathcal{U} [0, 1]$ and \tmtextbf{2)} standard normal
distribution $\mathcal{N} (0, 1)$.

\paragraph{Real matrices.} For real matrices, we take matrices from \texttt{SuiteSparse}
matrix collection {\cite{kolodziej2019suitesparse}}.
\begin{enumerate}[label=\textbf{\arabic*)}, ref=Ex\arabic*, leftmargin=*]
  \item \tmtextbf{Dataset generation}. To verify the
  marginal decrease behavior of the randomized approach in \textbf{Section \ref{sec:subspace}}, we choose $n = 30, \sigma \in \{
  0.3, 1.0 \}, \alpha = 10^{- 5}$. In the second
  part, we choose $n \in \{ 10^2, 10^4, 10^6, 10^7 \}$ with corresponding sparsity $\sigma \in \{ 10^{-
  1}, 10^{- 3}, 10^{- 6}, 5\times10^{-8} \}$, $\alpha = 10^{- 3}$.
  
  \item \tmtextbf{SDP/LP solver}. We employ \texttt{COPT} \cite{ge2022cardinal} and \texttt{HDSDP} \cite{gao2022hdsdp} to solve SDP
  problems and when verifying the marginal decrease behavior of randomized
  preconditioners. \texttt{COPT} or \texttt{Gurobi} \cite{gurobi2021gurobi} is applied to solve the LP subproblems
  constructed in the cutting plane approach in \textbf{Section \ref{sec:semiinf}}.
  
  \item \tmtextbf{Cutting plane configuration}. We initialize $\mathcal{V}_1,
  \mathcal{V}_2$ in the cutting plane method with 20 initial vectors generated
  randomly from normal distributions. If the problem is unbounded, we will continue
  adding random vectors until the LP becomes bounded. We apply the Lanczos
  method for extremal eigenvalue computation as the separation oracle. The relative
  tolerance of oracle is set to $10^{- 14}$.
  
  \item \tmtextbf{Subspace generation}. For the experiment of randomized subspace, we generate diagonal matrices from $\mathcal{U} [-0.5, 0.5]^n$.
  
  \item \tmtextbf{Experiment environment}. Experiments run on
  \texttt{Mac Mini} of 16\texttt{GB} memory with \texttt{M1 Apple
  Silicon}
\end{enumerate}

\subsection{Verification of theory}

We start by verifying our theoretical findings. We
focus on \tmtextbf{1)} \tmtextbf{Theorem \ref{thm:2}}, the marginal decrease
behavior of randomized preconditioners. \tmtextbf{2)} \tmtextbf{Theorem \ref{thm:3}}, the convergence behavior of the cutting plane approach. \\\

The first experiment's results are summarized in \tmtextbf{Figure \ref{fig:2}}. It verifies that the
effect of adding a new random preconditioner marginally decreases. Indeed, we
see a clear sublinear trend in $\kappa_k$ as $k$ increases. Therefore, adopting a moderate $k$ to achieve satisfactory improvements in condition number is reasonable.   \\
\begin{figure}[h]
\centering
\includegraphics[scale=0.78]{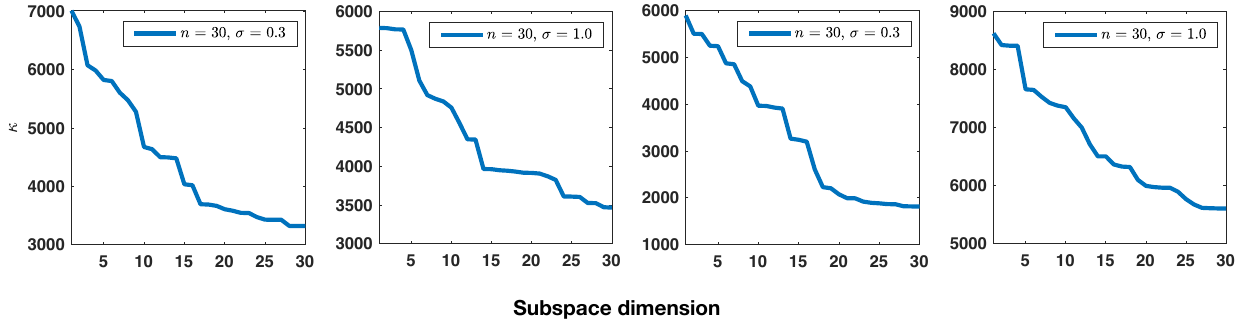}
  \caption{Condition number and randomized subspace dimension $k$. First two columns: uniform distribution; Last two columns: normal distribution;
 x-axis is subspace dimension $k$; y-axis is condition number. \label{fig:2}}
\end{figure}
\begin{figure}[h]
\centering
\includegraphics[scale=0.65]{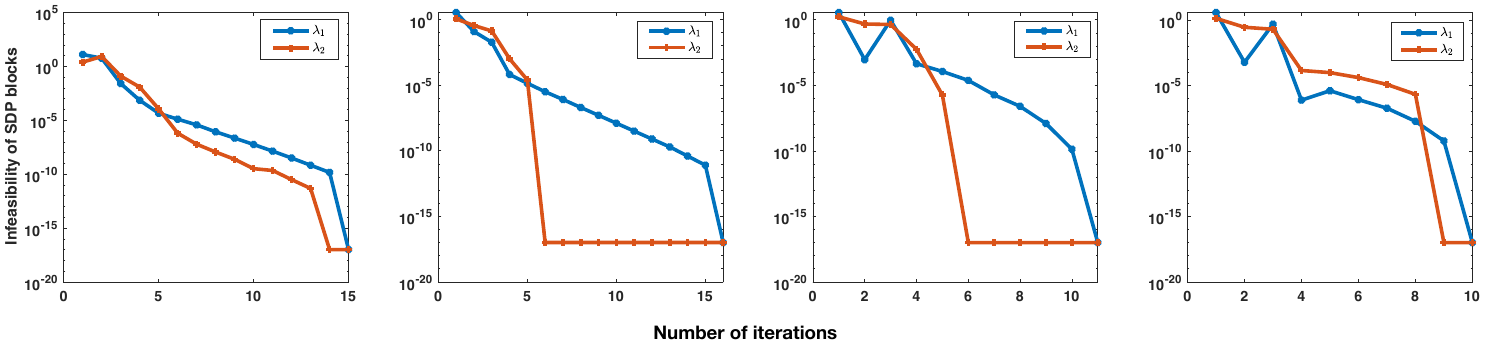}
\caption{Convergence behavior of the SIP approach to solving \eqref{eqn:odp-subspace}. From left to right: matrices of $(n, \sigma) \in \{(10^2, 10^{-1}), (10^4, 10^{-4}), (10^6, 10^{-6}), (10^7, 5\times10^{-8})\}$. x-axis: number of cutting plane iterations; y-axis, the infeasibility of the two SDP blocks, measured by the minimum eigenvalue of two SDP blocks.}
\label{fig:3}
\end{figure}

\textbf{Figure \ref{fig:3}} demonstrates the effectiveness of the cutting plane method in solving the SDP \eqref{eqn:odp-subspace}, as evidenced by the violation of the two SDP constraints. Notably, using the Lanczos method as a separation oracle proves to be practically efficient, requiring $\mathcal{O} ( \tmop{nnz} ( \M ) )$ operations. Additionally, for a moderate value of $k$, the cutting plane algorithm typically converges within 20 iterations. The second block converges more rapidly than the first, which is expected considering the additional variable $\tau$ involved in the first block.

\paragraph{Scalability.} Our next experiment evaluates the scalability of our approach to large-scale sparse matrices. According to \textbf{Table \ref{tab:1}}, our approach can deal with matrices of size up to $10^7$ in \texttt{Matlab}.

\begin{table}[h]
\centering
\caption{Time statistics for matrices in \textbf{Figure \ref{fig:3}} \label{tab:1}}
\begin{tabular}{cccc}
\toprule
$n$ & $\sigma$ & nnz & CPU Time \\
\midrule 
$10^3$ & $10^{-1}$ & 6000 & 0.2\\
$10^4$ & $10^{-4}$ & $10^6$ & 0.6 \\
$10^6$ & $10^{-6}$ & $2\times10^6$ & 15.3 \\
$10^7$ & $5\times10^{-8}$ & $1.3\times10^7$ & 186.2 \\
\bottomrule	
\end{tabular}
\end{table}

\subsection{Practical matrices}

In this section, we evaluate the performance of our method on the \texttt{SuiteSparse} matrix collection and matrices from practical applications.

\paragraph{SuiteSparse collection.}
We consider matrices from the \texttt{SuiteSparse} matrix collection. 
We tested our method on 1000 matrices of size $n \leq 5000$, as it is too expensive to compute the condition number for larger matrices explicitly.  For each matrix, we compare the following preconditioners:

\begin{itemize}[leftmargin=15pt]
	\item \emph{Original.} $\D = \mathbf{I}$.
	\item \emph{Jacobi.} $\D = \diagm (\diag(\M))$.
	\item \emph{Subspace.} $\D$ is the optimal diagonal preconditioner in subspace $\mathcal{S}_2 = \tmop{span} \{\1, \diag(\M) \}$.
	\item \emph{Iterative.} Run \textbf{Algorithm \ref{alg:2}} for 5 iterations and use the resulting preconditioner.
\end{itemize}

 If the cutting plane method fails to compute either the subspace preconditioner or the iterative preconditioner, the matrix result is discarded. 
Condition numbers are evaluated using \texttt{Matlab}'s \texttt{cond} function. We ended up getting 914 matrices, and 86 of the matrices tested failed due to huge condition numbers, and \texttt{Gurobi} fails to solve the linear program when $\tau$ is too small.\\

\textbf{Figure \ref{fig:distribution}} presents the distribution of these condition numbers on a logarithmic scale.  \textbf{Figure \ref{fig:improvement}} presents the same data to highlight the improvements in condition numbers. On average, both the subspace and iterative preconditioners achieve around 2-fold improvement in the condition number. Specifically, the subspace method, on average, reduces the condition number by a factor of 1.9, while the iterative preconditioner reduces by a factor of 2.1.
Notably, this improvement represents a 2.0-fold improvement over the Jacobi preconditioner, which suggests that the Jacobi preconditioner, on average, performs worse on the tested benchmark. There are more than 100 matrices where the iterative preconditioner outperforms the subspace preconditioner by a factor of 1.4.
\begin{figure}[h]
\centering
\includegraphics[scale=0.33]{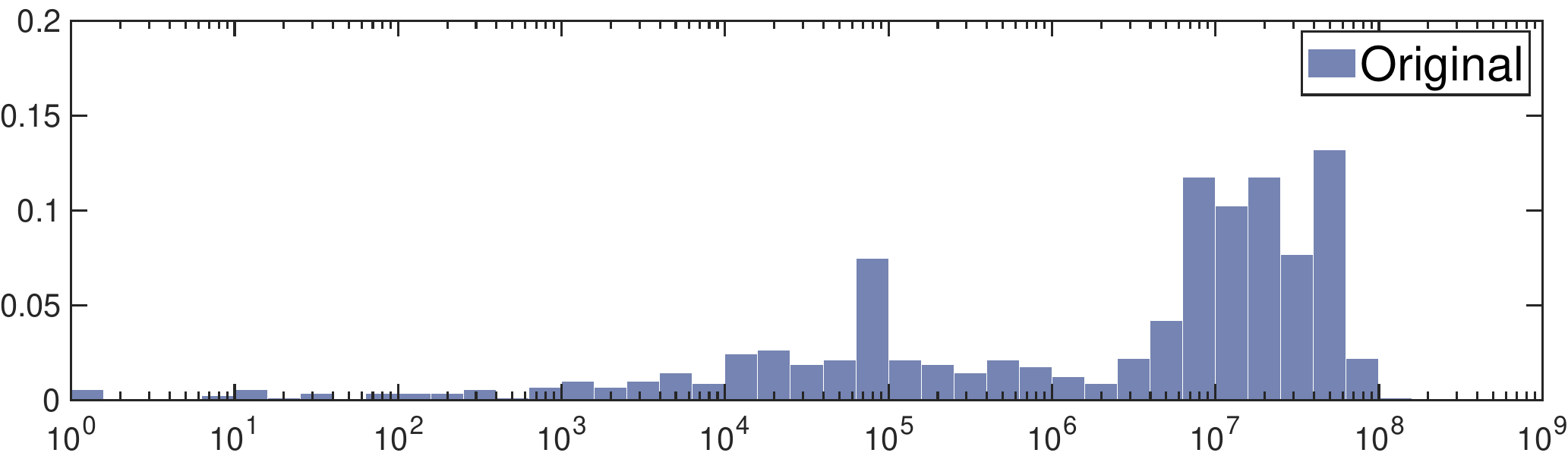} \\
\includegraphics[scale=0.33]{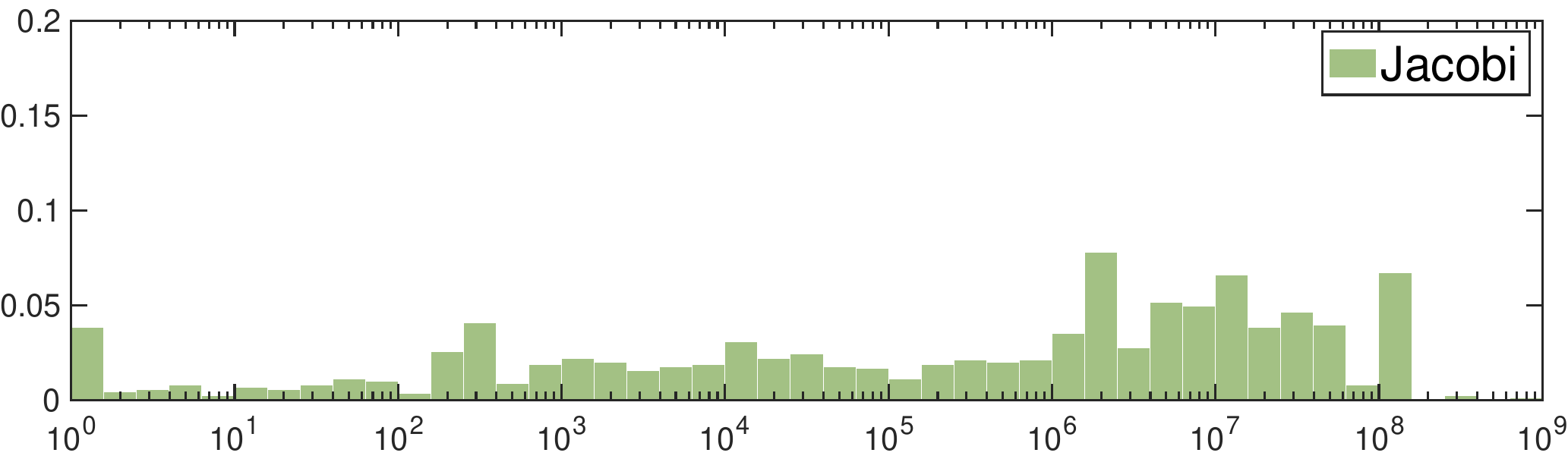}\\
\includegraphics[scale=0.33]{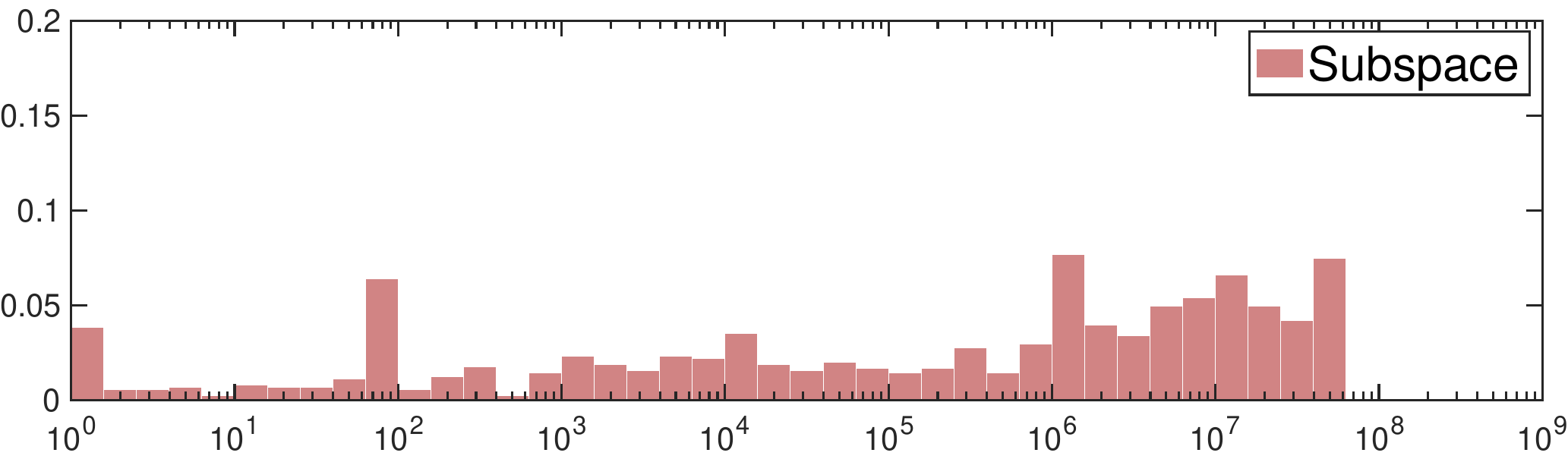}\\
\includegraphics[scale=0.33]{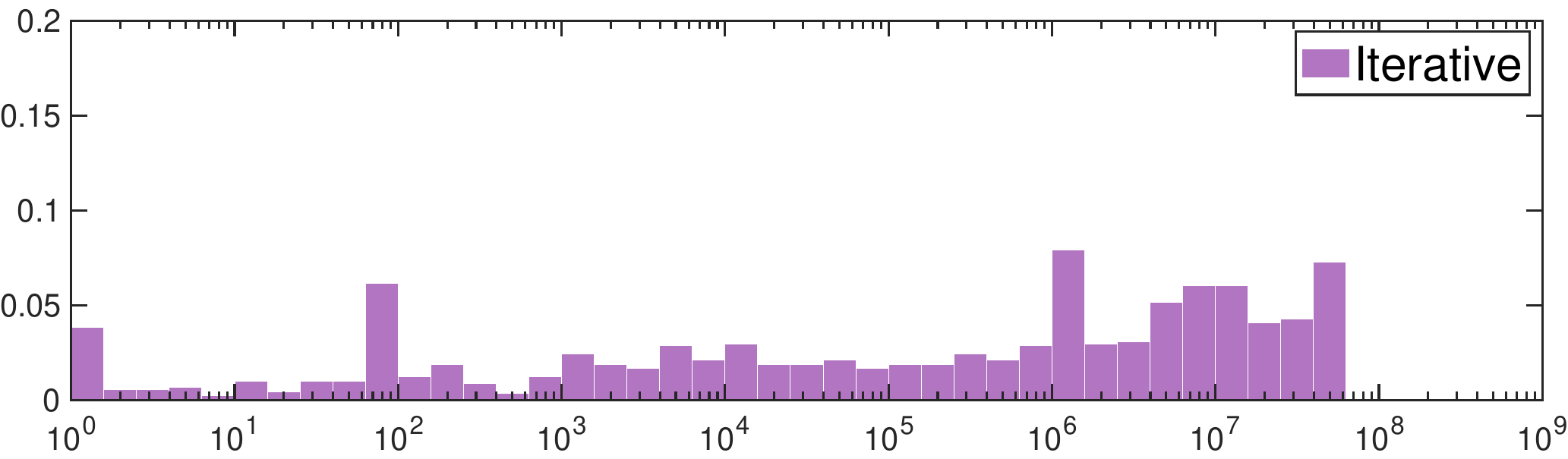}
\caption{Distribution of condition number for different preconditioners. \label{fig:distribution}}
\end{figure}

\begin{figure}[h]
\centering
\includegraphics[scale=0.22]{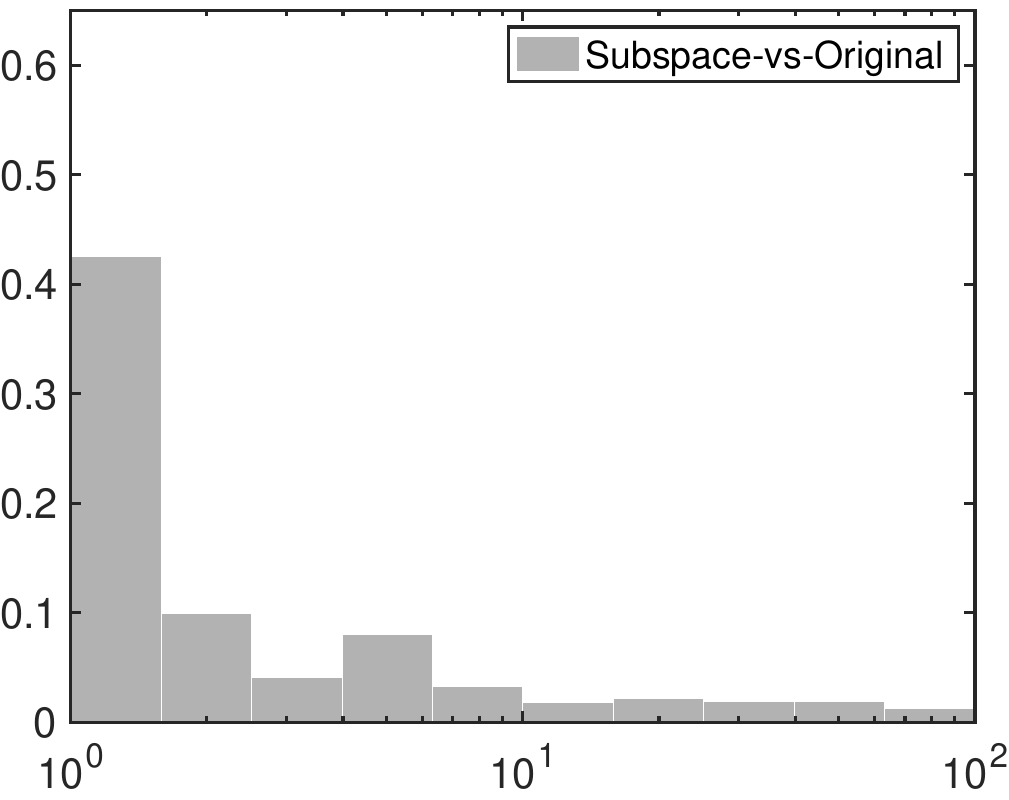}
\includegraphics[scale=0.22]{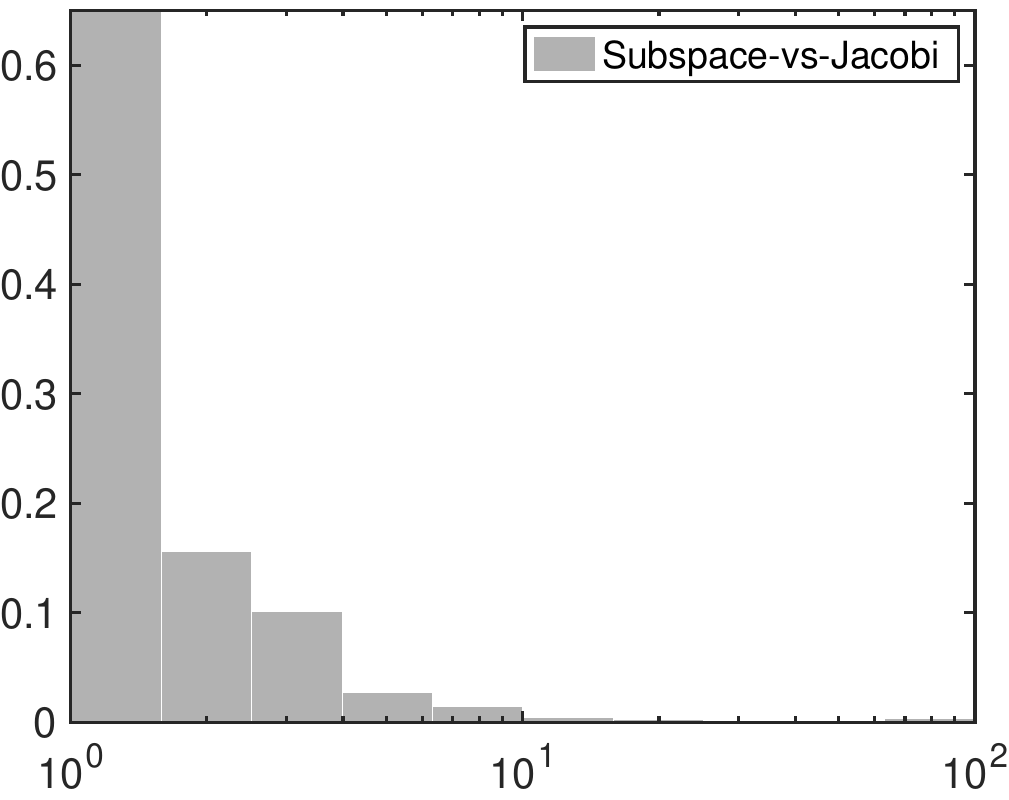}
\includegraphics[scale=0.22]{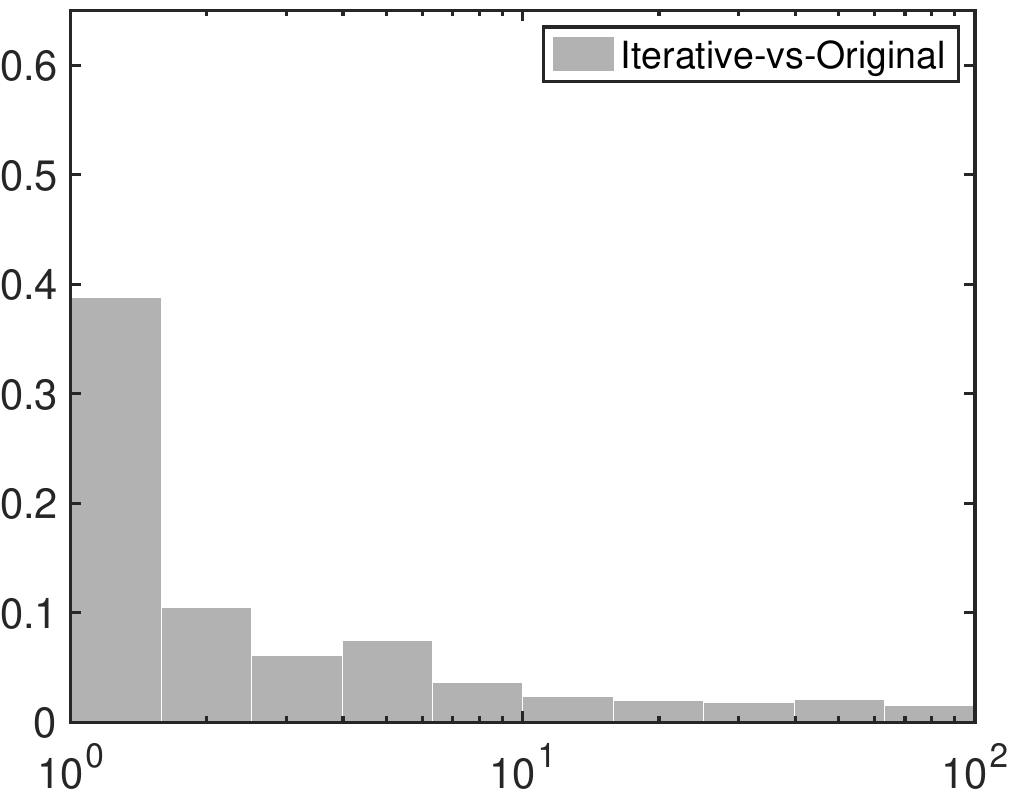}
\includegraphics[scale=0.22]{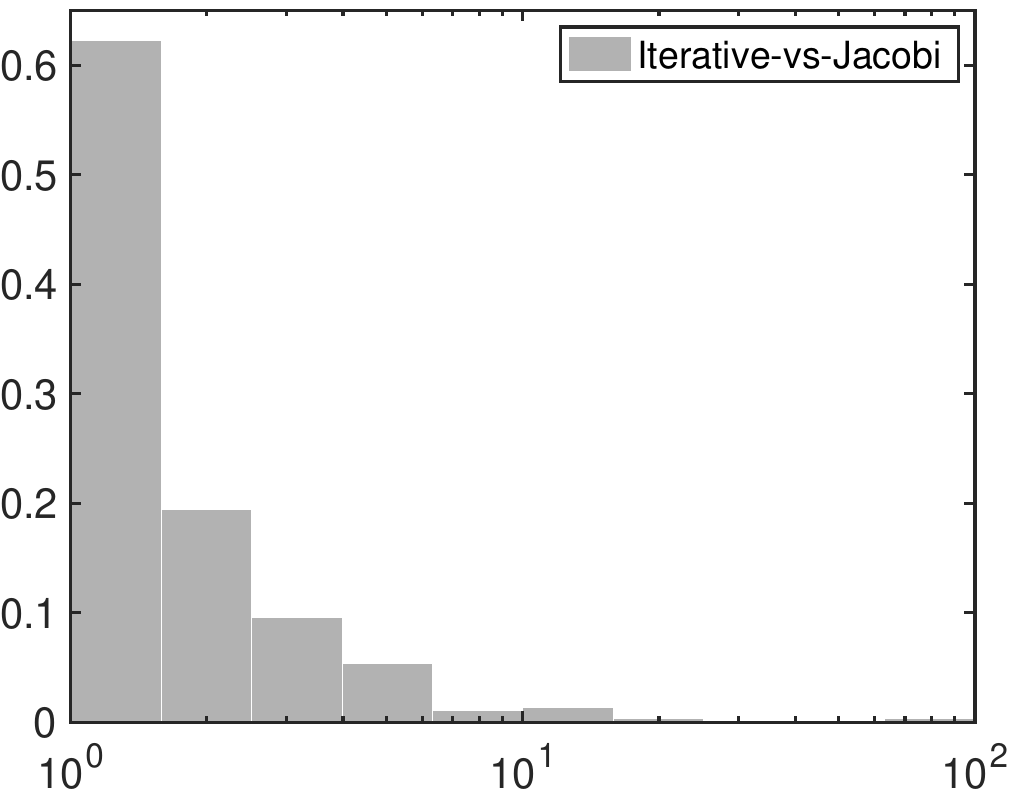}
\caption{Distribution of the ratio between condition numbers. \label{fig:improvement}
Both the subspace and iterative methods generally improve over Jacobi by a factor between 1 and 100.}
\end{figure}

\paragraph{Large ill-conditioned Hessian matrices.}
In this part of the experiment, we test our approach on large-scale linear systems arising in optimization problems, including
\begin{itemize}[leftmargin=30pt]
	\item Hessian matrices from regularized logistic regressions \cite{menard2002applied}, and
	\item normal equation from interior point methods \cite{ye2011interior}.
\end{itemize}
Both problems produce ill-conditioned matrices $\M = \nabla^2 f(\x)= \A \mathbf{B} \A^\top$ with fixed data matrix $\A$ and diagonal matrix $\mathbf{B}$ determined by variable $\x$. We test the effectiveness of our preconditioners using preconditioned conjugate gradient method. \\

\textbf{Table \ref{table:2}} shows the number of matrix-vector multiplications needed to solve the Newton system from logistic regression to $10^{-10}$. As is demonstrated in \textbf{Table \ref{table:2}}, the preconditioner developed by our algorithm reduces the total number of matrix-vector multiplications. This underscores the efficacy of our preconditioner.\\

Even more interestingly, \textbf{Figure \ref{fig:ipm}} plots the residual of PCG as the number of matrix-vector products increases. The matrix, taken from the normal equation of the interior point method close to convergence, is very ill-conditioned. Each line in the figure represents the convergence of PCG after an iterative preconditioning step from \textbf{Algorithm \ref{alg:2}}. For a fair comparison, we also count the number of matrix-vector multiplications required in the Lanczos procedure and delay the plot of PCG convergence by the same number of multiplications. Even if we perform a noticeable number of matrix-vector multiplications reducing the condition number, the total number of matrix-vector products will still be reduced.

\begin{table}[h]
\caption{Number of matrix-vector multiplications required to solve the logistic regression problem. \label{table:2}}
\centering
\begin{tabular}{ccccc}
\toprule
 $n$& 
 Iterative & Jacobi & Original\\
\midrule
 $10^4$& 2536 & 3202 & 2804\\
 $5\times10^5$& 2422 & 3004 & 2710\\
 $5\times10^5$& 2454 & 3054 & 2798\\
 $10^6$  & 1048 & 1246 & 1222\\
 $10^6$& 1106 & 1294 & 1266\\
\bottomrule
\end{tabular}	
\end{table}

\begin{figure}[h]
\centering
\includegraphics[scale=0.35]{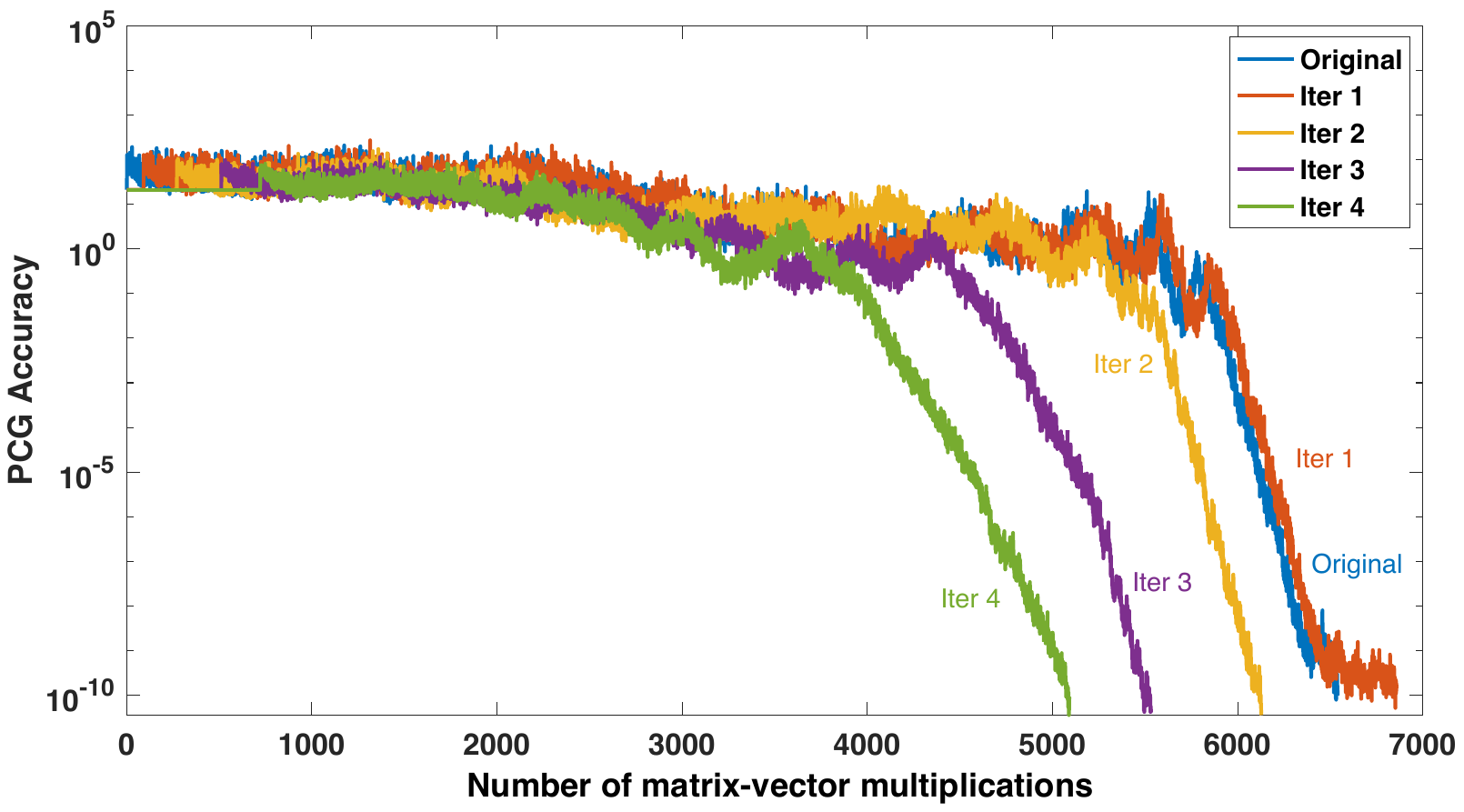}
\caption{Convergence of PCG on the normal equation from the interior point method. x-axis: number of matrix-vector multiplications including the cost of computing the preconditioner. y-axis: PCG accuracy. \label{fig:ipm} The normal equation is taken from \texttt{Netlib} instance \texttt{adlittle}.}
\end{figure}

\subsection{Extensions}
We conclude this section with some preliminary tests of our proposed extensions. These results indicate that our methodology could be independently interesting for other numerical analysis problems. However, it's important to note that these results are only applicable to \textit{a subset of} the matrices we tested.

\paragraph{Condition number estimation.}

Condition number estimation is important in many applications \cite{urschel2021uniform, ferng1991adaptive,cline1979estimate}. Computing the condition number is expensive, and estimating the condition number accurately is hard. Following \textbf{Remark \ref{rem:1}}, it is possible to apply the cutting plane method to \eqref{eqn:odp-subspace} with $\mathcal{S}_1 = \{\1\}$ to evaluate $\kappa(\M)$. We compare the following estimators:
\begin{itemize}[leftmargin=*]
	\item \emph{Condition number estimation using Lanczos extremal eigenvalue computation.} \\
We call \texttt{Matlab eigs} routine with maximum iteration 1200 for both $\hat{\lambda}_{\max}(\M)$ and $\hat{\lambda}_{\min}(\M)$. Condition number is then estimated by $\hat{\kappa}(\M) = \tfrac{\hat{\lambda}_{\max}(\M)}{\hat{\lambda}_{\min}(\M)}$.
	\item \emph{Condition number estimation using} \textbf{Algorithm \ref{alg:1}}.\\
 We solve \eqref{eqn:odp-subspace} and estimate $\hat{\kappa}(\M) = 1/\tau^\star$. Each separation oracle is limited to running 200 iterations, so the total number of matrix-vector multiplications is comparable.
\end{itemize}

We take the \texttt{bcsstk27} matrix from \texttt{SuiteSparse} and randomly subsample $n$ of its columns. For each $n \in \{100, 150, \ldots,400\}$, we record the relative error in condition number estimation as
$e =\frac{\kappa(\M) - \hat{\kappa}(\M)}{\kappa(\M)}$,

where $\kappa(\M)$ is computed using Matlab's \verb|cond| function.
We report results averaged over 10 matrices from the collection. \textbf{Figure \ref{fig:5}} (left) illustrates the comparison between the two methods. The cutting plane method achieves higher accuracy than the Lanczos estimator when the matrix size becomes larger.
\begin{figure} 
\centering
\includegraphics[scale=0.26]{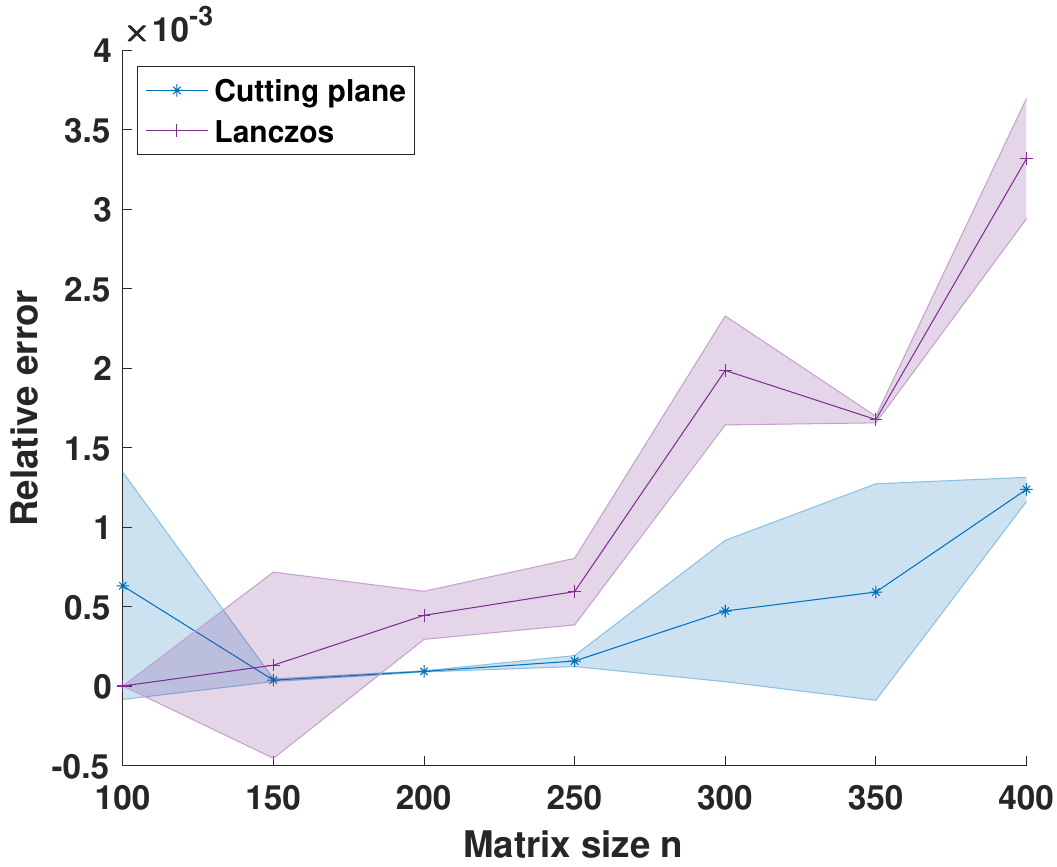} \quad\qquad
\includegraphics[scale=0.26]{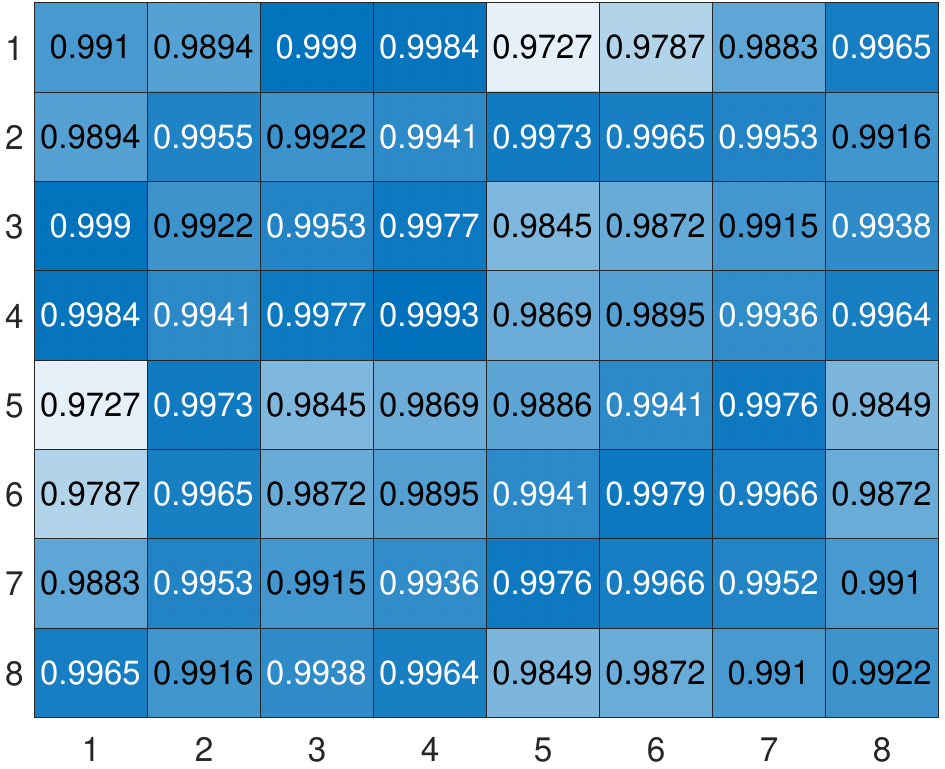}
\includegraphics[scale=0.26]{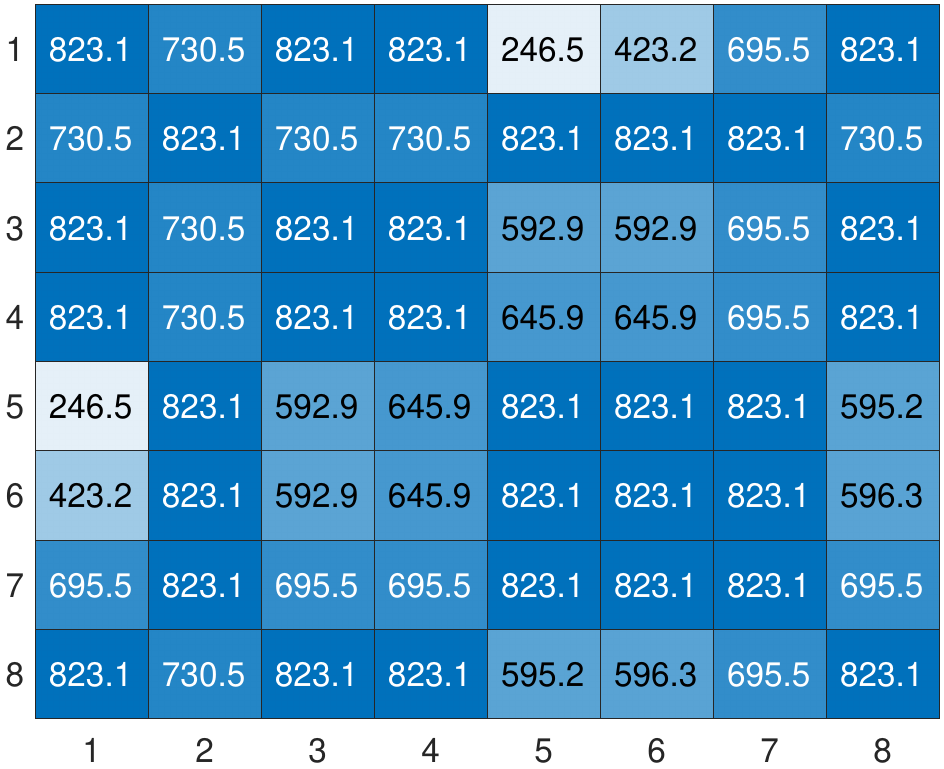}
\caption{Left: relative error in condition number estimation. Middle: absolute value of $\tmv_{\max} \tmv_{\max}^\top - \tmv_{\min} \tmv_{\min}^\top$ (elementwise). 
Right: optimal condition number of ``diagonal plus one off-diagonal'' sparsity pattern.
\label{fig:5}}
\end{figure}

\paragraph{Sparsity pattern for preconditioners.} This paper has focused on diagonal preconditioners. However, allowing sparse (and not simply diagonal) preconditioners can substantially reduce the condition number. 
The following experiment demonstrates that $|\X_1^\star - \X_2^\star| = |\tmv_{\max}\tmv_{\max}^\top - \tmv_{\min}\tmv_{\min}^\top|$, where $|\cdot|$ takes the element-wise absolute value of a matrix, 
can guide sparsity pattern selection. To illustrate this potential, we conduct the following experiment:

\begin{enumerate}[leftmargin=*, label=\textbf{Step \arabic*.}, ref=\textbf{Step \arabic*}]
	\item Generate an $8 \times 8$ positive definite matrix.
	\item Compute $|\tmv_{\max}\tmv_{\max}^\top - \tmv_{\min}\tmv_{\min}^\top|$.
	\item Compute the optimal preconditioner with ``diagonal plus one off-diagonal element''. For each off-diagonal element in the $8 \times 8$ matrix, we solve an optimal preconditioning problem to find the optimal preconditioned condition number with given sparsity pattern. \label{step3} 
	\item We display the results as $8 \times 8$ heat maps: the first records the magnitude of the elements of $|\tmv_{\max}\tmv_{\max}^\top - \tmv_{\min}\tmv_{\min}^\top|$, and the other figure records the condition numbers of the optimally preconditioned matrices from \ref{step3}.
\end{enumerate}

\textbf{Figure \ref{fig:5}} (right) suggests a strong correlation between $|\tmv_{\max}\tmv_{\max}^\top - \tmv_{\min}\tmv_{\min}^\top|$ and reductions in condition number. Interestingly, even with \textit{only one extra} off-diagonal element, the condition number can be reduced by a factor of 3
(from $823.1$ to $246.5$). 
We also see that the choice of the off-diagonal element is important.

%% file: conclusion.tex
\section{Conclusion}

We propose efficient algorithms to scalably find an approximate optimal diagonal preconditioner for a positive definite matrix. Our result builds on an SDP formulation of the optimal diagonal preconditioning problem, and we employ various techniques to enhance its scalability. Our result not only provides an efficient way for condition number optimization, but also yields results that may be of independent interest in numerical analysis.

%% file: app_proof.tex
\section{Auxiliary results}

\begin{lem}[Residual of randomized projection \cite{derezinski2020precise}]\label{lem:aux2}
  
  Let $\bar{\tmP} = \I - \mathcal{D}^{\top} ( \mathcal{D} \mathcal{D}^{\top} ) \mathcal{D}, \mathcal{D} =
  \Z \Sigma \in \mathbb{R}^{k \times n}$, where ${\Z}$ has i.i.d. sub-gaussian
  rows with zero mean and bounded covariance matrix. $\Sigma$ is an $n$ by $n$
  positive semi-definte matrix. Then for $k < n$, there exists universal
  constant $L$ such that
  \[ ( 1 - \tfrac{L}{\sqrt{r}} ) \bar{\tmP}_{\gamma} \preceq
     \mathbb{E} [ \bar{\tmP} ] \preceq ( 1 +
     \tfrac{L}{\sqrt{r}} ) \bar{\tmP}_{\gamma}, \]
  where $\tmP_{\gamma} \assign ( \gamma \Sigma + \I )^{- 1}$ and
  $\gamma$ is chosen such that $\tmop{tr} ( \tmP_{\gamma}) = n - k$; $r = \frac{\tmop{tr} (\Sigma)}{\| \Sigma \|}$ is the stable rank
  of $\Sigma$.
\end{lem}
\begin{lem}[Random projection]
  Let $\D \in \mathbb{R}^{n \times k}$ be a random matrix sampled from
  subgaussian distribution element-wise. Then any $\varepsilon \in (0,
  1)$, we have\label{lem:aux3}
  \[ \mathbb{P} \{ (1 - \varepsilon) \| \x \|^2 \leq \|
     \D \x \|^2 \leq (1 + \varepsilon) \| \x \|^2 \}
     \geq 1 - c \exp (- C \varepsilon^2 k) \]
  for some universal constants $c, C > 0$.
\end{lem}

\section{Proof of results in Section \ref{sec:subspace}}

In the proof of this section, we take $\D_0 = \I$ and associate it with an extra variable $z_0$ in \eqref{eqn:odp-subspace}, resulting in the augmented problem
\begin{equation*}
\label{eqn:odp-subspace-augmented}
\maxf{z_0, \z, \tau}  ~\tau^\star \quad \text{subject to}  \quad \M \tau^\star- \sum_{i = 1}^k \D_i z_i + \D_0 z_0 \preceq \0, \quad \sum_{i = 1}^k \D_i z_i + \D_0 z_0 \preceq \M.
\end{equation*}
And we define $\mathcal{D} := (\tmd_1, \tmd_2, \ldots, \tmd_k) \in \mathbb{R}^{n \times k}$ so that $\diag(\sum_{i = 1}^k \D_i z_i) = \sum_{i = 1}^k \tmd_i z_i = \mathcal{D} \z$.

\subsection{Proof of Lemma \ref{lem:2}}

We first prove the following lemma, which is a direct consequence of
\tmtextbf{Lemma \ref{lem:aux2}}.

\begin{lem}
  
  Let $\mathcal{D} \in \mathbb{R}^{n \times k}$ be a random matrix with i.i.d. subgaussian
  columns. Define the random projection matrix together with its residual projection
  \[ \tmP_{\mathcal{D}} \assign \mathcal{D} ( \mathcal{D}^{\top} \mathcal{D} )^{- 1} \mathcal{D}^{\top} \qquad
     \bar{\tmP}_{\mathcal{D}} \assign \I - \tmP_{\mathcal{D}} . \]
  Then there exists a universal constant $L > 0$ such that
  \[ ( 1 - \tfrac{L}{\sqrt{n}} ) ( 1 - \tfrac{k}{n} ) \I
     \preceq \mathbb{E} [ \bar{\tmP}_{\mathcal{D}} ] \preceq ( 1 +
     \tfrac{L}{\sqrt{n}} ) ( 1 - \tfrac{k}{n} ) . \]
\end{lem}

\begin{proof}
According to \tmtextbf{Lemma \ref{lem:aux2}}. We take $\Sigma = \I, r =
\frac{\tmop{tr} ( \I )}{\| \I \|} = n$. Then
\[ \tmop{tr} ( \tmP_{\gamma} ) = \tmop{tr} ( ( \gamma
   \Sigma + \I )^{- 1} ) = \tfrac{\tmop{tr} ( \I
   )}{\gamma + 1} = \tfrac{n}{\gamma + 1} = n - k. \]
Solving for $\gamma$ gives $\gamma = \frac{n}{n - k} - 1$ and plugging it
back, we get $\tmP_{\gamma} = ( 1 - \frac{k}{n} ) \I$ and this
completes the proof.
\end{proof}

Our analysis is based on the following intuition: we can construct a feasible
dual solution $( \tau', \z', z_0' )$ such that $\z' = \tmP_{\mathcal{D}}
\tmd^{\star}$ is the optimal solution to the following least squares problem
\[ \z' = \arg \min_{\z}  \{ \| \mathcal{D} \z - \tmd^{\star} \|^2 \} \]
When $k$ increases, $\| \tmd^{\star} - \z' \| = \|
\bar{\tmP}_{\mathcal{D}} \tmd^{\star} \|$ decreases at a rate of $\sqrt{1 - k / n}$ as
according to \tmtextbf{Lemma \ref{lem:aux3}}. Moreover, by properly choosing $z_0'$
and $\tau'$, we can find a dual feasible solution, thus getting guarantee
for optimality. Using \tmtextbf{{Lemma \ref{lem:aux3}}}, we have
\[ \mathbb{E} [ \| \bar{\tmP}_{\mathcal{D}} \tmd^{\star} \|^2
   ] =\mathbb{E} [ \langle \tmd^{\star}, \bar{\tmP}_{\mathcal{D}}
   \tmd^{\star} \rangle ] = \langle \tmd^{\star}, \mathbb{E}
   [ \bar{\tmP}_{\mathcal{D}} ] \tmd^{\star} \rangle \leq \|
   \tmd^{\star} \|^2 \big( 1 + \tfrac{L}{\sqrt{n}} \big) \big( 1 -
   \tfrac{k}{n} \big) =: e_k \]
where the first equality uses the property of projection. Next by Markov's
inequality, for $\alpha > 1$, we have
\[ \mathbb{P} \{ \| \bar{\tmP}_{\mathcal{D}} \tmd^{\star} \| \geq
   \alpha \sqrt{e_k} \} \leq \mathbb{P} \{ \|
   \bar{\tmP}_{\mathcal{D}} \tmd^{\star} \| \geq \alpha \mathbb{E} [
   \| \bar{\tmP}_{\mathcal{D}} \tmd^{\star} \| ] \} \leq
   \frac{1}{\alpha} . \]
Conditioned on the event $\| \bar{\tmP}_{\mathcal{D}} \tmd^{\star} \|
\leq \alpha \sqrt{e_k}$, letting $\z' = \tmP_{\mathcal{D}} \tmd^{\star}$, we have
\[ \sum_{i = 1}^k \D_i z_i' = \D^{\star} + \Big( \sum_{i = 1}^k \D_i z_i' -
   \D^{\star} \Big) \preceq \D^{\star} + \alpha \sqrt{e_k} \cdummy \I, \]
where we use the fact that
\[\textstyle \| \sum_{i = 1}^k \D_i z_i' - \D^{\star}\| = \| \mathcal{D} \z' -
   \tmd^{\star} \|_{\infty} \leq \| \mathcal{D} \z' - \tmd^{\star} \| =
   \| \bar{\tmP}_{\mathcal{D}} \tmd^{\star} \| \leq \alpha \sqrt{e_k}.\]
Now we take $z_0' = - \alpha \sqrt{e_k}$ and
\[ \sum_{i = 1}^k \D_i z_i' + \I \cdummy z_0' = \D^{\star} + \Big( \sum_{i =
   1}^k \D_i z_i' - \D^{\star} \Big) - \alpha \sqrt{e_k} \cdot \I \preceq
   \D^{\star} \preceq \M . \]
On the other hand, we take $\tau' = \tau^{\star} - \frac{2 \alpha
\sqrt{e_k}}{\lambda_{\min} ( \M )}$ and deduce that
\[ \sum_{i = 1}^k \D_i z_i' + \I \cdummy z_0' = \D^{\star} + \Big( \sum_{i =
   1}^k \D_i z_i' - \D^{\star} \Big) - \alpha \sqrt{e_k} \cdummy \I \succeq
   \Big( \tau^{\star} - \frac{2 \alpha \sqrt{e_k}}{\lambda_{\min} ( \M
   )} \Big) \M . \]
Now that $( \tau', \z', z_0 )$ is a feasible solution, we know that
by optimality condition $\tau_k \geq \tau'$ and
\[ \mathbb{P} \{ \tau_k \geq \tau^{\star} - \frac{20
   \sqrt{e_k}}{\lambda_{\min} ( \M )} \} \geq 1 -
   \frac{1}{\alpha} . \]
Finally we set $\alpha = 20$ to complete the proof.

\

\subsection{Proof of Lemma \ref{lem:3}}
Our proof is inspired by \cite{liberti2021random}. Define $\z' = \mathcal{D}^{\top} \tmd^{\star}$, and we show that with high
probability,
\begin{align}
  \M \tau^{\star} - \sum_{i = 1}^k \D_i z_i' \preceq{} & \varepsilon \|
  \tmd^{\star} \|, \quad
  \sum_{i = 1}^k \D_i z_i' \preceq{}  \M + \varepsilon \| \tmd^{\star}
  \| . \nonumber
\end{align}

Define $\bs_\mathcal{D}^1 \assign \M \tau^{\star} - \sum_{i = 1}^k \D_i z_i'$ and $\bs^1
= \M \tau^{\star} - \sum_{i = 1}^n \e_i \e_i^{\top} d^{\star}_i$. We
successively deduce that, for any $\x \in \mathbb{R}^n$,
\begin{align}
  \langle \x, \bs_\mathcal{D}^1 \x \rangle ={} & \langle \x, \Big( \M
  \tau^\star- \sum_{i = 1}^k \D_i z_i' \Big) \x \rangle \nonumber\\
  ={} & \langle \x, \M \x \rangle \tau^\star- \sum_{i = 1}^k \langle
  \x, \D_i \x \rangle z_i' \nonumber\\
  ={} & \langle \x, \M \x \rangle \tau^\star- \sum_{i = 1}^k \langle
  \x, \sum_{j = 1}^n \e_j \e_j^{\top} d_{j i} \x \rangle z_i' \label{proof-4}
  \\
  ={} & \langle \x, \M \x \rangle \tau^\star- \sum_{j = 1}^n \langle
  \x, \e_j \e_j^{\top} \x \rangle \sum_{i = 1}^k d_{j i} z_i'
  \nonumber\\
  ={} & \langle \x, \M \x \rangle \tau^\star- \sum_{j = 1}^n \langle
  \x, \e_j \e_j^{\top} \x \rangle [ \mathcal{D} \mathcal{D}^{\top} \tmd^{\star}
  ]_j, \label{proof-5}
\end{align}
where in \eqref{proof-4} $d_{j i}$ denotes the $j$-th diagonal element of $\D_i$; \eqref{proof-5} uses the definition of matrix-vector multiplication and that $\z' = \mathcal{D}^\top \tmd^\star$. Now further assume that $\| \x \| = 1$. We have
\begin{align}
  \langle \x, \bs^1 \x \rangle ={} & \langle \x, ( \bs^1 -
  \bs_\mathcal{D}^1 + \bs_\mathcal{D}^1 ) \x \rangle \nonumber\\
  ={} & \langle \x, \bs_\mathcal{D}^1 \x \rangle + \langle \x, (
  \bs^1 - \bs_\mathcal{D}^1 ) \x \rangle \nonumber\\
  ={} & \langle \x, \bs_\mathcal{D}^1 \x \rangle + \sum_{j = 1}^n \langle
  \x, \e_j \e_j^{\top} \{ [ \mathcal{D} \mathcal{D}^{\top} \tmd^{\star} ]_j -
  d^{\star}_j \} \x \rangle \nonumber\\
  \leq{} & \langle \x, \bs_\mathcal{D}^1 \x \rangle + \| ( \mathcal{D}
  \mathcal{D}^{\top} - \I ) \tmd^{\star} \|_{\infty}, \nonumber
\end{align}
where the last inequality is by
\[ \Big\| \sum_{j = 1}^n \e_j \e_j^{\top} \big\{ [ \mathcal{D} \mathcal{D}^{\top}
   \tmd^{\star} ]_j - d^{\star}_j \big\} \Big\| = \|\text{diagm}
   ( ( \mathcal{D} \mathcal{D}^{\top} - \I ) \tmd^{\star} ) \| =
   \| ( \mathcal{D} \mathcal{D}^{\top} - \I ) \tmd^{\star} \|_{\infty} . \]
Therefore, by the property of random projection, we have, with probability $1 - 3
n c \exp (- C \varepsilon^2 k)$, that $\| ( \mathcal{D} \mathcal{D}^{\top} - \I
) \tmd^{\star} \|_{\infty} \leq \varepsilon \| \tmd^{\star}
\|$, which further implies that $\langle \x, \bs^1 \x \rangle \leq \langle \x, \bs_\mathcal{D}^1 \x
   \rangle + \varepsilon = \langle \x, ( \bs_\mathcal{D}^1 + \varepsilon
   \| \tmd^{\star} \| \I ) \x \rangle$
for all $\x, \| \x \| = 1$. Finally we deduce that
\[ \lambda_{\max} ( \bs^1_\mathcal{D} ) = \max_{\| \x \| = 1} 
   \langle \x, \bs^1_D \x \rangle \leq \max_{\| \x \| =
   1}  \langle \x, ( \bs_\mathcal{D}^1 + \varepsilon \| \tmd^{\star}
   \| \I ) \x \rangle = \lambda_{\max} ( \bs_\mathcal{D} )
   + \varepsilon \| \tmd^{\star} \| \leq \varepsilon \|
   \tmd^{\star} \|, \]
where we use the relation $\bs_\mathcal{D} \preceq \0$ since it is feasible for the
original problem. Overall we have shown that
\[\textstyle \mathbb{P} \{ \M \tau^\star- \sum_{i = 1}^k \D_i z_i' \preceq \varepsilon
   \| \tmd^{\star} \| \} \geq 1 - 3 n c \exp (- C
   \varepsilon^2 k) \]
For the other block, we can similarly define $\bs_\mathcal{D}^2 \assign \M -
\sum_{i = 1}^k \D_i z_i'$ and $\bs^2 = \M - \sum_{i = 1}^n \e_i
\e_i^{\top} d^{\star}_i$ and apply the same argument to show that
\[ \textstyle\mathbb{P} \{ \sum_{i = 1}^k \D_i z_i' - \M \preceq \varepsilon
   \| \tmd^{\star} \| \} \geq 1 - 3 n c \exp (- C
   \varepsilon^2 k) \]
Taking union bound over the two events, we have, with probability $1 - 6 n c
\exp (- C \varepsilon^2 k)$, that
\begin{align}
  \M \tau^\star- \sum_{i = 1}^k \D_i z_i' \preceq{} & \varepsilon \| \tmd^{\star}
  \| \cdummy \I \nonumber\\
  \sum_{i = 1}^k \D_i z_i' \preceq{} & \M + \varepsilon \| \tmd^{\star}
  \| \cdummy \I \nonumber
\end{align}
and this completes the first part of the proof. Next we derive a bound for the
dual objective by showing that a dual feasible solution $( \tau', \z',
z_0' )$ exists. By choosing $z_0' = - \varepsilon \| \tmd^{\star}
\|$, we have
\[ \sum_{i = 1}^k \D_i z_i' + \I \cdot z_0' - \M = \sum_{i = 1}^k \D_i z_i' - \M -
   \varepsilon \| \tmd^{\star} \| \cdummy \I \preceq \varepsilon
   \| \tmd^{\star} \| \cdummy \I - \varepsilon \| \tmd^{\star}
   \| \cdummy \I = \0 . \]
Then we choose $\tau' = \tau^{\star} - \frac{2 \varepsilon \| \tmd^{\star}
\|}{\lambda_{\min} ( \M )}$ and notice that
\begin{align}
	& \M \tau' - \sum_{i = 1}^k \D_i z_i' - z_0 \I \\
	={} & \M \tau^{\star} - \frac{2
   \varepsilon \| \tmd^{\star} \|}{\lambda_{\min} ( \M )}
   \M - \sum_{i = 1}^k \D_i z_i' + \varepsilon \| \tmd^{\star} \|
   \cdummy \I \nonumber \\
   \preceq{} & 2 \varepsilon \| \tmd^{\star} \| \cdummy \I +
   \frac{2 \varepsilon \| \tmd^{\star} \|}{\lambda_{\min} ( \M
   )} \M \preceq \0 \nonumber
\end{align}
Since $\tau_k$ is the optimal value of the problem, we have
$\tau_k \geq \tau' = \tau^{\star} - \frac{2 \varepsilon \| \tmd^{\star}
   \|}{\lambda_{\min} ( \M )}$
and this completes the proof.

\subsection{Proof of Theorem \ref{thm:2}}

Recall that $\tau_k^{\star} \geq \kappa ( \M )^{- 1}$ since $\I$ is
included in the subspace. Then we apply union bound to get, with probability
$0.8$, that
\[ \tau_k \geq \max \{ \tau^{\star} - 2 \sqrt{\tfrac{\log (60 s n)}{S}}
   \tfrac{\| \tmd^{\star} \|}{\lambda_{\min} ( \M )
   \sqrt{k}}, \tau^{\star} - 20 \sqrt{1 + L n^{- 1 / 2}} \tfrac{\|
   \tmd^{\star} \|}{\lambda_{\min} ( \M )} \sqrt{1 -
   \tfrac{k}{n}} \} \]
Taking inverse on both sides, we have
\[ \kappa_k \leq \min \Big\{ \tfrac{\kappa^{\star}}{1 - G_1 \kappa^{\star}
   \sqrt{\frac{\log n}{k}}}, \tfrac{\kappa^{\star}}{1 - G_2 \kappa^{\star}
   \sqrt{1 - \frac{k}{n}}}, \kappa ( \M ) \Big\} \]
and this completes the proof.

%% file: add_exp.tex
\section{Additional Experiments}

In this section, we provide three new sets of
experiments that demonstrate the cases where our proposed approach can apply.

\subsection{Amortize Time across Multiple Solves}

In this experiment, we simulate the scenario where linear system $\mathbf{Mx}
= \mathbf{b}$ is solved for multiple $\mathbf{b}$ but the same $\mathbf{M}$.
Denote
\begin{itemize}[leftmargin=15pt]
  \item $T_{\tmop{pre}}$
  
  the time for solving for finding an approximate optimal diagonal
  preconditioner $\hat{\mathbf{P}}$
  
  \item $T_{\tmop{IC}}$
  
  the time for solving for finding a no-fill incomplete Cholesky decomposition
  implemented by \texttt{MATLAB ichol}.
  
  \item $T_{\tmop{solve}}$
  
  the average solution time of solving $\mathbf{Mx} = \mathbf{b}$ with Jacobi
  preconditioner
  
  \item $T_{\tmop{Psolve}}$
  
  the average solution time of solving $\mathbf{Mx} = \mathbf{b}$ with
  $\hat{\mathbf{P}}$
  
  \item $T_{\tmop{ICsolve}}$
  
  the average solution time of solving $\mathbf{Mx} = \mathbf{b}$ with
  incomplete Cholesky preconditioner
\end{itemize}
Each of $T$ above is averaged over 100 solves with different right-hand sides.

If $T_{\tmop{Psolve}} < T_{\tmop{solve}}$, theoretically after $N_1 =
\lceil \frac{T_{\tmop{pre}}}{T_{\tmop{solve}} - T_{\tmop{Psolve}}}
\rceil$ solves our approach becomes faster than using Jacobi
preconditioner. Similarly, if $T_{\tmop{Psolve}} < T_{\tmop{ICsolve}}$, our
approach becomes faster than the incomplete Cholesky preconditioner after at most
$N_2 = \lceil \frac{T_{\tmop{pre}} - T_{\tmop{IC}}}{T_{\tmop{ICsolve}} -
T_{\tmop{Psolve}}} \rceil$ solves. Note that if $T_{\tmop{Psolve}} \geq
T_{\tmop{ICsolve}}$, $\hat{\mathbf{P}}$ is inferior to an incomplete Cholesky
preconditioner when there is no memory restriction. We consider a special case
where preconditioner subspace is taken to be $\tmop{span} \{ \mathbf{I},
\tmop{diag} (\mathbf{M}) \}$. \ PCG is allowed to run 1000 iterations. If PCG
fails to converge to $10^{- 5}$ tolerance after 1000 iterations, we record the
time as $+ \infty$ and let $N = 1$.

\paragraph{Random Matrices}Our first experiment runs on random matrices as in
the \tmtextbf{Section 6}, where we take $n = 10^3$ and $\sigma = 1.0$. PCG with an incomplete Cholesky preconditioner does not
converge in 1000 iterations for the tested random matrices. \tmtextbf{Figure 1} displays the distribution of
$N_1$.\\

\begin{figure}[h]
\centering
{\includegraphics[scale=0.3]{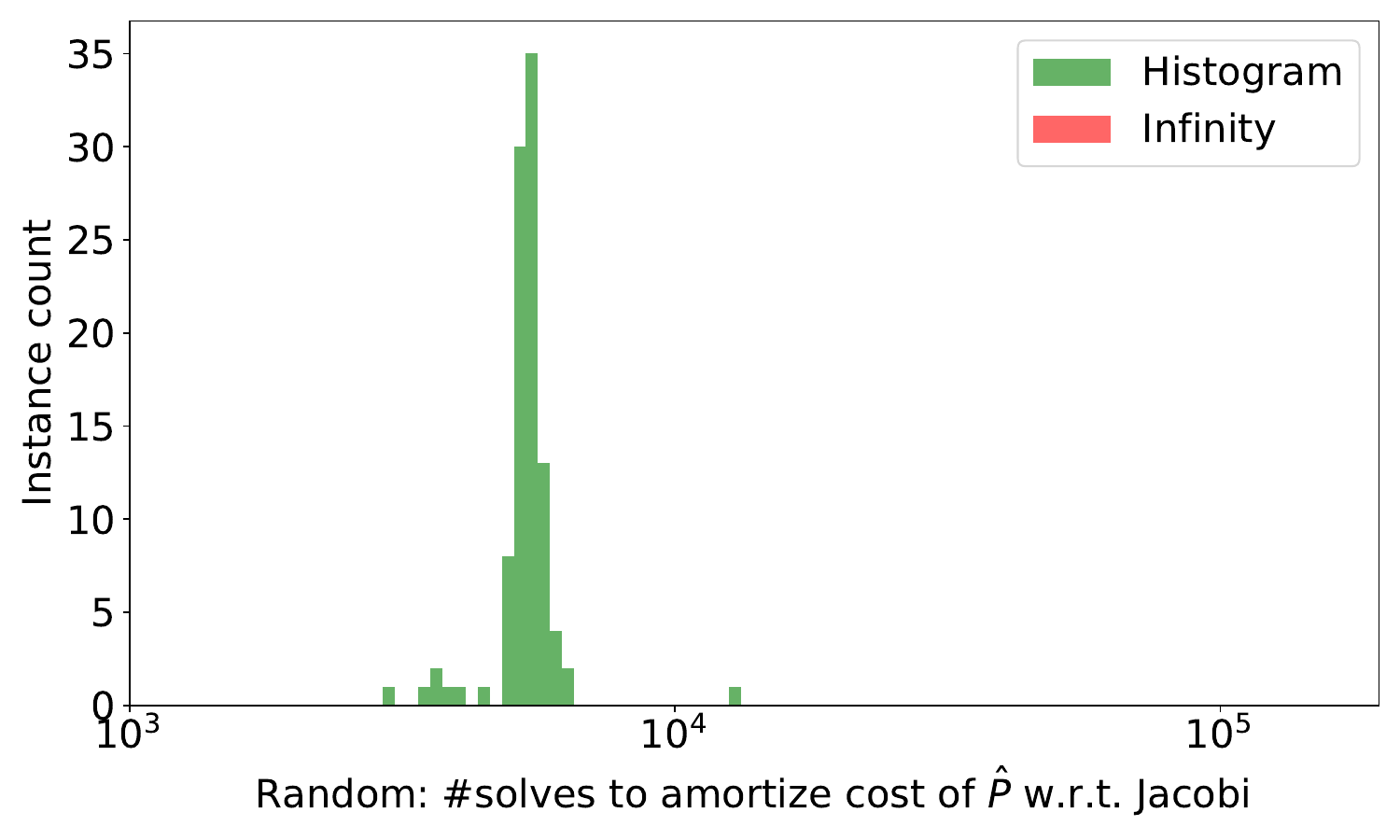}}
  \caption{Frequency plot of \#solves to amortize the cost of finding
  $\hat{\mathbf{P}}$ \label{fig:add-1}}
\end{figure}

According to the figure, we see that in most cases, after round 5000 solves, the
solution time improvement of finding $\hat{\mathbf{P}}$ can be amortized
(with respect to Jacobi preconditioner).

\paragraph{SuiteSparse} We also test on 300 matrices from
\texttt{SuiteSparse} and plot the distribution of $N_1$ and $N_2$ over
these instances.

\subsection{More matrices from IPM normal equation}

This section gives more normal matrix instances from
\texttt{NETLIB/MIPLIB}. We transform each problem into standard form,
apply the standard primal-dual interior point method till $10^{- 8}$ accuracy.
Then we take the normal matrix at iteration $\min \{ 10, T \}$, where $T$ is
the IPM iteration count

\begin{figure}[h]
\centering
\includegraphics[scale=0.2]{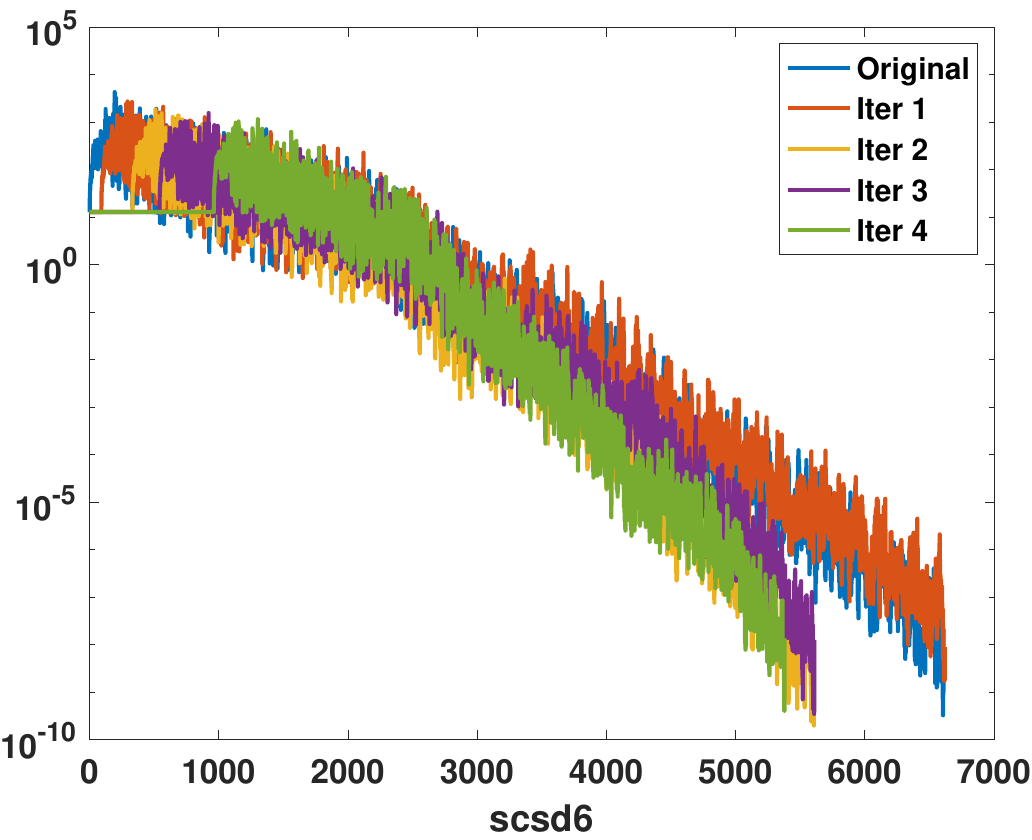}
{\includegraphics[scale=0.2]{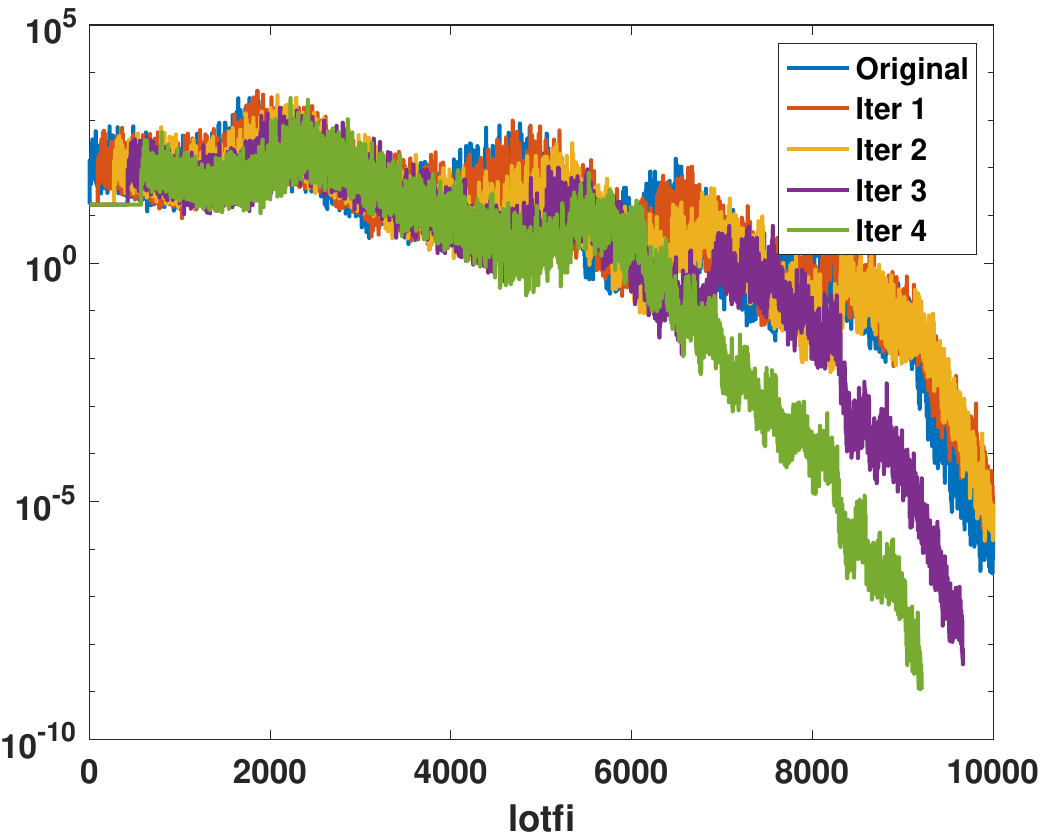}}{\includegraphics[scale=0.2]{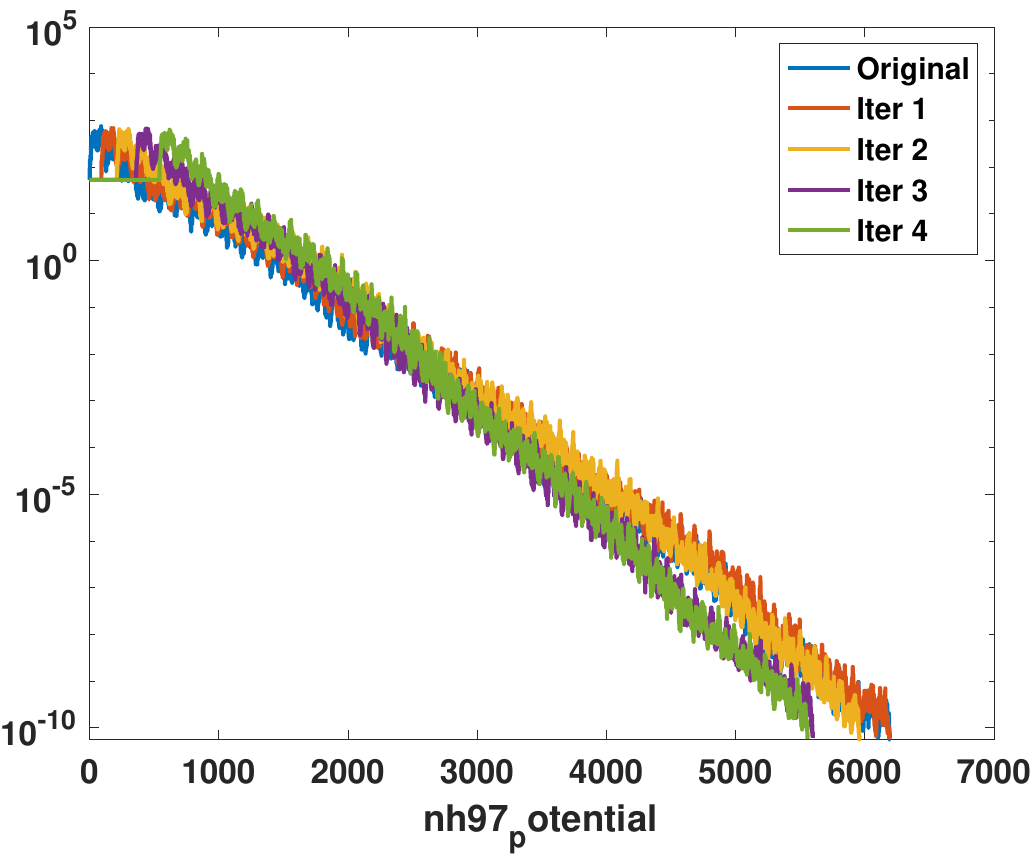}}{\includegraphics[scale=0.2]{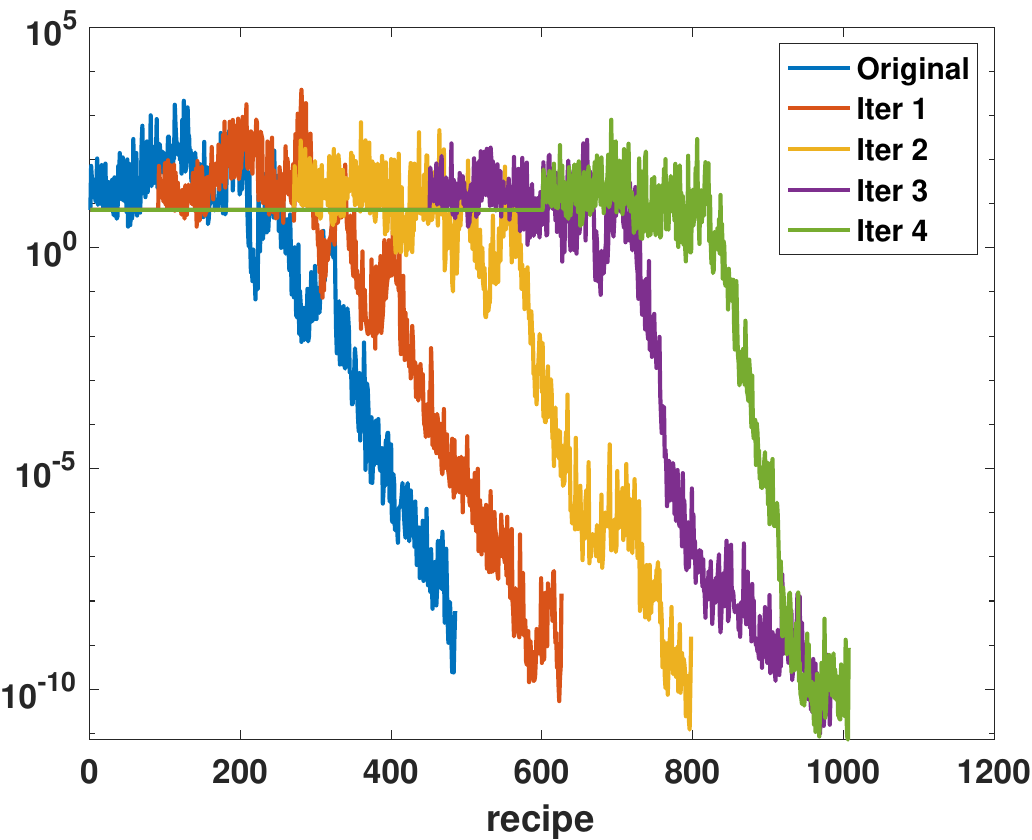}}
  
  {\includegraphics[scale=0.2]{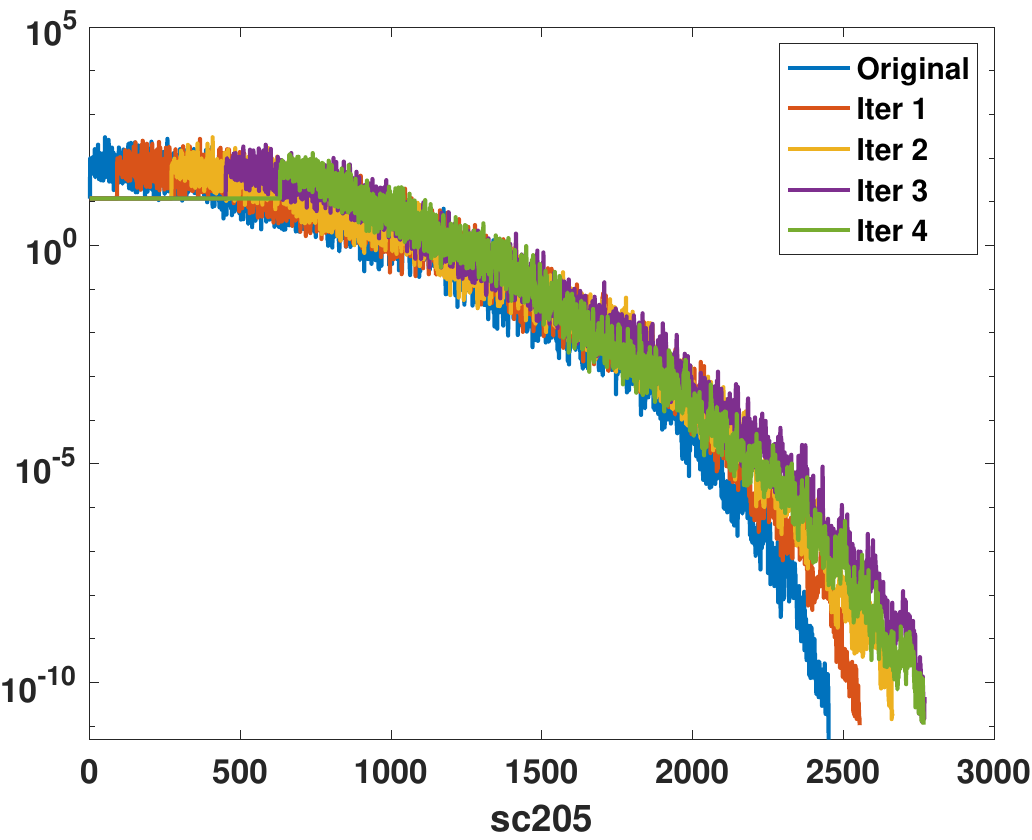}}{\includegraphics[scale=0.2]{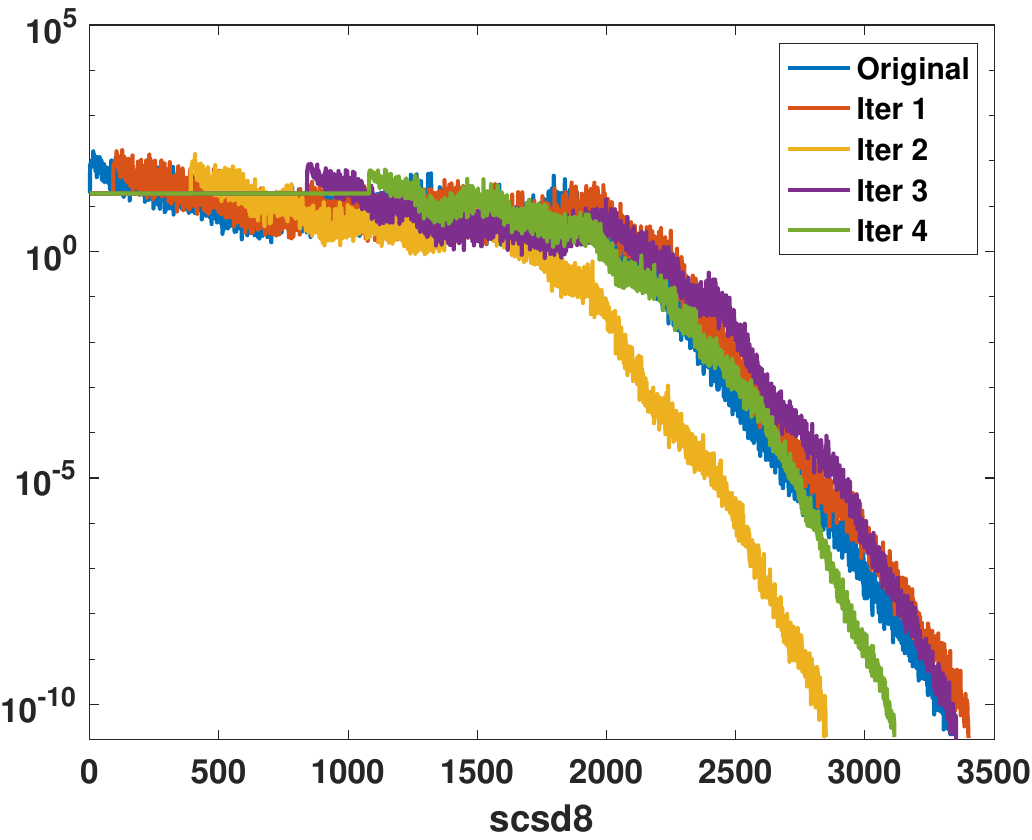}}{\includegraphics[scale=0.2]{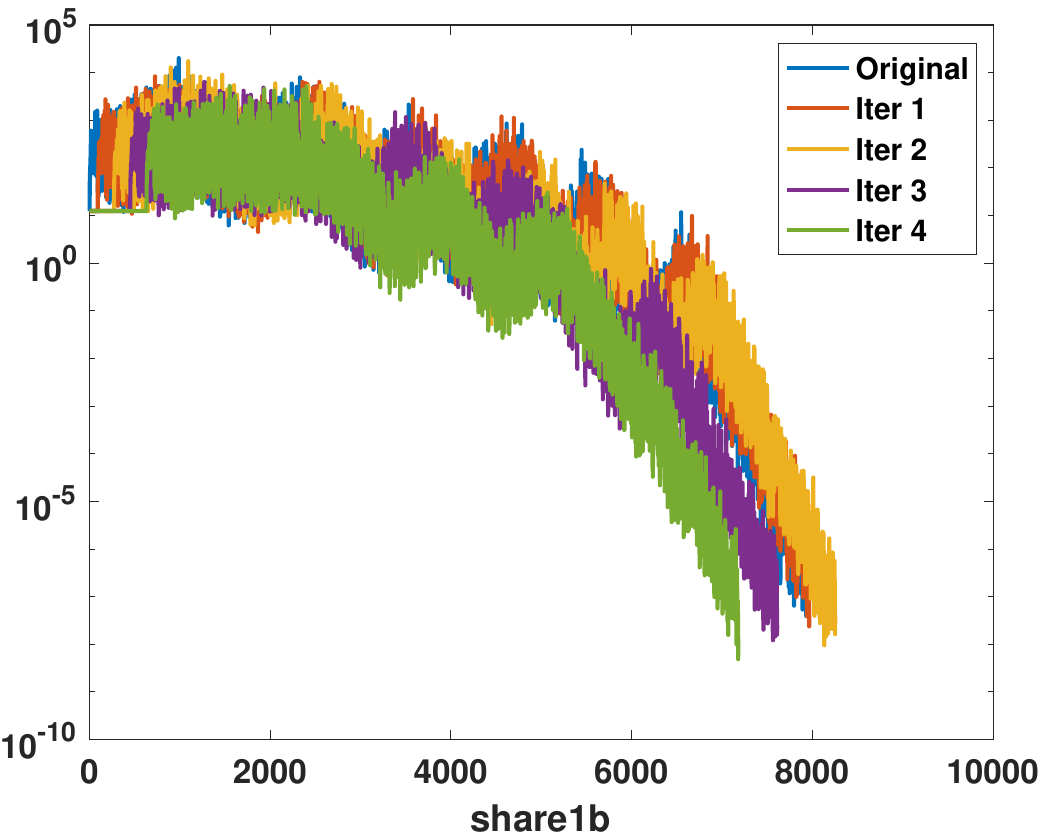}}{\includegraphics[scale=0.2]{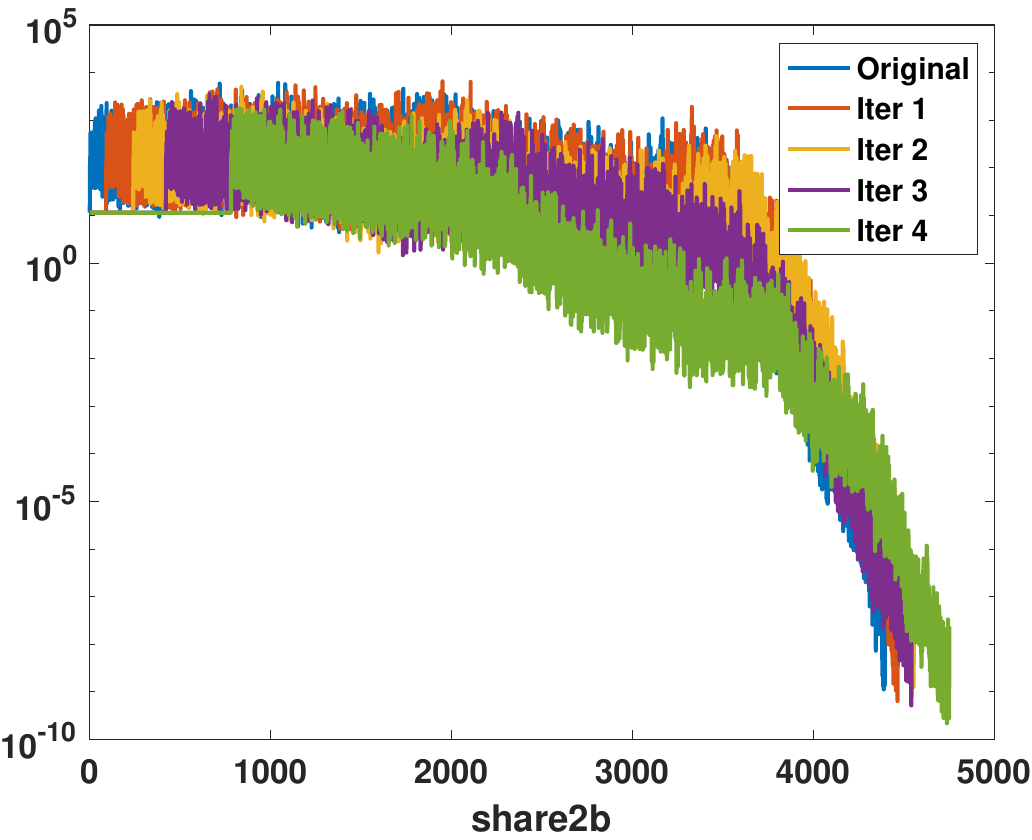}}
  \caption{Convergence behavior of PCG on normal matrices}
\end{figure}

\subsection{Diagonal preconditioning for iterative gradient-based algorithms}

In this section, we conduct additional experiments that show how diagonal
preconditioner can be applied in an iterative gradient-based method to
accelerate convergence.

\paragraph{Experiment setup}

We consider regularized logistic regression problem from \texttt{LIBSVM}
dataset
\begin{eqnarray*}
  \min_{x \in \mathbb{R}^n} & f (\mathbf{x}) \assign \frac{1}{m} \sum_{i =
  1}^m y_i \log \big( \tfrac{1}{1 + \exp (\mathbf{a}_i^{\top} \mathbf{x})}
  \big) + (1 - y_i ) \log \big( \tfrac{\exp (\mathbf{a}_i^{\top}
  \mathbf{x})}{1 + \exp (\mathbf{a}_i^{\top} \mathbf{x})} \big) +
  \frac{\lambda}{2} \| \mathbf{x} \|^2 & 
\end{eqnarray*}
with $\lambda = \frac{1}{10}$. We run preconditioned gradient descent
\[ \mathbf{x}^{k + 1} = \mathbf{x}^k - \mathbf{P}^{- 1}_k \nabla f
   (\mathbf{x}^k) \]
with the following choices of $\mathbf{P}_k$
\begin{itemize}[leftmargin=10pt]
  \item $\mathbf{P}_k \equiv \alpha \mathbf{I}$
  
  gradient descent with constant stepsize
  
  \item $\mathbf{P}_k \equiv \alpha \hat{\mathbf{P}} (\nabla^2 f
  (\mathbf{x}^1))$
  
  gradient descent preconditioned by approximate preconditioner in
  $\tmop{span} \{ \tmop{diag} (\nabla^2 f (\mathbf{x}^k)), \mathbf{I} \}$
  
  \item $\mathbf{P}_k = \alpha \tmop{diag} (\nabla^2 f (\mathbf{x}^k))$
  
  gradient descent preconditioned by the diagonal of Hessian.
\end{itemize}
Note that other preconditioners cannot be computed in this context since
$\textbf{} \nabla^2 f (\mathbf{x}^k) = \mathbf{A}^{\top} \mathbf{B} \mathbf{A}
+ \lambda \mathbf{I}$ is hard to compute and store. The hessian diagonal
preconditioner is updated every 500 iterations. Gradient descent runs until
$\| \nabla f (\mathbf{x}^k) \| \leq 10^{- 4}$ or the maximum iteration (50000)
is reached. Time of computing $\hat{\mathbf{P}} (\nabla^2 f
(\mathbf{x}^1))$ is reflected in the plot.

\begin{figure}[h]
\centering
  {\includegraphics[scale=0.25]{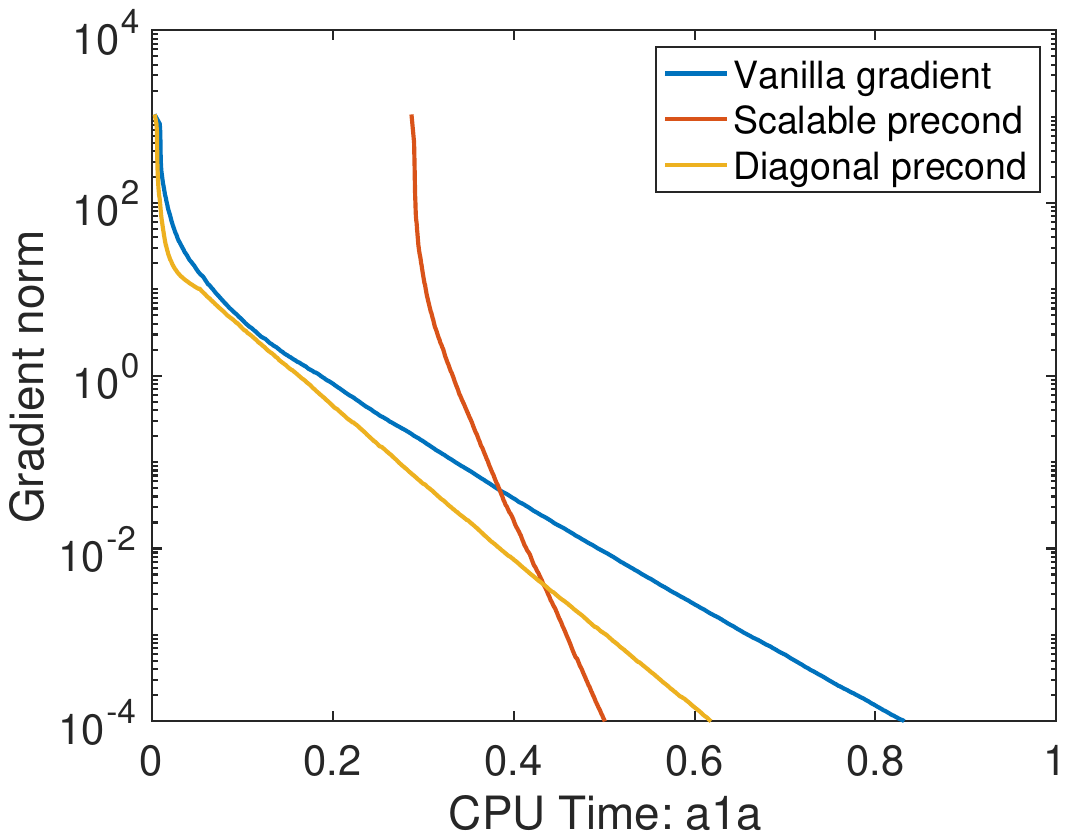}}{\includegraphics[scale=0.25]{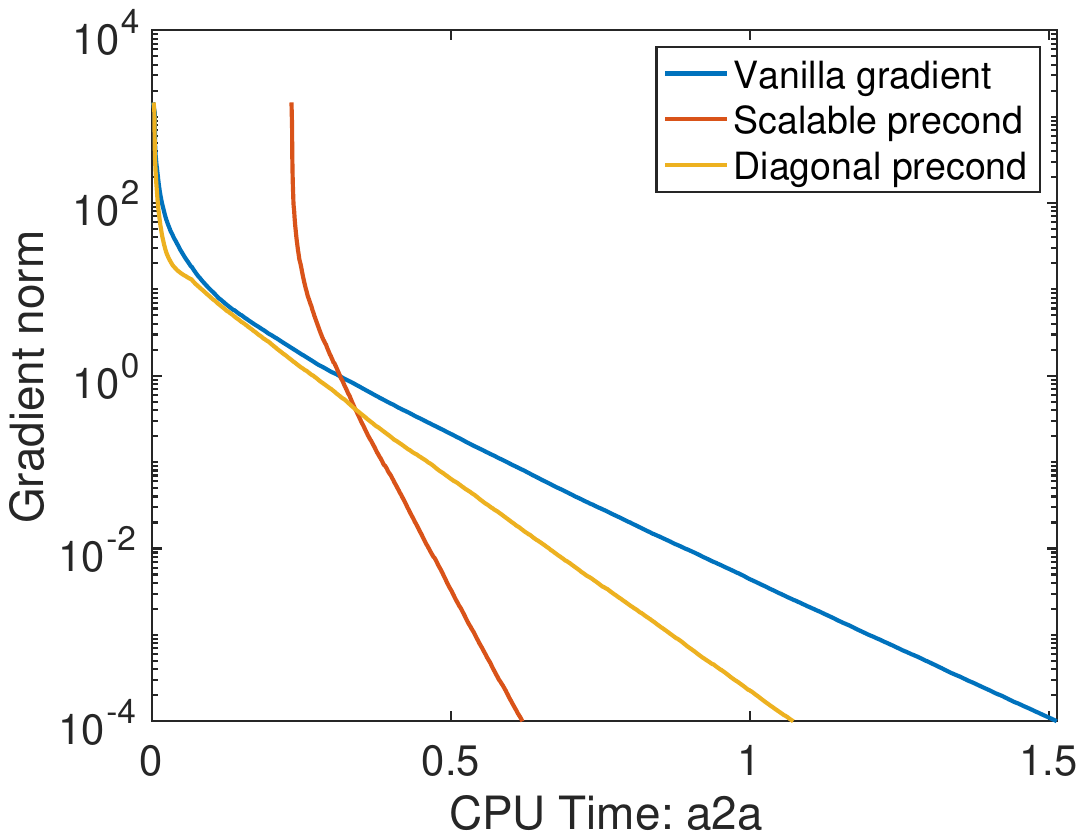}}{\includegraphics[scale=0.25]{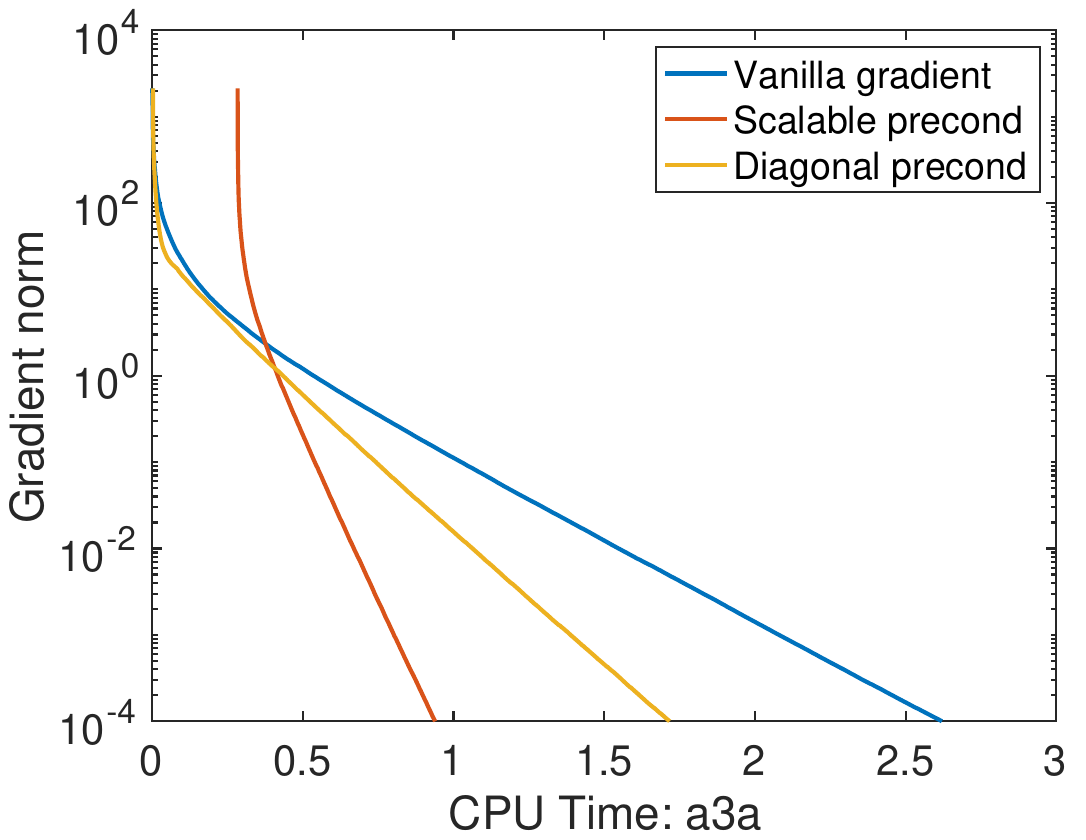}}
  
  {\includegraphics[scale=0.25]{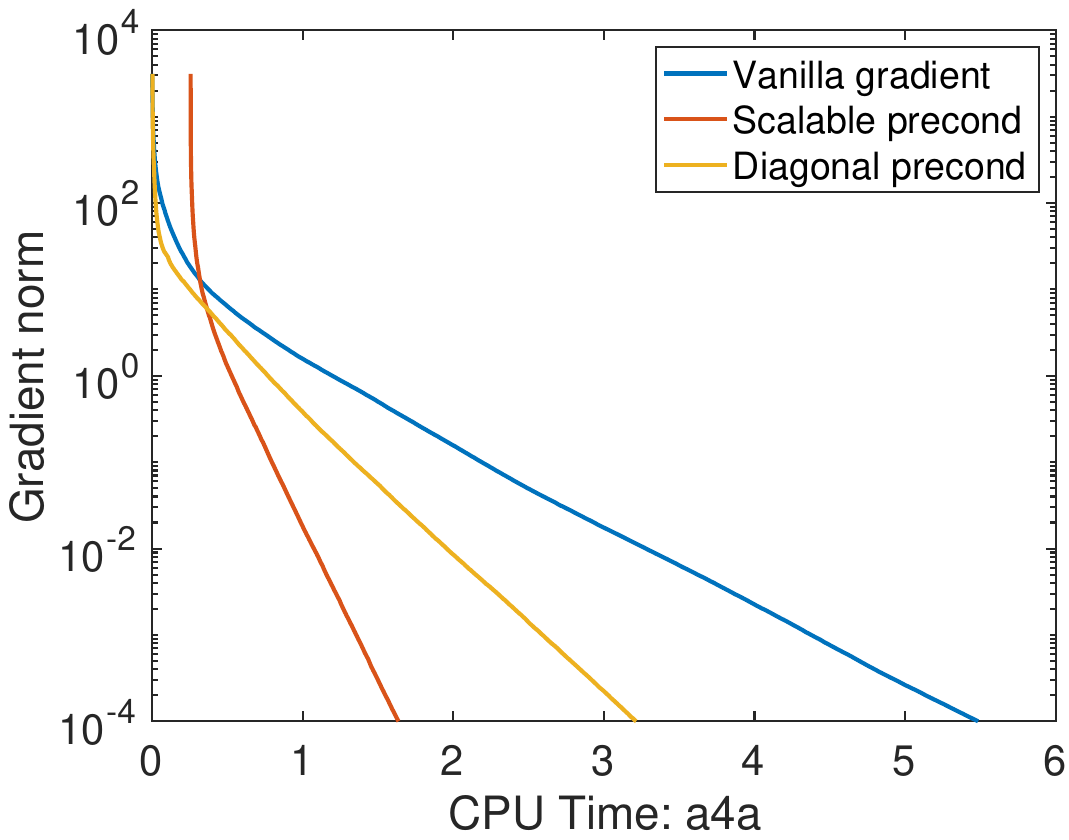}}{\includegraphics[scale=0.25]{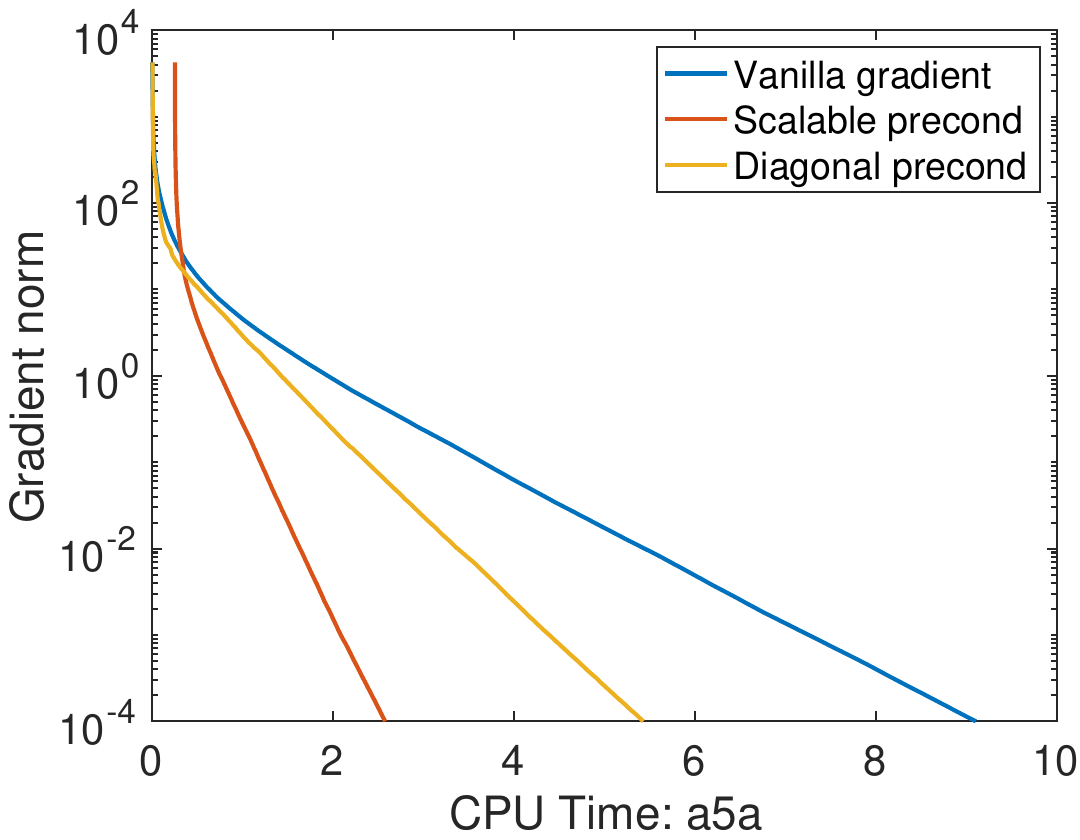}}{\includegraphics[scale=0.25]{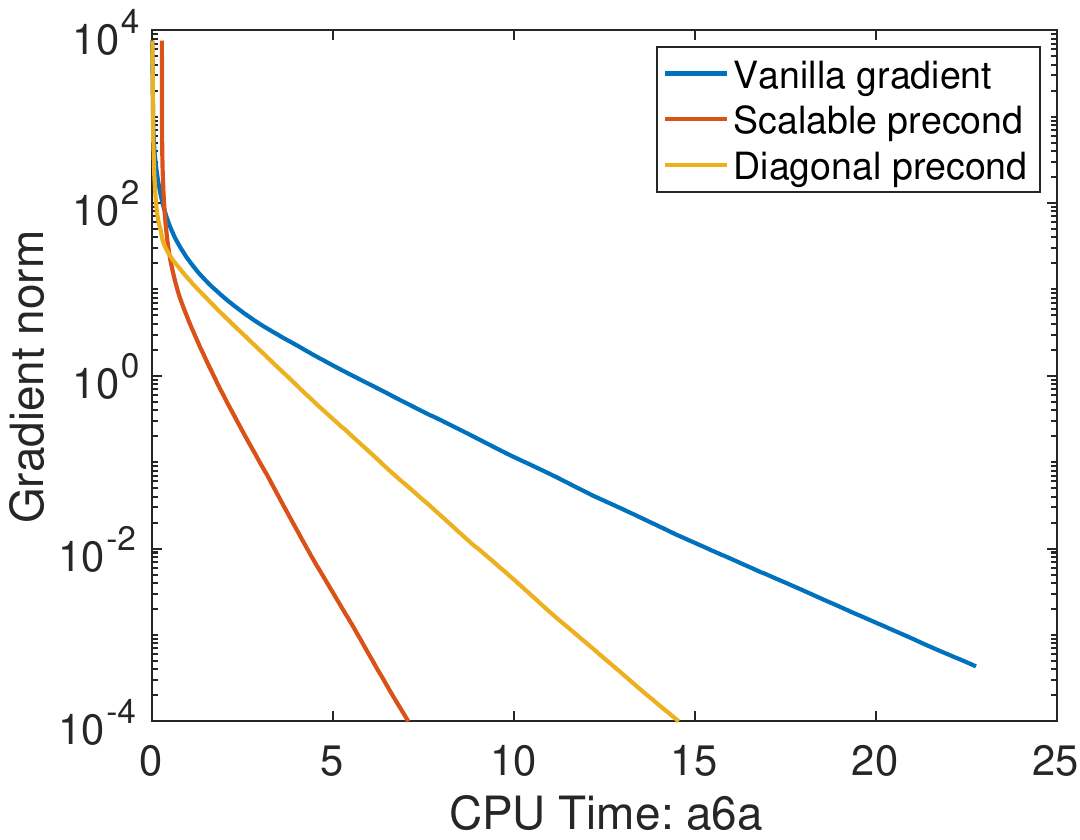}}
  
  {\includegraphics[scale=0.25]{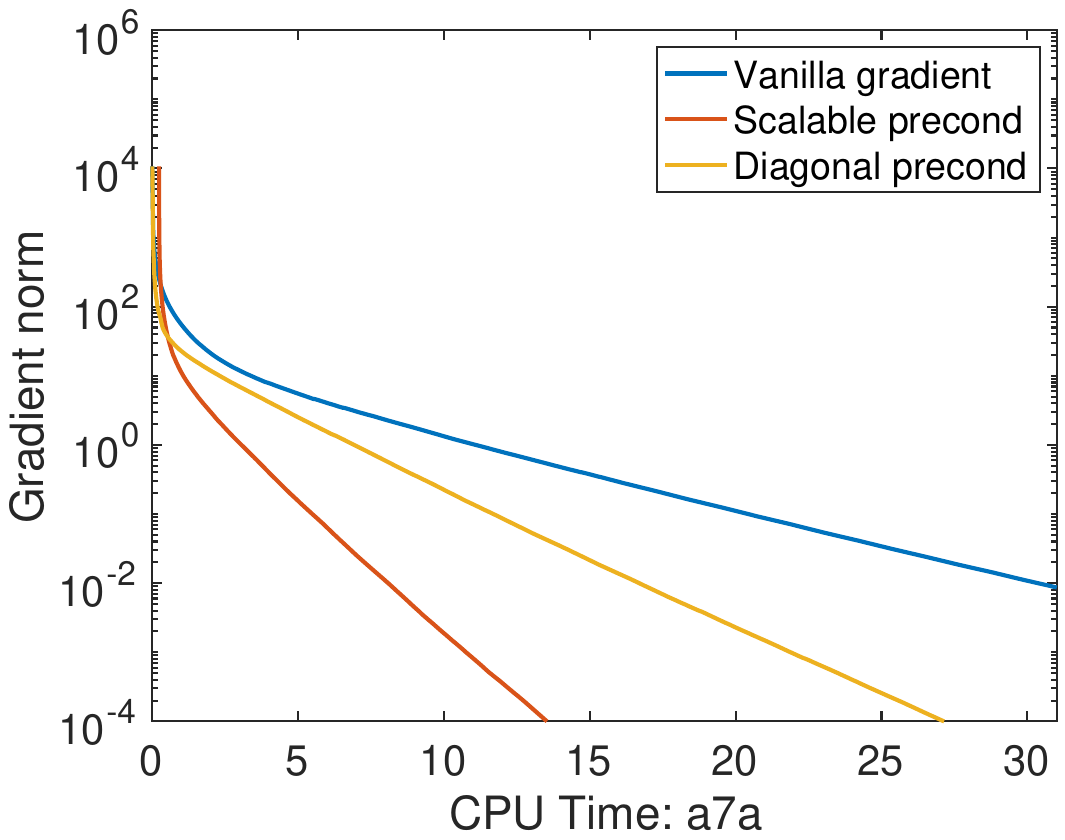}}{\includegraphics[scale=0.25]{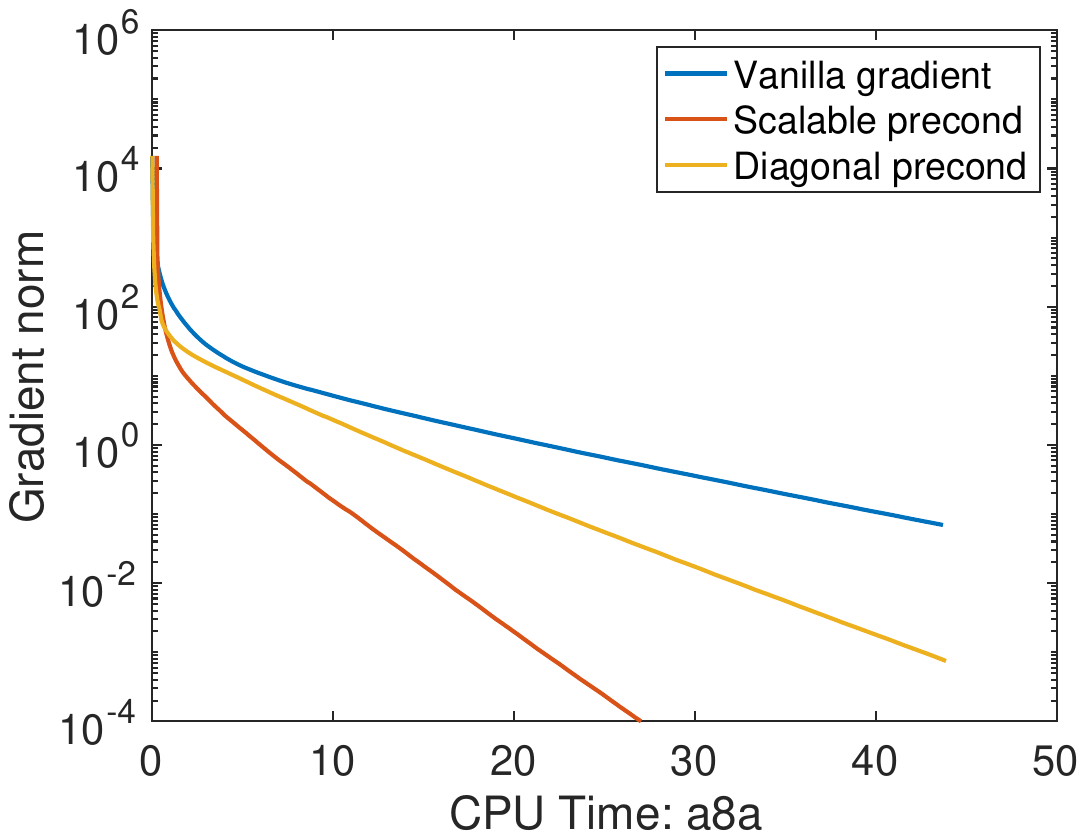}}{\includegraphics[scale=0.25]{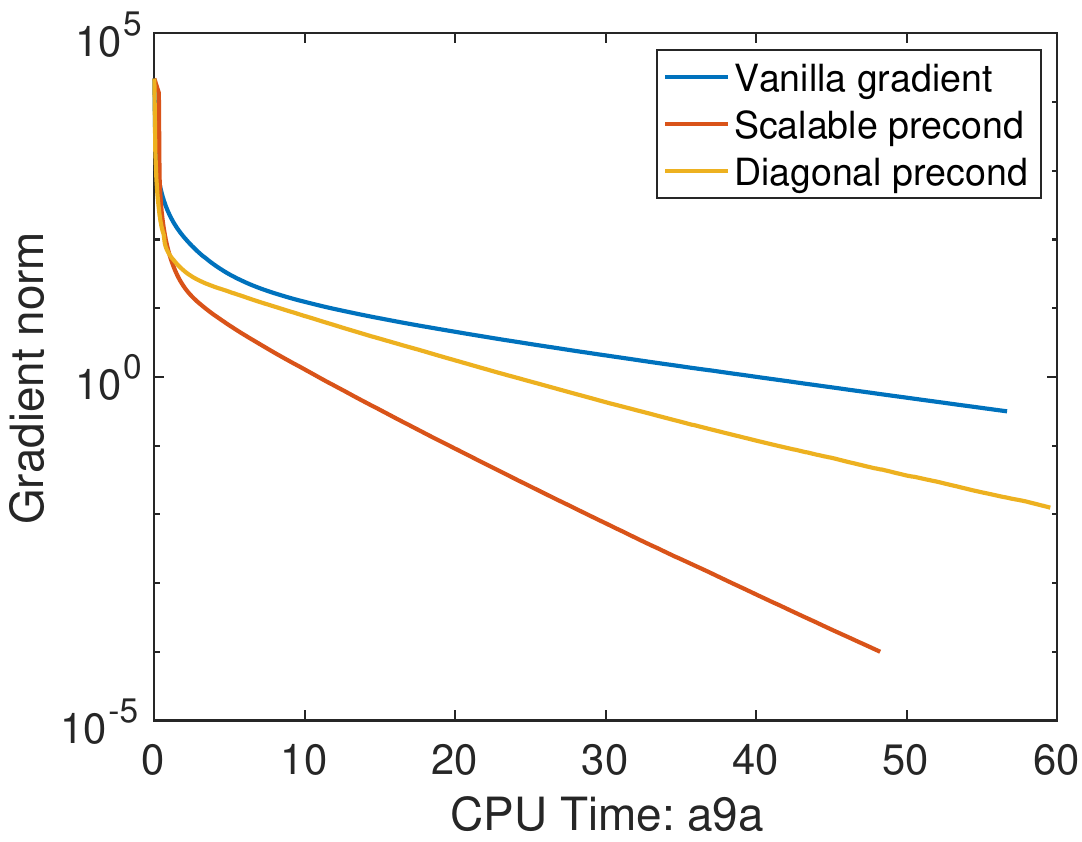}}
  \caption{Convergence behavior of preconditioned gradient descent on logistic
  regression}
\end{figure}

From the figure, we can see that gradient-based methods preconditioned by our
$\hat{\mathbf{P}}$ outperforms traditional preconditioning techniques.
Moreover, the time for finding $\hat{\mathbf{P}}$ is negligible compared
to time for subsequent gradient updates. Interestingly, for these logistic
regression problems, a reasonably good preconditioner is approximately
\[ \hat{\mathbf{P}} = 0.3 \times \mathbf{I} + 0.7 \times \tmop{diag}
   (\nabla^2 f (\mathbf{x}^k)) . \]

%% file: draft.bbl
\begin{thebibliography}{10}

\bibitem{agrawal2020disciplined}
Akshay Agrawal and Stephen Boyd.
\newblock Disciplined quasiconvex programming.
\newblock {\em Optimization Letters}, 14:1643--1657, 2020.

\bibitem{benzi1999comparative}
Michele Benzi and Miroslav Tuma.
\newblock A comparative study of sparse approximate inverse preconditioners.
\newblock {\em Applied Numerical Mathematics}, 30(2-3):305--340, 1999.

\bibitem{benzi2003robust}
Michele Benzi and Miroslav Tuma.
\newblock A robust incomplete factorization preconditioner for positive
  definite matrices.
\newblock {\em Numerical Linear Algebra with Applications}, 10(5-6):385--400,
  2003.

\bibitem{braatz1994minimizing}
Richard~D Braatz and Manfred Morari.
\newblock Minimizing the {E}uclidean condition number.
\newblock {\em SIAM Journal on Control and Optimization}, 32(6):1763--1768,
  1994.

\bibitem{bradley2010algorithms}
Andrew~M Bradley.
\newblock Algorithms for the equilibration of matrices and their application to
  limited-memory {Q}uasi-{N}ewton methods.
\newblock Technical report, STANFORD UNIV CA, 2010.

\bibitem{burer2003nonlinear}
Samuel Burer and Renato~DC Monteiro.
\newblock A nonlinear programming algorithm for solving semidefinite programs
  via low-rank factorization.
\newblock {\em Mathematical programming}, 95(2):329--357, 2003.

\bibitem{chen2006modified}
X~Chen, KC~Toh, and KK~Phoon.
\newblock A modified {SSOR} preconditioner for sparse symmetric indefinite
  linear systems of equations.
\newblock {\em International Journal for Numerical Methods in Engineering},
  65(6):785--807, 2006.

\bibitem{chen2011minimizing}
Xiaojun Chen, Robert~S Womersley, and Jane~J Ye.
\newblock Minimizing the condition number of a {G}ram matrix.
\newblock {\em SIAM Journal on optimization}, 21(1):127--148, 2011.

\bibitem{cline1979estimate}
Alan~K Cline, Cleve~B Moler, George~W Stewart, and James~H Wilkinson.
\newblock An estimate for the condition number of a matrix.
\newblock {\em SIAM Journal on Numerical Analysis}, 16(2):368--375, 1979.

\bibitem{deng2022new}
Qi~Deng, Qing Feng, Wenzhi Gao, Dongdong Ge, Bo~Jiang, Yuntian Jiang, Jingsong
  Liu, Tianhao Liu, Chenyu Xue, Yinyu Ye, et~al.
\newblock New developments of admm-based interior point methods for linear
  programming and conic programming.
\newblock {\em arXiv preprint arXiv:}, 2209.01793, 2022.

\bibitem{derezinski2020precise}
Michal Derezinski, Feynman~T Liang, Zhenyu Liao, and Michael~W Mahoney.
\newblock Precise expressions for random projections: Low-rank approximation
  and randomized newton.
\newblock {\em Advances in Neural Information Processing Systems},
  33:18272--18283, 2020.

\bibitem{ding2023revisiting}
Lijun Ding and Benjamin Grimmer.
\newblock Revisiting spectral bundle methods: Primal-dual (sub) linear
  convergence rates.
\newblock {\em SIAM Journal on Optimization}, 33(2):1305--1332, 2023.

\bibitem{ding2021simplicity}
Lijun Ding and Madeleine Udell.
\newblock On the simplicity and conditioning of low rank semidefinite programs.
\newblock {\em SIAM Journal on Optimization}, 31(4):2614--2637, 2021.

\bibitem{ding2021optimal}
Lijun Ding, Alp Yurtsever, Volkan Cevher, Joel~A Tropp, and Madeleine Udell.
\newblock An optimal-storage approach to semidefinite programming using
  approximate complementarity.
\newblock {\em SIAM Journal on Optimization}, 31(4):2695--2725, 2021.

\bibitem{ferng1991adaptive}
William~R Ferng, Gene~H Golub, and Robert~J Plemmons.
\newblock Adaptive lanczos methods for recursive condition estimation.
\newblock {\em Numerical Algorithms}, 1(1):1--19, 1991.

\bibitem{frangella2023randomized}
Zachary Frangella, Joel~A Tropp, and Madeleine Udell.
\newblock Randomized nystr{\"o}m preconditioning.
\newblock {\em SIAM Journal on Matrix Analysis and Applications},
  44(2):718--752, 2023.

\bibitem{gao2022hdsdp}
Wenzhi Gao, Dongdong Ge, and Yinyu Ye.
\newblock {HDSDP}: Software for semidefinite programming.
\newblock {\em arXiv preprint arXiv:2207.13862}, 2022.

\bibitem{ge2022cardinal}
Dongdong Ge, Qi~Huangfu, Zizhuo Wang, Jian Wu, and Yinyu Ye.
\newblock Cardinal {O}ptimizer ({COPT}) user guide.
\newblock {\em arXiv preprint arXiv:2208.14314}, 2022.

\bibitem{giselsson2014diagonal}
Pontus Giselsson and Stephen Boyd.
\newblock Diagonal scaling in {D}ouglas-{R}achford splitting and admm.
\newblock In {\em 53rd IEEE Conference on Decision and Control}, pages
  5033--5039. IEEE, 2014.

\bibitem{greenbaum1989optimal}
Anne Greenbaum and GH~Rodrigue.
\newblock Optimal preconditioners of a given sparsity pattern.
\newblock {\em BIT Numerical Mathematics}, 29(4):610--634, 1989.

\bibitem{gurobi2021gurobi}
LLC Gurobi~Optimization.
\newblock Gurobi optimizer reference manual, 2021.

\bibitem{helmberg2000spectral}
Christoph Helmberg and Franz Rendl.
\newblock A spectral bundle method for semidefinite programming.
\newblock {\em SIAM Journal on Optimization}, 10(3):673--696, 2000.

\bibitem{jacobi1845ueber}
Carl~GJ Jacobi.
\newblock Ueber eine neue aufl{\"o}sungsart der bei der methode der kleinsten
  quadrate vorkommenden line{\"a}ren gleichungen.
\newblock {\em Astronomische Nachrichten}, 22(20):297--306, 1845.

\bibitem{jambulapati2023structured}
Arun Jambulapati, Jerry Li, Christopher Musco, Kirankumar Shiragur, Aaron
  Sidford, and Kevin Tian.
\newblock Structured semidefinite programming for recovering structured
  preconditioners.
\newblock {\em arXiv preprint arXiv:2310.18265}, 2023.

\bibitem{jambulapati2020fast}
Arun Jambulapati, Jerry Li, Christopher Musco, Aaron Sidford, and Kevi Tian.
\newblock Fast and near-optimal diagonal preconditioning.
\newblock {\em arXiv preprint arXiv:}, 2008.01722, 2020.

\bibitem{knight2013fast}
Philip~A Knight and Daniel Ruiz.
\newblock A fast algorithm for matrix balancing.
\newblock {\em IMA Journal of Numerical Analysis}, 33(3):1029--1047, 2013.

\bibitem{knight2014symmetry}
Philip~A Knight, Daniel Ruiz, and Bora U{\c{c}}ar.
\newblock A symmetry preserving algorithm for matrix scaling.
\newblock {\em SIAM journal on Matrix Analysis and Applications},
  35(3):931--955, 2014.

\bibitem{kolodziej2019suitesparse}
Scott~P Kolodziej, Mohsen Aznaveh, Matthew Bullock, Jarrett David, Timothy~A
  Davis, Matthew Henderson, Yifan Hu, and Read Sandstrom.
\newblock The suitesparse matrix collection website interface.
\newblock {\em Journal of Open Source Software}, 4(35):1244, 2019.

\bibitem{krishnan2003properties}
Kartik Krishnan and John~E Mitchell.
\newblock Properties of a cutting plane method for semidefinite programming.
\newblock {\em submitted for publication}, 2003.

\bibitem{kunstner2023searching}
Frederik Kunstner, Victor~S. Portella, Mark Schmidt, and Nick Harvey.
\newblock Searching for optimal per-coordinate step-sizes with multidimensional
  backtracking.
\newblock In {\em Thirty-seventh Conference on Neural Information Processing
  Systems}, 2023.

\bibitem{lemon2016low}
Alex Lemon, Anthony Man-Cho So, Yinyu Ye, et~al.
\newblock Low-rank semidefinite programming: Theory and applications.
\newblock {\em Foundations and Trends{\textregistered} in Optimization},
  2(1-2):1--156, 2016.

\bibitem{liberti2021random}
Leo Liberti, Pierre-Louis Poirion, and Ky~Vu.
\newblock Random projections for conic programs.
\newblock {\em Linear Algebra and its Applications}, 626:204--220, 2021.

\bibitem{lu2011minimizing}
Zhaosong Lu and Ting~Kei Pong.
\newblock Minimizing condition number via convex programming.
\newblock {\em SIAM Journal on Matrix Analysis and Applications},
  32(4):1193--1211, 2011.

\bibitem{lubbecke2005selected}
Marco~E L{\"u}bbecke and Jacques Desrosiers.
\newblock Selected topics in column generation.
\newblock {\em Operations research}, 53(6):1007--1023, 2005.

\bibitem{marechal2009optimizing}
Pierre Mar{\'e}chal and Jane~J Ye.
\newblock Optimizing condition numbers.
\newblock {\em SIAM Journal on Optimization}, 20(2):935--947, 2009.

\bibitem{menard2002applied}
Scott Menard.
\newblock {\em Applied logistic regression analysis}.
\newblock Number 106. Sage, 2002.

\bibitem{o2016conic}
Brendan O'donoghue, Eric Chu, Neal Parikh, and Stephen Boyd.
\newblock Conic optimization via operator splitting and homogeneous self-dual
  embedding.
\newblock {\em Journal of Optimization Theory and Applications},
  169:1042--1068, 2016.

\bibitem{oskoorouchi2011interior}
Mohammad~R Oskoorouchi, Hamid~R Ghaffari, Tam{\'a}s Terlaky, and Dionne~M
  Aleman.
\newblock An interior point constraint generation algorithm for semi-infinite
  optimization with health-care application.
\newblock {\em Operations research}, 59(5):1184--1197, 2011.

\bibitem{oskoorouchi2007matrix}
Mohammad~R Oskoorouchi and Jean-Louis Goffin.
\newblock A matrix generation approach for eigenvalue optimization.
\newblock {\em Mathematical programming}, 109:155--179, 2007.

\bibitem{overton1992large}
Michael~L Overton.
\newblock Large-scale optimization of eigenvalues.
\newblock {\em SIAM Journal on Optimization}, 2(1):88--120, 1992.

\bibitem{pataki1998rank}
G{\'a}bor Pataki.
\newblock On the rank of extreme matrices in semidefinite programs and the
  multiplicity of optimal eigenvalues.
\newblock {\em Mathematics of operations research}, 23(2):339--358, 1998.

\bibitem{pock2011diagonal}
Thomas Pock and Antonin Chambolle.
\newblock Diagonal preconditioning for first order primal-dual algorithms in
  convex optimization.
\newblock In {\em 2011 International Conference on Computer Vision}, pages
  1762--1769. IEEE, 2011.

\bibitem{qu2022optimal}
Zhaonan Qu, Wenzhi Gao, Oliver Hinder, Yinyu Ye, and Zhengyuan Zhou.
\newblock Optimal diagonal preconditioning: Theory and practice.
\newblock {\em arXiv preprint arXiv:2209.00809}, 2022.

\bibitem{ruiz2001scaling}
Daniel Ruiz.
\newblock A scaling algorithm to equilibrate both rows and columns norms in
  matrices.
\newblock Technical report, CM-P00040415, 2001.

\bibitem{saad2003iterative}
Yousef Saad.
\newblock {\em Iterative methods for sparse linear systems}, volume~82.
\newblock SIAM, 2003.

\bibitem{shapiro1982optimally}
Alexander Shapiro.
\newblock Optimally scaled matrices, necessary and sufficient conditions.
\newblock {\em Numerische Mathematik}, 39:239--245, 1982.

\bibitem{sivaramakrishnan2002linear}
Kartik~Krishnan Sivaramakrishnan.
\newblock {\em Linear programming approaches to semidefinite programming
  problems}.
\newblock Rensselaer Polytechnic Institute, 2002.

\bibitem{takapoui2016preconditioning}
Reza Takapoui and Hamid Javadi.
\newblock Preconditioning via diagonal scaling.
\newblock {\em arXiv preprint arXiv:1610.03871}, 2016.

\bibitem{tu2014practical}
Stephen Tu and Jingyan Wang.
\newblock Practical first order methods for large scale semidefinite
  programming.
\newblock Technical report, Technical report, Technical report, University of
  California, Berkeley, 2014.

\bibitem{urschel2021uniform}
John~C Urschel.
\newblock Uniform error estimates for the {L}anczos method.
\newblock {\em SIAM Journal on Matrix Analysis and Applications},
  42(3):1423--1450, 2021.

\bibitem{van1969condition}
Abraham Van~der Sluis.
\newblock Condition numbers and equilibration of matrices.
\newblock {\em Numerische Mathematik}, 14(1):14--23, 1969.

\bibitem{watson1991algorithm}
GA~Watson.
\newblock An algorithm for optimal $\ell$ 2 scaling of matrices.
\newblock {\em IMA journal of numerical analysis}, 11(4):481--492, 1991.

\bibitem{wolkowicz2012handbook}
Henry Wolkowicz, Romesh Saigal, and Lieven Vandenberghe.
\newblock {\em Handbook of semidefinite programming: theory, algorithms, and
  applications}, volume~27.
\newblock Springer Science \& Business Media, 2012.

\bibitem{ye2011interior}
Yinyu Ye.
\newblock {\em Interior point algorithms: theory and analysis}.
\newblock John Wiley \& Sons, 2011.

\bibitem{yurtsever2021scalable}
Alp Yurtsever, Joel~A Tropp, Olivier Fercoq, Madeleine Udell, and Volkan
  Cevher.
\newblock Scalable semidefinite programming.
\newblock {\em SIAM Journal on Mathematics of Data Science}, 3(1):171--200,
  2021.

\end{thebibliography}
